\documentclass[11pt]{amsart}

\usepackage{amssymb,amscd,fullpage}

\usepackage[all]{xy}
\CompileMatrices

\title{Microlocal condition for non-displaceability\footnote{Partially supported by an NSF grant}}

\author{Dmitry Tamarkin}

\address{}

\begin{document}
\maketitle
\newtheorem{theorem}{Theorem}[section]
\newtheorem{Axiom}[theorem]{Axiom}
\newtheorem{Claim}[theorem]{Claim}
\newtheorem{Conjecture}[theorem]{Conjecture}
\newtheorem{Lemma}[theorem]{Lemma}
\newtheorem{sublemma}[theorem]{Sublemma}
\newtheorem{Corollary}[theorem]{Corollary}
\newtheorem{Proposition}[theorem]{Proposition}
\newtheorem{Theorem}[theorem]{Theorem}
\newtheorem{Definition}[theorem]{Definition}
\newtheorem{Definition-Proposition}[theorem]
{Definition-Proposition}
\newtheorem{Condition}[theorem]{Condition}
\def\tf{(f^t)}
\def\Ker{{\text{Ker}}}
\def\Int{{\text{Int}}}
\def\bfZ{{\mathbb{Z}}}
\def\Real{{\text{Re}}}
\def\bn{{\mathbf{n}}}
\def\pr{{\mathbf{pr}}}
\def\bT{{\mathbf{T}}}
\def\bZ{{\mathbf{Z}}}
\def\N{{\mathbf{N}}}
\def\cO{{\mathcal{O}}}
\def\cF{{\mathbf{f}}}
\def\conv{ \bullet}
\def\pmin{\pm}
\def\cC{{\mathcal{C}}}
\def\cD{{\mathcal{D}}}
\def\cI{{\mathcal{I}}}
\def\cV{{\mathcal{V}}}
\def\cT{{\mathbb{T}}}
\def\cX{{\mathcal{X}}}
\def\cY{{\mathcal{Y}}}
\def\bL{\mathbb{L}}
\def\Re{{\mathbb{Re}}}
\def\Co{\mathbb{C}}
\def\gf{{\mathbb{K}}}
\def\Sh{\text{Sh}}
\def\mS{\text{SS}}
\def\ms{{\mS}}
\def\Zentrum{\mathbf{Z}}
\def\sp{\mathfrak{S}}
\def\bfS{{\mathcal{S}}}
\def\cU{{\mathcal{ U}}}
\def\O{{\mathcal{O}}}
\def\g{\mathfrak{g}}
\def\h{\mathfrak{h}}
\def\k{{\mathfrak{k}}}
\def\t{\mathfrak{t}}
\def\SU{\text{SU}}
\def\SO{\text{SO}}
\def\CP{\mathbb{CP}}
\def\RP{\mathbb{RP}}
\def\su{{\mathfrak{su}}}
\def\so{\mathfrak{so}}

\def\diag{\text{diag}}
\def\Inj{{\text{Inj}}}
\def\Subsets{\text{Subsets}}
\def\subsets{\Subsets}
\def\ihom{\underline{\text{Hom}}}
\def\Cone{\text{Cone}}
\def\cone{\Cone}
\def\ve{\varepsilon}
\def\Id{\text{Id}}
\def\dB{{\partial B}}
\def\limdir{{\varinjlim}}
\def\liminv{{\varprojlim}}
\def\Tr{{\text{Tr}}}
\def\op{\text{op}}
\def\Ad{\text{Ad}}
\def\into{\hookrightarrow}
\def\pt{\mathbf{pt}}
\def\bfq{\mathbf{q}}
\def\char{\text{char }}
\def\gl{{\frak{gl}}}
\def\dist{{\text{dist}}}
\def\bN{{\mathbf{N}}}
\def\umin{{U^{-}}}
\def\follows{\Rightarrow}
\def\trian{<}
\def\sph{\Sigma}
\def\orient{\text{or}}
\def\Fl{\mathcal{FL}}
\def\GL{{\text{GL}}}
\def\FL{\Fl}
\def\bft{{\mathbf{t}}}
\def\bfb{\mathbf{b}}
\def\complexes{\mathbf{Complexes}}
\def\valuesp{{\mathbb{S}}}
\def\Con{{\mathbf{Con}}}
\def\vs{\sigma}
\def\image{\text{Image }}
\def\DBSh{{\text{DBSh}}}
\def\strict{{\text{strict}}} 

\centerline{\em To Boris Tsygan on his 50-th birthday}

\begin{abstract} We formulate a sufficient condition for non-displaceability
(by Hamiltonian symplectomorphisms which are  identity outside of a compact)
of a pair of subsets in a cotangent bundle. This condition
is based on micro-local analysis of sheaves on manifolds
by Kashiwara-Schapira.
This condition  is used to prove  that the real projective space and the Clifford torus
inside the complex projective space are mutually non-displaceable
\end{abstract}
\section{Introduction} Let $M$ be a symplectic manifold
and $A,B\subset M$ its compact subsets. $A$ and $B$ are called
non-displaceable if $A\cap X(B)\neq \emptyset$, where $X$
is any Hamiltonian symplectomorphism of $M$ which is identity
outside of a compact.  Given such $A$ and $B$, it is, in general, a 
non-trivial problem to  decide, whether they are displaceable
or not (see, for example, \cite{Pol} and the 
literature therein).  In non-trivial cases (when, say,
$A$ and $B$ can be displaced by a diffeomorphism), all the methods known so far use different versions
of Floer cohomology. 

In this paper we introduce a sufficient condition for non-displaceability in the case when $M=T^*X$ with the standard 
symplectic
structure. Our approach is based on 
Kashiwara-Shapira's microlocal theory of sheaves on manifolds
and is independent of Floer's theory. 
We apply our condition in the following setting. Let
our sympleectic manifold be $\CP^N$ with the standard
symplectic structure and let our subsets be
$\RP^N\subset \CP^N$ and $\mathbb{T}^N\subset \CP^N$,
where $\mathbb{T^N}$ is the Clifford torus consisting
of all points $(z_0:z_1:\cdots:z_n)$ such that
$|z_0|=|z_1|=\cdots=|z_n|$. Let $A$ and $B$ be arbitrarily
chosen from the two subsets specified, we show that such 
$A$ and $B$ are non-displaceable. Same result has been
proven in \cite{Pol} using Hamiltonian Floer theory. Non-displaceability of Clifford
torus has been proven in \cite{Cho}
via computing Floer cohomology.

Observe that our condition applies despite $\CP^N\neq T^*X$.
We use a  certain Lagrangian correspondence between
$T^*\SU(N)$ and $\CP^N\times (\CP^N)^\text{opp}$,
where the symplectic form on $(\CP^N)^\text{opp}$ equals
the opposite to that on $\CP^N$, see Sec. \ref{svedenie}.
This way our original problem gets reduced to 
non-displaceablity of certain subsets in $T^*\SU(N)$.

Let us now get back to the non-displaceability condition
for subsets in  a symplectic manifold $T^*X$, where
$X$ is a smooth manifold.  Fix a ground field $\gf$.
We start with a category
$\cD(X)$ which is defined as a  full subcategory 
of the unbounded derived category of sheaves of $\gf$-vector
spaces on $X\times \Re$,
consisting of all objects $F\in D(X\times\Re)$
satisfying the following condition: for any open $U\subset
X$ and any $c\in \Re\cup \{\infty\}$, 
$R\Gamma_c(U\times (-\infty,c);F)=0$. 
 The category
$\cD(X)$ admits a microlocal definition. Let $\partial_t$
be the vector field on $X\times \Re$ corresponding
to the infinitesimal shifts along $\Re$.
 Let $\Omega_{\leq 0}\subset T^*(X\times \Re)$ be the subset consisting
of all 1-forms $\eta$ satisfying $i_{\partial_t}\eta\leq 0$.
Let $\cC_{\leq 0}\subset D(X\times \Re)$ be the 
 full subcategory
 consisting of all objects microsupported
on $\Omega_{\leq 0}$. One can show that $\cD(X)$ is the left
orthogonal complement to $\cC_{\leq 0}$.  

 One can show
that the embedding  $\cC_{\leq 0}\subset D(X\times \Re)$ 
admits a left adjoint. Therefore, $\cD(X)$ can be 
identified with   a quotient $D(X\times \Re)/\cC_{\leq 0}$.
This motivates us to define microsupports of objects
from $\cD(X)$ as conic closed subsets of
$\Omega_{>0}:=T^*(X\times \Re)\backslash \Omega_{\leq 0}$. 
Thus, we set $\mS_\cD(F):=\mS(F)\cap \Omega_{>0}$ for
any $F\in \cD(X)$.

Let us identify $T^*(X\times \Re)=T^*X\times T^*\Re$.
Let $A\subset T^*X$ be a subset. Define 
$\Cone(A)\subset \Omega_{>0}$ to consist
of all points $(\eta,\alpha)\in T^*X\times T^*\Re$
such that $i_{\partial_t}\alpha>0$ (meaning that
$(\eta,\alpha)\in \Omega_{>0}$) and 
$$
\frac{\eta}{i_{\partial_t}\alpha}\in A.
$$

Let $\cD_A(X)\subset \cD(X)$ be the full subcategory
consisting of all $F\in \cD(X)$ such that $\mS_\cD(F)\subset
\Cone(A)$. This way we can link subsets of $T^*X$ with
the category $\cD(X)$.

 Let $c\in \Re$, let $T_c:X\times \Re\to X\times \Re$ be 
the shift by $c$: $T_c(x,t)=(x,t+c)$. One sees that
$T_c(\Cone(A))=\Cone(A)$. Therefore, 
  the endofunctor
$T_{c*}:D(X\times \Re)\to D(X\times \Re)$ preserves
$\cD_A(X)$ for all $A$.  For any $c>0$,
one can construct a natural transformation
$\tau_c:
\Id\to T_{c*}$ of endofunctors on $\cD_A(X)$ for any $A$,
see Sec. \ref{preob:sdvig}.

We can now formulate the non-displaceablity condition 
(Theorem \ref{main}).

{\em Let $A,B\subset T^*X$ be compact subsets. Suppose
there exist $F_A\in \cD_A(X)$; $F_B\in \cD_B(X)$ such that
for any $c\geq 0$, the natural map
$R\hom(F_A;F_B)\to R\hom(F_A;T_{c*}F_B)$, induced by $\tau_c$,
does not vanish.  Then $A$ and $B$ are non-displaceable.}

{\bf Remark.}  For $c\in \Re$ set
 $H_c(F_A,F_B):=H_c:=R\hom(F_A;T_{c*}F_B)$. For any $d\geq 0$, the natural transformation $\tau_d$
induces a map $\tau_{c,c+d}:H_c\to H_{c+d}$. 

 Let $H(F_A,F_B):=H\subset \prod_{c\in \Re}H_c$ be defined as a subset consisting of all collections
$h_c\in H_c$ such that  there exists a sequence
$c_1<c_2<\cdots<c_n<\cdots$; $c_n\to \infty$ such that
$h_c=0$ for all $c\notin \{c_1,c_2,\ldots,c_n,\ldots\}$.
The maps $\tau_{c,c+d}$ induce  maps
$\tau_d:H\to H$ for all $d\geq 0$. This way we get 
an action of the semigroup $\Re_{\geq 0}$ on $H$. This implies
that Novikov's ring, which is a group ring of $\Re_{\geq 0}$,
acts on $H$. There are indications that thus defined
module over Novikov's ring $H$ is related to Floer
cohomology of the pair $A,B$. In this language, 
our nondisplaceability condition means that $H(F_A,F_B)$
 has a non-trivial non-torsion part.

{\bf Remark} It seems likely that under an appropriate version
of Riemann-Hilbert correspondence our picture should become
similar to the setting of \cite{NT}. This paper can be considered
as an attempt to translate \cite{NT} into the language of constructible
sheaves. 

{\bf Remark} There is some similarity between our theory and the
approach from \cite{NZ} where the authors identify the derived 
category of constructible 
sheaves on $X$ with a certain version of the Fukaya category
on $T^*X$. The authors use Lagrangian subsets of $T^*X$ which are close
to being conic, whereas we work with compact subsets of $T^*X$.

Let us now briefly describe the way our non-displaceability
 condition 
is applied to the above mentioned example 
$\RP^N,\mathbb{T}^N\subset \CP^N$.  As was explained, 
the problem can be reduced to proving non-displaceability
of certain subsets of $T^*\SU(N)$. Given such a subset, say
$A$, it is, in general, a non-trivial problem to
construct a non-zero object $F\in \cD_A(\SU(N))$. Our major
tool here is a certain object $\sp\in D(G\times \h)$ which
is defined uniquely up-to a unique isomorphism
by certain microlocal conditions to  be now specified.
Here $G=\SU(N)$ and $\h$ is the Cartan subalgebra of 
$\g$, the Lie algebra of $\SU(N)$. 

Let $C_+\subset \h$ be the positive Weyl chamber. 
For every $A\in \g$ there exists a unique element 
$\|A\|\in \h$ such that $\|A\|$ is conjugated with $A$.
Let us identify $T^*(G\times \h)=G\times \h\times \g^*\times \h^*$ (via interpreting $\g^*$ as the space of right-invariant
1-forms on $G$). Let us identify $\g^*=\g$, $\h^*=\h$ by means
of  Killing's form.
 Let $\Omega_\sp\subset G\times \h\times\g\times \h=
\Omega_\sp\subset G\times \h\times\g^*\times \h^*$ 
consist of all points of the form
$(g,X,\omega,\eta)$, where $\eta=\|\omega\|$.
Let also $i_0:G\to G\times \h$ be the embedding
$i_0(g)=(g,0)$. 
We then define $\sp$ as an object of $D(G\times \h)$
such that $\mS(\sp)\subset \Omega_\sp$ and $i_0^{-1}\sp\cong 
\gf_{e}$, where $\gf_e$ is the skyscraper at the unit 
$e\in G$. One can show that this way $\sp$ is determined
uniquely up-to a unique isomorphism.  It turns out
that the required  objects $F_A\in \cD_A(X),  F_B\in \cD_B(X), \ldots,$ can be easily expressed in terms of $\sp$. 

Our next task is to compute the graded
vector spaces  $R\hom(F_A,T_{c*}F_B)$ 
and to make sure that the maps
$\tau_c:R\hom(F_A,F_B)\to R\hom(F_A,T_{c*}F_B)$ are not zero
for all $c\geq 0$.  This problem gets gradually reduced
to finding an explicit description of the restriction
$i_e^{-1}\sp \in D(\h)$, where $i_e:\h\to G\times \h$,
$i_e(X)=(e,X)$, and $e\in G$ is the unit.

{\bf Remark.}
Let $C_-:=-C_+$, let $C_-^\circ\subset C_-$ be the interior.
It turns out that the stalks of $(i_e^{-1}\sp|_{C_-})$ have a transparent topological meaning (however, this meaning won't be
used in our proofs). Let $X\in C_-$; let
$O(X):=\sp|_{e\times X}[-\dim \h].$  

On the other hand, let us consider the smooth loop space
$\Omega G$. For $\gamma:[0,1]\to G$ being a smooth loop, we
set $\|\gamma\|\in C_+$,
$$
\|\gamma\|:=\int\limits_0^1 \|\gamma'(t)\|dt,
$$
where $\gamma'(t)\in \g$ is the $t$-derivative of $\gamma$.
Let $\Omega_X\subset \Omega(G)$ be the subspace
consisting of all loops $\gamma$ such that 
$\|\gamma\|\leq -X$ (here $Y\leq -X$ means
$<Y+X,C_+>\leq 0$, where $<,>$ is the restriction
of the positive definite invariant  form on $\g$ onto $\h$).
It can be shown that $\O(X)\cong H_\bullet(\Omega_X)$.

In regard with this setting, one can ask the following
question (which will be probably discussed in a subsequent
paper).  We have an obvious concatenation map
$\Omega_X\times \Omega_Y\to \Omega_{X+Y}$ whence a product
$\O(X)\otimes \O(Y)\to \O(X+Y)$. One can show that this 
product is commutative so that the  spaces
$\O()$  form a filtered commutative algebra. 
It can be shown that this  
algebra can be obtained in the following algebro-geometric
way. Let $\Fl$ be the projective $\gf$-variety  of
complete flags in $\gf^N$.  Fix a regular nilpotent operator
$n:\gf^N\to \gf^N$ (that is, $n$ consists of one Jordan block).
Let $\mathbf{Pet}\subset \Fl(N)$ be the closed subvariety
consisting of all flags $0=V_0\subset V_1\subset \cdots V_N=
\gf^N$ satisfying $nV_i\subset V_{i+1}$ for all $i<N$.
This variety was  discovered by Peterson, see e.g.
 \cite{Kost}. 
 
Let $\bL\in \h$ be the lattice formed by all 
elements $X$ such that $e^X$ is in the center of $G$.
Given $l\in \bL$ we canonically have a line bunble
$L_l$ on $\Fl$. It turns out that for all
$l\in \bL\cap C_+$ we have an isomorphism
$O(l)=\Gamma(\mathbf{Pet};L_l|_{\mathbf{Pet}})$,
and this isomorphism is compatible with the natural
product on both sides.

A  related result  is proven
in \cite{Bez}, where, among other interesting 
results, the authors identify $H_\bullet(\Omega(G))$
with the algebra of functions on a certain affine
open subset of $\mathbf{Pet}$.

Let us now go over the content of the paper.
In Sec. \ref{Generalities}-\ref{condition} we formulate
and prove the non-displaceability condition.

In Sec. \ref{cepeen} we start applying
the non-displaceability condition to 
$\RP^N,\cT^N\subset\CP^N$.  Finally, the problem is
reduced to the existence of an
object $u_\O\in \cD(G)$ satisfying certain properties
(see Proposition \ref{diagonal}). 

In Sec. \ref{const:uo} the object $u_\O$ gets constructed
out of $\sp$ (where we use certain properties
of $\sp$ to be proven in the subsequent sections).
 
The rest of the paper is devoted
to  constructing and studying $\sp$. 
In Sec. \ref{specp} we construct an object $\sp$ and
prove its uniqueness. 

In Sec. \ref{restrC-} we compute an isomorphism type
of $\sp|_{z\times C_-^\circ}$ where $z$ is any element
in the center of $G$. In essense, the computation is a version
of Bott's computation of $H_\bullet(\Omega(G))$ 
using Morse theory. 

The goal of  Sec. \ref{bsche} is to extend the result of
the previous section to $z\times \h$. This is done by means
of establishing a certain periodicity propery of
$\sp$ with respect to shifts along $\h$ by elements
of the lattice $\bL=\{X\in \h|e^X\in \Zentrum\}$,
where $\Zentrum\subset G$ is the center.  Namely, we show
that $\sp$ is, what we call, {\em a strict $B$-sheaf.}
(see Sec. \ref{strict:sh}).  We show that any
strict $B$-sheaf can be recovered from its
restricton onto $\Zentrum\times C_-^\circ$. By virtue of this
statement we are able to identify the isomorphism type
of $\sp|_{\Zentrum\times G}$.

There are two appendices. In the first one we 
introduce the notation used when working with $\SU(N)$ and
its Lie algebra. We also included a couple of useful Lemmas
(which, most likely, can be found elsewhere in the literature).
These Lemmas are mainly used when constructing and
studying $\sp$. The notation is used systematically
starting from Sec. \ref{const:uo}.

In the second appendix we list, for the reader's convenience, the rules for computing
the microsupport from \cite{KS}.  These rules are used 
throughout the paper. 

Strictly speaking, these rules are proved in \cite{KS} 
for the bounded derived category. However, one sees that
they carry over directly to the unbounded derived category,
in which case we use them.

{\em Acknowledgements} I would like to thank Boris Tsygan and Alexander 
Getmanenko for motivation and numerous fruitful discussions. I am  grateful
to Pavel Etingof, Roman Bezrukavnikov, Ivan Mirkovich, and David
Kazhdan for their explanations on Peterson varieties.

 \section{Generalities}\label{Generalities}

\subsection{Unbounded derived category}
\subsubsection{} Fix a ground field $\gf$.
 Abelian category $\Sh_M$ of sheaves of $\gf$-vector spaces 
on a smooth manifold $M$ is of finite injective dimension.  Therefore, one has
a simple model of the unbounded derived category $D(M)$, namely one can take 
 unbounded complexes of injective sheaves on $M$; given two such complexes, we define
 $\hom_{D(M)}(I_1,I_2):=H^0\hom^\bullet(I_1,I_2)$.
This definition is stable under quasi-isomorphisms precisely because of finite
injective dimension of $\Sh_M$. The main results of the formalism of 6 functors
remain valid for $D(X)$ (excluding the Verdier duality).

\subsubsection{} We still have a notion of singular support of an object
of $D(M)$ and it is defined in the same way as in \cite{KS}
The results on functorial properties of singular support from Ch 5,6 of 
\cite{KS} are still valid for the  unbounded derived category, and we will
freely use them. For the convenience of the reader the results
from \cite{KS} used in this paper are listed in 
Sec. \ref{appendix}

\subsection{Sheaves on $X\times \Re$}  Let $X$ be a smooth manifold.
We will work with the manifold $X\times \Re$. Let $t$ be the coordinate on
$\Re$ and let $V=\partial/\partial t$ be the vector field
corresponding to the infinitesimal shift along $\Re$.
Let $\Omega_{\leq 0}\subset T^*(X\times \Re)$ be the closed subset
consisting of all 1-forms $\omega$ with $(\omega,V)\leq 0$. Let
$\Omega_{>0}\subset T^*(X\times \Re)$ be the complement to
$\Omega_{\leq 0}$, i.e. the set of all 1-forms $\omega$ such that
$(\omega,V)>0$.

Let $C_{\leq 0}(X)\subset D(X\times \Re)$ be the full subcategory 
of objects microsupported on $\Omega_{\leq 0}$.
Let $\cD(X):=D(X\times \Re)/C_{\leq 0}(X)$.

\begin{Proposition} The embedding $C_{\leq 0}(X)\to D(X\times \Re)$
has a left adjoint. Therefore, $\cD(X)$ is equivalent
to the left orthogonal complement to $C_{\leq 0}(X)$
in $D(X\times \Re)$.
\end{Proposition}
\begin{proof} Let $p_1:X\times \Re\times \Re\to X\times \Re$;
$p_2:X\times \Re\times \Re\to \Re$; $a:X\times \Re\times \Re\to
\Re$ be given by $p_1(x,t_1,t_2)=(x,t_1)$; $p_2(x,t_1,t_2)=t_2$;
$a(x,t_1,t_2)=t_1+t_2.$. For $F\in D(X\times \Re)$ and
$S\in D(\Re)$ set $F*_\Re S:=Ra_!(p_1^{-1}F\otimes  p_2^{-1}S)$.

It is clear that $F*_\Re \gf_0\cong F$ where $\gf_0$ is the 
sky-scraper at $0\in \Re$. 

We have a natural map $\gf_{[0,\infty)}\to \gf_0$ in $D(\Re)$.

For an $F\in D(X\times \Re)$, consider the induced map
\begin{equation}\label{phi}
F*_\Re \gf_{[0,\infty)}\to F*_\Re \gf_0=F.
\end{equation}

1) Let us show that {\em $F*_\Re \gf_{[0,\infty)}$ is in the left 
orthogonal complement to $C_{\leq 0}(X)$.}

Indeed, let $G\in C_{\leq 0}(X)$. Let $U\subset X$ be an open subset
and let $(a,b)\subset \Re$. Any object $F\in D(X\times\Re)$
can be produced from  objects of the type $\gf_{U\times (a,b)}$
for various $U$ and $(a,b)$ by  taking direct limit.
Therefore, without loss of generality, one can assume $F=\gf_{U\times(a,b)}$. One then has
$$
R\hom_{X\times\Re}(F*_\Re \gf_{[0,\infty)};G)=
$$
$$
=R\hom_{X\times\Re}(\gf_{U\times (a,b)}*_\Re \gf_{[0,\infty)};G)
$$
$$
=R\hom_{X\times \Re}(\gf_{U\times [a,\infty)}[-1];G)
$$
$$
=\Cone(R\Gamma(U\times \Re;G)\stackrel r\to R\Gamma(U\times (-\infty;a);G)).
$$
The map $r$ is an isomorphism because $G\in C_{\leq 0}$. Therefore,
$\Cone(r)=0$, whence the statement.

2) {\em  Cone of the map (\ref{phi})  is in $C_{\leq 0}(X)$.}
Indeed, consider the cone of the map $\gf_{[0,\infty)}\to 
\gf_{0}$.
This cone is isomorphic to $\gf_{(0,\infty)}[1]$. One then   has to  check that
$F*_\Re \gf_{(0,\infty)}\in C_{\leq 0}(X)$. One can represent $F$ as an inductive
limit of compactly supported objects. Therefore, without loss
of generality, one can assume $F$ is compactly supported. One then 
can estimate the microsupport of $F*_\Re \gf_{[0,\infty)}$ using functorial 
properties of microsupport. Indeed, let us identify

$$
T^*(X\times \Re\times \Re)=T^*X\times T^*(\Re\times \Re).
$$

Let us also identify $T^*(\Re\times \Re)=\Re^4$ so that a point 
$(t_1,t_2,k_1,k_2)\in \Re^4$ corresponds to the 
1-form $k_1dt_1+k_2dt_2$ at the 
point $(t_1,t_2)\in \Re\times \Re$. We then have
$$
p_1^{-1}F\otimes p_2^{-1}\gf_{(0,\infty)}=
F\boxtimes \gf_{(0,\infty)};
$$
$$
\mS(F\boxtimes \gf_{(0,\infty)})
$$
$$
\subset \{(\omega,t_1,t_2,k_1,k_2)
\in  T^*X\times \Re^4
|
(t_2,k_2)\in \mS(\gf_{(0,\infty)})\}.
$$
 This means that either $t_2=0$ and $k_2\leq 0$ or $t_2>0$ and 
$k_2=0$.

As $F$ is compactly supported, it follows that 
the map $a$ is proper on the support of
$ F\boxtimes \gf_{(0,\infty)}.
$

Therefore, $\mS  Ra_!( F\boxtimes \gf_{(0,\infty)})
$ is contained in the set of all
points $(\omega,t,k)\in T^*X\times \Re^2$ such that
there exists a point $(\omega,t_1,t_2,k_1,k_2)\in \mS(F\boxtimes \gf_{(0,\infty)})$ such that $t=t_1+t_2$; $k_1=k_2=k$.
This implies that $k\leq 0$, therefore,
$$
F*_\Re \gf_{(0,\infty)}=Ra_!(F\boxtimes \gf_{(0,\infty)})\in
C_{\leq 0}(X),
$$
as was required.

The statements 1) and 2) imply that we have an exact triangle
$$
\to F*_\Re \gf_{(0,\infty)}\to F*_\Re\gf_{[0,\infty)} 
\to F\to F*_\Re \gf_{(0,\infty)}[1]\to\cdots,
$$                                                                   
  where $F*_\Re \gf_{(0,\infty)}[1]$ is in $C_{\leq 0}(X)$ and 
$ F*_\Re\gf_{[0,\infty)}$
is in the left orthogonal complement to $C_{\leq 0}(X)$. 
Therefore, $F\mapsto  F*_\Re \gf_{(0,\infty)}[1]$ is the left adjoint
functor  to
the embedding $C_{\leq 0}(X)\to D(X\times \Re)$.
\end{proof}

Thus, we have proven
\begin{Proposition}\label{checkorthogonal}
An object $F\in D(X\times \Re)$ is in the left orthogonal 
complement to $C_{\leq 0}(X)$ iff the map
 (\ref{phi}) is an isomorphism.
\end{Proposition}
\subsubsection{}
From now on we identify $\cD(X)$ with a full subcategory
of $D(X\times \Re)$ which is the left orthogonal complement to $C_{\leq 0}(X)$.
Thus, the arrow (\ref{phi})
is an isomorphism for any  $F\in\cD(X)\subset D(X\times \Re)$
(and only for objects from $\cD(X)$).
\subsubsection{}\label{preob:sdvig}
 Let $T_c:X\times \Re\to X\times \Re$
be the shift along $\Re$ by $c$: $T_c(x,t)=(x,t+c)$. 
We have $T_{c*}F=F*_\Re\gf_{c}$. If $F\in \cD$, we have
\begin{equation}\label{sdvigident}
T_{c*}F\cong F*_\Re \gf_{[0,\infty)}*_\Re\gf_c\cong
F*_\Re\gf_{[c,\infty)}.
\end{equation}
One can easily check that $T_{c*}F\in \cD(X)$; for example, this follows from an isomorphism
$$
F*_\Re\gf_{[c,\infty)}\cong T_{c*}F*_\Re \gf_{[0,\infty)},
$$
which is the case for any $F\in D(X\times \Re)$.

For all $c\geq d$ we then have a natural map $T_{d*} F\to T_{c*}F$ which is induced by the
embedding $[c,\infty)\subset [d,\infty)$ and we use the identification (\ref{sdvigident}).
 This implies that we have natural transformations
$\tau_{dc}: T_{d*}\to T_{c*}$ of endofunctors on $\cD(X)$ for all $d\leq c$.                                                                                                                                     It is clear that $\tau_{dc}\tau_{ed}=\tau_{ec}$ for all
$e\leq d\leq c$.

 \subsubsection{} \label{torsio} Call an object $F\in \cD(X)$ 
{\em a torsion    object} if there exists $c>0$ such that 
the natural map $\tau_{0c}:F\to T_{c*}F$ is zero in $\cD(X)$.                             
               
\subsubsection{} Still thinking of $\cD(X)$ as 
a quotient $D(X\times\Re)/C_{\leq 0}(X)$,
the microsupport of an object $F\in \cD(X)$ 
is naturally defined as a closed subset
 of $\Omega_{>0}\subset T^*(X\times \Re)$. Denote this microsupport
by $\mS_\cD(F)\subset \Omega_{>0}$. 

Let us see what this means in terms of the identification of $\cD(X)$ with a full subcategory
of $\cD(X)$ which is the  left orthogonal
complement to $C_\leq (X)$. Let $F\in\cD(X)\subset 
D(X\times \Re)$. We then have
$\mS_\cD(F)=\mS(F)\cap \Omega_{>0}$, where $\mS(F)$ is the microsupport of $F$ which is
viewed as an object of $D(X\times \Re)$.

\subsubsection{} Let us identify $T^*\Re=\Re\times \Re$ so that $(t_0,k)\in \Re\times \Re$  corresponds
to the 1-form $kdt$ at the point $t_0\in \Re$.
We then have an induced identification $T^*(X\times \Re)=T^*X \times \Re\times \Re$.

Let $A\subset T^*X$ be a subset. Define the conification $\cone(A)\subset \Omega_{>0}$ to consist of all points
$(\omega,t,k)\in T^*X\times \Re\times \Re$ such that  $k>0$, $(x,\omega/k)\in A$.
Let $\cD_A(X)\subset \cD(X)$ be the full subcategory
consisting of all objects $F\in \cD(X)$ such that
 $\mS_{\cD}(F)\subset \cone(A)$.

\section{Non-displaceability condition}\label{condition}
Let $X$ be a compact manifold. Let $L_1,L_2\subset T^*X$ be compact
subsets. Call $L_1,L_2$ mutually non-displaceable
if for every Hamiltonian symplectomorphism $\Phi$ of $T^*X$ which is
identity outside of a compact, $\Phi(L_1)\cap L_2 \neq \emptyset$.
Our goal is to prove 
\begin{Theorem}\label{main} Suppose there exist objects $F_i\in \cD_{L_i}(X)$,
$i=1,2$ such that for all $c>0$ the natural map
$$
\tau_c:R\hom(F_1,F_2)\to R\hom(F_1,T_cF_2)
$$
is not zero. Then $L_1$ and $L_2$ are mutually non-displaceable.
\end{Theorem}

The proof will occupy the whole section.
\subsection{Disjoint supports}
Our goal is to prove:
\begin{Theorem}\label{disjoint} let $F_i \in \cD_{A_i}(X)$, where $i=1,2$,
$A_i\subset T^*X$ are compact sets and $A_1\cap A_2=\emptyset$.
We then have $R\hom_{\cD(X)}(F_1,F_2)=0$.
\end{Theorem}

\subsubsection{Lemma}  Let $M$ be a smooth manifold let $E$ be
a finite-dimensional real vector space of dimension $\geq 1$.
Let $p:M\times E\to M$ be the projection. 
Let $F\in D(M\times E)$. Let $\omega\in T^*M$,
$\omega\neq 0$.
 Let $U\subset T^*M$ be a neighborhood of $\omega$. 
Let $V\subset E^*$ be a neighborhood of 0 in the dual vector space.
Let us identify $T^*(M\times E)=T^*M\times E\times E^*$.
\begin{Lemma}\label{impropernonsingular} Suppose that

$F$ is non-singular on the set
 $$ U\times E\times V\subset T^*M\times E\times E^*
=T^*(M\times E)
$$
Then $Rp_! F$ and $Rp_*F$  are non-singular at $\omega$.
\end{Lemma}
\begin{proof} We will only prove Lemma for $Rp_!F$; the proof
for $Rp_*F$ is similar.

 Fix a Euclidean inner product $<,>$ on $E$.
 Without loss of generality one can assume that $V=B\subset E^*$
is an open unit ball.

Let
$\theta:[0,\infty)\to [0,1)$ be a  function
such that:

--- $\theta'(x)>0$ for all $x\geq 0$;

---  there exists an $\ve>0$ such that for all $x\in [0,\ve]$
 we have
$\theta(x)=x$.

--- there exists an $M>0$ such that
 for all $x>M$, $\theta(x)=1-1/x$.

 Let $B:=\{v\in E|\ |v|<1\}$. Let
$Z:E\to B$ be the embedding given by
$$
Z(v)=\frac{\theta(|v|)}{|v|}v.
$$

It follows that $Z$ is a  diffeomorphism.
Let $J:B\to E$ be the open embedding
Let us split $p:M\times E\to M$ as
$$
M\times E\stackrel{\Id\times Z}\to 
M\times B\stackrel {\Id\times j}\to
M\times E\stackrel p\to X.
$$
Denote $z:=\Id\times Z$; $j:=\Id\times J$.
We have $Rp_!F=Rp_!j_!z_!F$. We then see that $p$ is proper
on the support of $j_!z_!F$. Let us estimate $\mS(z_!F)$.

Let $a(x):[0,1)\to [0,\infty)$ be the inverse function to $\theta$.
It follows that $a(x)=x$ for $x<\ve$ and
there exists $\delta>0$ such that for all $x\in (1-\delta;1)$, 
$a(x)=1/(1-x)$.
We then get $Z^{-1}v=(a(|v|)/|v|)v$.
 The condition of Lemma implies
that for all $\omega\in U$ and for all $c\in B$ 
we have  $(\omega,v,c)\notin \mS(F)$, where $v\in E$.
Let 
$$S_{UB}=\{(\omega,v,c)| \omega\in U;v\in E;c\in B\}\subset
T^*(M\times E).
$$

We then see that  the set 
$(Z^{-1})^*S_{UB}\subset T^*(X\times B)$;
 $$
(Z^{-1})^*S_{UB}=\{(\omega,v,
\sum_j c_jd((a(|v|)/|v|)v^j)\},
$$
where $\omega\in U$, $v\in B$, and $|c|<1$.
Let us now estimate $\mS(j_!z_!F)$. According to
 \ref{ks:openembedding},
we have
$$
\mS(j_!z_!F)\subset \mS(z_!F)\hat{+}
N^*(X\times  B)^a,
$$
where on the RHS we have a Witney sum of the following conic subsets of $T^*(M\times E)$: 

---we identify $\mS(z_!F)$ with a conic subset of $T^*(M\times E)$ as follows:
 $\mS(z_!F)\subset T^*(M\times B)\subset T^*(M\times E)$;

--- $N^*(M\times  B)^a$ is the exterior conormal cone to the boundary of $M\times B\subset M\times E$. We have
$$
N^*(M\times  B)^a=\{(\omega,b,tb)\in T^*M\times E\times E| |b|=1;t\geq 0\},
$$
where we identify $T^*M\times E\times E=T^*M\times E\times E^*$.
 
By definition one has:
$$
\mS(z_!F)\hat{+} N^*(M\times  B)^a=\mS(z_!F)\cup \Lambda,
$$
where 
$\Lambda$ consists of all points of the form
$(\omega,b,\eta)\in T^*M\times E\times E$ where

---$\omega\in T^*_{x_0}M$; so let us choose a nieghborhood $U_{x_0}$ of $x_0$ in $M$ and identify
$T^*U=U\times \Re^{\dim M}$; let us denote points of $T^*U$ by $(x,\zeta)$, $x\in U$;
$\zeta\in \Re^{\dim X}$;

---$b\in \dB$ and there exists a sequence
of points $(x_k,\omega_k,b_k,\eta_k)\in \mS(Z_!F)\cap T^*(U_{x_0}\times B)$;
$(\beta_k;t_k)\in \dB\times \Re_{\geq 0}$
 where $x_k\to x_0$; $b_k\to b$; $\beta_k\to b$;
$\omega_k\to \omega$; $\eta_k+2t_k \sum_j \beta_jdv_j\to\eta$;
$t_k(|\beta_k-b_k|+|x_k-x_0|)\to 0$.

We will show that $(x_0,b,\omega,0)\notin \Lambda$ for any
$b\in\dB$.  Let us prove the
  statement by contradiction.
Indeed, without loss of generality,
one can assume that $(x_k,\omega_k)\in U$, therefore,
$(b_k,\eta_k)\notin Z^{-1*}(E\times V).$ As $V\subset E^*$ is an
open unit ball, this means that
$(b_k,\eta_k)$ is of the form

$$
\eta_k =\sum_j c_k^j d(a(|b_k|)b_k^j/|b_k|)
$$
and $|c_k|\geq 1$.  as $b_k\to b$, $|b_k|\to 1$ and without loss
of generality one can assume $|b_k|>1-\delta$ so that
$a(|b_k|)=1/(1-|b_k|)$. Thus
$$
\eta_k=\sum_j c_k^j d(b_k^j/(|b_k|(1-|b_k|))
$$
Let $R_k=|b_k|$. We then have
$$
\eta_k =<c_k,db_k>/(R_k(1-R_k))+<c_k,b_k>\frac{2R_k-1}{R_k^3(1-R_k)^2}<b_k,db_k>
$$
so that
$$<\eta_k,\eta_k>=<c_k,c_k>/(R_k^2(1-R_k)^2)+<c_k,b_k>^2
\frac{(2R_k-1)^2}{R_k^4(1-R_k)^4}
$$
$$
+2<c_k,b_k>^2\frac{2R_k-1}{R_k^4(1-R_k)^3}
$$
$$
><c_k,c_k>/(R_k^2(1-R_k)^2)>1/(1-R_k)^2
$$
as long as  $R_k>1/2$ which is the case for all $k$ large enough,
without loss of generality we can assume that $R_k>1/2$ for all $k$.
Thus, $|\eta_k|>1/(1-R_k)$.

Therefore,$$
|\eta_k+2t_k \sum_j \beta_k^jdv_k^j|\geq |\eta_k|-2|t_k||\beta|>
1/(1-R_k)-2t_k
$$

By assumption $|\eta_k+2t_k \sum_j \beta_k^jdv_k^j|\to 0$, hence
$$
1/(1-R_k)-2t_k\to 0
$$

and $2t_k(1-R_k)\to 1$. On the other hand,
we have
$$
t_k(|b_k-\beta_k|)\geq t_k(1-R_k),
$$
because $|\beta_k|=1$ and $|b_k|=R_k$.
Therefore, $t_k(1-R_k)\to 0$.
We have a contradiction which shows that
as long as $(x,\omega)\in U$, $(x,\omega,e,0)\notin \mS(j_!Z_!F)$.
Since the map $p:X\times E\to X$ is proper on the support 
of $j_!Z_!F$ (i.e.  $X\times \overline{B}$) we know that
$(x,\omega)\notin \mS(Rp_!j_!Z_!F)$ which proves Lemma
\end{proof}
\begin{Corollary}\label{impropersupport}\label{improperstar}
Let $F\in D(X\times E)$ and let $p:X\times E\to
X$, $\kappa:T^*X\times E\times E^*\to T^*X\times E^*$ be the projections.  Let $\cI:T^*X\to T^*X\times E^*$ be the embedding
given by $I(x,\omega)=(x,\omega,0)$. We then have
$$
\mS(Rp_!F),\mS(Rp_*F)\subset \cI^{-1}\overline{\kappa(\mS(F))},
$$
where the bar means the closure.
\end{Corollary}
\begin{proof} Clear
\end{proof}

\subsubsection{Kernels and convolutions}\label{kernelconv}
 Let $X_1,X_2,X_3$ be manifolds.
We are going to  define a functor $$D(X_1\times X_2\times \Re)
\times D(X_2\times X_3\times \Re)\to D(X_1\times X_3\times \Re).$$

Let 
\begin{equation}\label{pij}
p_{ij}:X_1\times X_2\times X_3 \times \Re\times \Re\to 
X_i\times X_j\times \Re
\end{equation}
 be the following maps
$$
p_{12}(x_1,x_2,x_3,t_1,t_2)=(x_1,x_2,t_1);
$$
$$
p_{23}(x_1,x_2,x_3,t_1,t_2)=(x_2,x_3,t_2);
$$

$$
p_{13}(x_1,x_2,x_3,t_1,t_2)=(x_1,x_3,t_1+t_2).
$$

Let $A\in D(X_1\times X_2\times \Re)$ and 
$B\in D(X_2\times X_3\times\Re)$. Set
$$
A\bullet_{X_2} B:=Rp_{13!}(p_{12}^{-1}A\otimes p_{23}^{-1}B),
$$
$A\bullet_{X_2} B\in D(X_1\times X_3\times \Re)$.

Let now $X_k$, $k=1,2,3,4$, are manifolds and
let $A_k\in D(X_k\times X_{k+1}\times \Re)$, $k=1,2,3$.
We then have a natural isomorphism
$$
(A_1\bullet_{X_2} A_2)\bullet_{X_3} A_3\cong
A_1\bullet_{X_2}(A_2\bullet_{X_3} A_3).
$$

Let $A\in D(X\times \Re)$ and $S\in D(\Re)$.  Let $\pt$ be a point.
We then have $A*_\Re S\cong A\bullet_\pt S\cong S\bullet_\pt A$.

Let  $A\in \cD(X_1\times X_2)$ and 
$B\in D(X_2\times X_3\times \Re)$. 
Then $A\bullet_{X_2} B\in \cD(X_1\times X_3\times \Re)$. Indeed, according to Lemma \ref{checkorthogonal}, we need to check that the natural map
$$
\gf_{[0,\infty)}*_\Re (A\bullet_{X_2} B)\to
 \gf_0*_\Re (A\bullet_{X_2} B)
$$
is an isomorphism.

It follows that this map is isomorphic to a map
$$
(\gf_{[0,\infty)}\bullet_\pt A)\bullet_{X_2} B\to
(\gf_0\bullet_\pt A)\bullet_{X_2} B
$$
which is, in turn, induced by the natural map
$$
\gf_{[0,\infty)}\bullet_\pt A\to\gf_{0}\bullet_\pt A
$$
 which is an isomorphism because $A\in \cD(X_1\times X_2)$.

In particular, it follows that 
$$
\bullet_{X_2}:\cD(X_1\times X_2)\times \cD(X_2\times X_3)\to 
\cD(X_1\times X_3).
$$
\subsubsection{Fourier transform}\label{Fourier}
Let $E=\Re^n$ be a real vector space and let $E^*$ be the dual
space. Let $G\subset E\times E^*\times \Re$ be a closed subset
$G=\{(X,P,t)|<X,P>+t\geq 0\}$, where $<,>:E\times E^*\to \Re$ is 
the pairing.  One sees that $\gf_G\in \cD(E\times E^*)$.
Let $\Gamma\subset E^*\times E\times\Re$ be 
a closed subset $G=\{(P,X,T)|-<P,X>+t\geq 0\}$. 
Again, we have $\gf_\Gamma\in \cD(E^*\times E\times \Re)$.

Define functors $F: \cD(E)\to \cD(E^*)$; $\Phi:\cD(E^*)\to \cD(E)$
 as follows. Set
$$
F(A):=A\bullet_E \gf_G;
$$
$$
\Phi(B):=B\bullet_{E^*} \gf_\Gamma.
$$
$F,\Phi$ are called 'Fourier transform'.

Let us study
  the composition 
$\Phi\circ F:\cD(E)\to \cD(E)$. We have an isomorphism
$$
\Phi\circ F(A)\cong A\bullet_E (\gf_G\bullet_{E^*}\gf_\Gamma).
$$

Let us compute $\gf_G\bullet_{E^*}\gf_\Gamma$. 
Let $$q:E\times E^*\times E\times \Re\times \Re\to E\times E\times\Re$$ be given by $q(X_1,P,X_2,t_1,t_2)=(X_1,X_2,t_1+t_2)$.
By definition, we have
$$
\gf_G\bullet_{E^*}\gf_\Gamma=Rq_!\gf_K,
$$

where $$
K=\{(X_1,P,X_2,t_1,t_2)| t_1+<X_1,P>\geq 0; t_2-<X_2,P>\geq 0\}
$$
Let us decompose $q=q_1q_2$, where
$$
q_2:E\times E^*\times E\times\Re\times \Re\to E\times E^*\times E\times \Re
$$
$q_2(X_1,P,X_2,t_1,t_2)=(X_1,P,X_2,t_1+t_2)$;
and 
$$
q_1: E\times E^*\times E\times\Re\to E\times E\times \Re,
$$
$$
q_1(X_1,P,X_2,t)=(X_1,X_2,t).
$$
We see that $q_2(K)=L:=\{(X_1,P,X_2,t)| t+<X_1-X_2,P>\geq 0\}$.
Furthermore, the map $q_2|_K:K\to L$ is proper; it is also
a Serre fibration
with a contractible  fiber. Therefore, we have an isomorphism
$Rq_{2!}\gf_K\cong \gf_L$.

Let us now compute $Rq_{1!}\gf_L$. Let $\Delta\subset E\times E^*\times E\times \Re$ be given by
$$
\Delta=\{(X_1,P,X_2,t)|X_1=X_2;t\geq 0\}.
$$
We have $\Delta\subset L$ so that we have an induced map
$$
\gf_L\to \gf_\Delta.
$$
It is easy to check that the induced map
$$
Rq_{1!}\gf_L\to Rq_{1!}\gf_\Delta
$$
is an isomorphism. 

We also have an isomorphism $Rq_{1!}\gf_\Delta\cong
\gf_{\{(X_1,X_2,t)|X_1=X_2;t\geq 0\}}[-n].$

Thus, we have an isomoprhism  
$$
Rq_! \gf_K=\gf_{\{(X_1,X_2,t)|X_1=X_2;t\geq 0\}}[-n]
$$

For any $A\in D(E\times \Re)$, we have an isomorphism
$$
A\bullet_E \gf_{\{(X_1,X_2,t)|X_1=X_2;t\geq 0\}}\cong
A*_\Re \gf_{[0,\infty)}.
$$
Thus we have an isomorphism of  functors
$$
\Phi(F(\cdot))\cong (\cdot)*_\Re \gf_{[0,\infty)}[-n]
$$
The functor on the RHS acts on $\cD(E)$ as the shift by $-n$.
Thus we have established an isomorphism of functors
$
\Phi\circ F\cong \Id[-n].
$
Analogously, we can prove $F\circ \Phi\cong \Id[-n]$. We have proven:
\begin{Theorem} $\Phi[n]$ and $F$ are mutually inverse
equivalences of $\cD(E)$ and $\cD(E^*)$.
\end{Theorem}
\subsubsection{}
 Let us now study the effect of the Fourier transform on the microsupports. 
 Let $$a:T^*E=E\times E^*\to T^*E^*=E^*\times E$$ be given by
$a(X,P)=(-P,X)$. It is clear that $a$ is a symplectomorphism.
\begin{Theorem}\label{Fouriersupport} Let $A\subset T^*E$ be a closed subset
and  $S\in \cD_A(E)$. Then $F(S)\in \cD_{a(A)}(E^*)$.

Let $B\in T^*E^*$ be a closed subset and $S\in \cD_B(E^*)$.
Then $\Phi(S)\in \cD_{a^{-1}(B)}(E)$.
\end{Theorem}
\begin{proof} By definition, we have $$
F(S)=Rp_{13!}(p_{12}^{-1}S\otimes p_{23}^{-1}\gf_G).
$$
Here the maps $p_{ij}$ are the same as in (\ref{pij}) for $X_1=\pt$; $X_2=E$;
$X_3=E^*$.

The condition $S\in \cD_A(E)$ means that
$\mS(S)$ is contained in the set $\Omega_0$ of all points
$$(x,t,\omega,k)\in E\times \Re\times E^*\times \Re=T^*(E\times \Re),$$
where either $k\leq 0$ or $k>0$ and $(x,\omega/k)\in A$.

Therefore, 
$$
\mS(p_{12}^{-1}S)\subset  \Omega_1:=
\{(X,P,t_1,t_2,\omega,0,k,0)|(X,t,\omega,k)\in \Omega_0\}.
$$

As $G\subset E\times E^*\times \Re$ is defined by the equation
$t+<X,P>\geq 0$, we know that 
$\mS(\gf_G)$ consists of all points of the form
$$
(X,P,t,kP,kX,k)\in E\times E^*\times \Re\times E^*\times E\times \Re
$$
where $t+<X,P>\geq 0$, $k\geq 0$ and $k>0$ implies $t+<X,P>=0$.

Therefore
$$
\mS(p_{23}^{-1}\gf_G)=\Omega_2:=\{(X,P,t_1,t_2,k_1P,k_1X,0,k_1)|(X,P,t_2,
k_1P,k_1X,k_1)\in \mS(\gf_G)\}.
$$
We see that $\Omega_1\cap -\Omega_2$ is contained
in the zero section of $T^*(E\times E^*\times \Re\times \Re)$. Therefore,
$$
\mS((p_{12}^{-1}S)\otimes (p_{23}^{-1}\gf_G))
$$
is contained in the set of all poits of the form $\omega_1+\omega_2$
where $\omega_i\in \Omega_i$ and $\omega_1,\omega_2$ are in
the same fiber of $T^*(E\times E^*\times \Re\times \Re)$.

We have
$$
\mS(p_{12}^{-1}S\otimes p_{23}^{-1}\gf_G)\subset \Omega_3,
$$
where
$
\Omega_3
$
consists of all points of the form
$$
(X,P,t_1,t_2,\omega+k_1P,k_1X,k,k_1)
$$
where:

--- if $k>0$, then $(X,\omega/k)\in A$;

--- $t_2+<X,P>\geq 0$;

--- $k_1\geq 0$;

--- if $k_1>0$, then $t_2+<X,P>=0$.

Let $I:E\times E^*\times \Re\times \Re\to E\times E^*\times \Re\times\Re$ be given by
$$
I(X,P,t_1,t_2)=(X,P,t_1+t_2;t_2).
$$
Let 
$\pi:E\times E^*\times \Re\times \Re\to  E^*\times \Re$
be given by $\pi(X,P,t_1,t_2)=(P,t_1)$. We then have
$p_{13}=\pi I$; 
$$
Rp_{13!}(p_{12}^{-1}S\otimes p_{23}^{-1}\gf_G)\cong
R\pi_! I_!(p_{12}^{-1}S\otimes p_{23}^{-1}\gf_G).
$$

It is easy to see that
$$
\mS I_!(p_{12}^{-1}S\otimes p_{23}^{-1}\gf_G)
$$
is contained in the set $\Omega_4$ of
 of all points
$(X,P,t_1+t_2,t_2,\omega,\eta,k,k_1-k)$
where  $(X,P,t_1,t_2,\omega,\eta,k,k_1)\in \Omega_3$.

Suppose that a point $(P,\xi)\in E^*\times E=T^*E^*$ does not belong to
$a(A)$, that is $(-\xi,P)\notin A$. We will prove that
$R\pi_! I_!(p_{12}^{-1}S\otimes p_{23}^{-1}\gf_G)$ is non-singular
at any point of the form $(P,t,\xi,1)\in E^*\times \Re\times E\times \Re=T^*(E^*\times \Re)$ (this means precisely that
$ R\pi_! I_!(p_{12}^{-1}S\otimes p_{23}^{-1}\gf_G)\in \cD_{a(A)}(E^*)$.)

According to Lemma \ref{impropernonsingular}, it suffices to 
find an $\ve>0$ such that any point of the form
$$
(X,P',t_1,t_2,\omega,\eta,k',k_1')
$$
with $|P'-P|<\ve$; $|\omega|<\ve$; $|\eta-\xi|<\ve$; $|k'-1|<\ve$;
$|k'_1|<\ve$ is not in $\Omega_4$. Assume it is, then there should
exist a point $(X,P',t_1,t_2,\omega+k_1P',k_1X,k,k_1)\in \Omega_3$
such that $|P'-P|<\ve$; $|\omega+k_1P'|<\ve$; $|k_1X-\xi|<\ve$;
$|k-1|<\ve$; $|k'-k|<\ve$.  If $\ve$ is small enough, we
have $k,k_1>0$ and $(X,\omega/k)\in A$.  For any $\delta>0$, there exists a $\ve>0$ such that   these conditions imply:
\begin{equation}\label{bliz}
|\omega+P|<\delta; |X-\xi|<\delta.
\end{equation}

However,we know that $(\xi;-P)=a^{-1}(P,\xi)\notin A$. As $A$
is closed, for $\delta$ small enough, there will be no
points in $A$ satisfying (\ref{bliz}).

The proof of Part 2 is similar.
\end{proof}
\subsubsection{Lemma}\begin{Lemma}\label{first} Let $S\in \cD_A(X)$ where $A$ is a compact.
  Then $\ms(S)\cap \Omega_{\leq 0}(X)\subset
T^*_{X\times \Re}(X\times \Re)$. That is $S$ is non-singular at every point of the form $(x,t,\omega,kdt)$, where  either 
$k\leq 0$ and $\omega\neq 0$ or  $k<0$.
\end{Lemma}
\begin{proof} Choose a point $x_0\in X$, coordinates $x^i$
near $x_0$ so that $x_0$ has zero coordinates and let $U$
be a small neighborhood of $x_0$ given by $|x^i|<1$ for all $i$.  Consider the set $A\cap T^*U$. This set is contained
in the set $B:=\{(x,\sum a_idx^i)| |a_i|\leq M\}$, for some 
$M>0$ large enough. Let $\psi:\Re\to (-1,1)$ be an increasing 
surjective smooth function whose derivative is bounded (say 
$\psi(x)= (2/\pi) \text{arctan}(x)$).  Fix a constant $C>0$
such that $0<\psi'(x)\leq C$ for all $x$.

we then have a diffeomorphism $\Psi:E:=\Re^n\to U$, $\Psi(X^1,X^2,\ldots,X^n)=(\psi(X^1),\psi(X^2),\ldots, \psi(X^n))$. It then follows that the set
$\Psi^{-1}B$ consists of all points $(X,\sum a_i d\psi(X^i))$,
where $|a_i|<M$. But 
$\sum a_i d\psi(X_i)=\sum a_i \psi'(X_i)dX_i$. We know that
$|a_i\psi'(X_i)|<CM=:M_1$. Let $V\subset E^*$ be given by
$\{\sum_i b_idX^i| |b_i|\leq M_1\}$ so that 
$\Psi^{-1}B$ is contained in the set  $E\times V\subset E\times E^*=
T^*E$.  

Let $S\in \cD_A(X)$. It  follows that
$G:=\Psi^{-1}(S|_{U\times \Re})\in \cD_{E\times V}(E)$.
Our task now reduces to showing: {\em let $G\in \cD_{E\times
V}(E).$
Then $G$ is nonsingular
at a point  $(X,t,\omega, kdt)\in E\times \Re\times E^*\times
\Re$ if either $k<0$ or $k=0$ and $\omega\neq 0$.}

The statement will be proven using  the 
Fourier transform. 

First, we have an isomorphism $G=\Phi(F(G))[n]$.
Next, Theorem \ref{Fouriersupport} 
implies that $H:=F(G)\in \cD_{V\times E}(E^*)$.
Let $W\subset E^*\backslash V$ be an open subset such that 
its closure is also a subset of  $E^*\backslash V$.
We then see that the restriction
$H|_{W\times \Re}$ is both in $ C_{\leq 0}(U)$ (clear) and
in the left orthogonal complemenent to $ C_{\leq 0}(U)$ 
(follows from (\ref{checkorthogonal})).  Therefore, 
$H|_{W\times \Re}=0$.
Hence $H$ is supported on $V\times \Re\subset E^*\times \Re$.
Let us now study $\Phi(H)[n]=G$.
We have 
$$
\Phi(H)=Rp_{13!}(p_{12}^{-1}H\otimes
 p_{23}^{-1}\gf_{\{(P,X,t)|t-
<X,P>\geq 0\}}),
$$
where $p_{ij}$ are the same as in (\ref{pij}) with
$X_1=\pt$; $X_2=E^*$; $X_3=E$.
We need to show that {\em if $(X,t,\omega,k)\in \mS(\Phi(H))$ and
$k\leq 0$, then $k=0$ and $\omega=0$.}

We have
$$
\mS(p_{12}^{-1}H)\subset\Omega_1= \{(P,X,t_1,t_2,\pi,0,k_1,0)|
P\in V\};
$$
$$
\mS( p_{23}^{-1}\gf_{\{(P,X,t)|t-
<X,P>\geq 0\}})\subset \Omega_2=\{(P,X,t_1,t_2,-kX,-kP,0,k)|k\geq 0
\}
$$
Let $\omega_i\in \Omega_i$ belong to the fiber of  
$T^*(E^*\times E\times\Re\times \Re)$ over a point $(P,X,t_1,t_2)$. 
It is clear that $\omega_1+\omega_2=0$ implies that $\omega_2=
\omega_1=0$. Therefore, we have
$$
\mS(p_{12}^{-1}H\otimes
 p_{23}^{-1}\gf_{\{(P,X,t)|t-
<X,P>\geq 0\}})
$$
$$
\subset \Omega_3=\{(P,X,t_1,t_2,\pi-kX,-kP,k_1,k)|k\geq 0;
 P\in V\}
$$

Let us decompose $p_{13}=pI$, where
$I:E^*\times E\times \Re\times \Re\to E^*\times E\times \Re\times\Re$ is given by $I(P,X,t_1,t_2)=(P,X,t_1+t_2,t_2)$ 
and $p(P,X,T_1,T_2)=(P,T_1)$.
We then see that $$\mS(I_!(p_{12}^{-1}H\otimes
 p_{23}^{-1}\gf_{\{(P,X,t)|t-
<X,P>\geq 0\}}))\subset \Omega_4,$$
where 
$
\Omega_4
$
consists of all points of the form
$(P,X,t_1,t_2,\pi-kX,-kP,k_1,k-k_1)$, where
$(P,X,t_1,t_2,\pi-kX,-kP,k_1,k)\in \Omega_3$.

Assume $$(X',t,\omega,k')\in \mS(Rp_!I_!(p_{12}^{-1}H\otimes
 p_{23}^{-1}\gf_{\{(P,X,t)|t-
<X,P>\geq 0\}}))$$
and $k'\leq 0$.  We are to show $k'=0,\omega=0$.

According to Lemma \ref{impropernonsingular}
for any $\ve>0$ there should exist  a point
$(P,X,t_1,t_2,\pi-kX,-kP,k_1,k-k_1)\in \Omega_4$
such that  $|-kP-\omega|<\ve$; 
$|k_1-k'|<\ve$;
  $|k-k_1|<\ve$, $P\in V$, $k\geq 0$, $k'\leq 0$.
Therefore, $-k'\leq k-k'=|k-k'|\leq |k-k_1|+|k_1-k'|<2\ve$.
Similarly, $k<2\ve$.
Since $\ve$ can be made arbitrarily small, $k'=0$. 
Next, $|\omega|<\ve+|k||P|$. As $V$ is bounded, 
there exists $D>0$ such that $|P|<D$. Thus, 
$|\omega|<\ve(1+2D)$ for any $\ve>0$. Therefore, $\omega=0$. 
\end{proof}

\subsubsection{} Choose $F_1,F_2$ in the left orthogonal complement
to $C_{\leq 0}(X)$.

Consider the following sheaf on $X\times \Re$:
$$
H:=Rp_{2*}R\ihom(p_1^{-1}F_1;a^!F_2),
$$
where $p_1,p_2,a:X\times \Re\times \Re\to X\times \Re$
are given by: $p_i(x,t_1,t_2)=(x,t_i)$; $a(x,t_1,t_2)=(x,t_1+t_2)$.

Let $q:X\times \Re\to X$ be the projection.
\begin{Lemma}\label{third}
One has
1) $R\hom(F_1,F_2)=R\hom_\Re(\gf_0;Rq_*H)$;

2) $R\hom_\Re(\gf_{\Re};Rq_*H)=0$;

3) $Rq_*H$ is locally constant along $\Re$.
\end{Lemma}

\begin{proof}

Let $S\in D(\Re)$. We have
$$
R\hom_\Re(S;Rq_*H)=R\hom_\Re(S;R\pi_*\ihom(p_1^{-1}F_1;a^!F_2)),$$
where $\pi=qp_2:X\times\Re\times\Re\to \Re$;
$\pi(x,t_1,t_2)=t_2$.

Next,
$$
R\hom_\Re(S;R\pi_*\ihom(p_1^{-1}F_1;a^!F_2))
$$
$$
\cong R\hom_{X\times \Re}(Ra_!(\pi^{-1}S\otimes p_1^{-1}F_1);F_2)
$$
$$
\cong  R\hom_{X\times \Re}(F_1*_\Re S;F_2).
$$

Thus,
$$
R\hom(S;Rq_*H)\cong R\hom_{X\times \Re}(F_1*_\Re S;F_2).
$$

Let us now prove 1)

We have:
$$
R\hom_\Re(\gf_0;Rq_*H)=R\hom_{X\times \Re}(F_1*_\Re \gf_0;F_2)=
R\hom(F_1,F_2).
$$

2) We have $$
R\hom_\Re(\gf_{\Re};Rq_*H)=R\hom_{X\times \Re}(F_1*_\Re \gf_\Re;F_2)
$$

As $F_1\in\cD(X)$, we have an isomorphism
$$
F_1*_\Re \gf_{[0,\infty)}*_\Re \gf_\Re\to  F_1*_\Re\gf_{\Re}.
$$
However, one can easily check that $\gf_{[0,\infty)}*_\Re \gf_\Re=0$.Therefore, 

$$
F_1*_\Re\gf_{[0,\infty)}*_\Re \gf_\Re=0,
$$
whence the statement.

3) Let us identify $T^*(X\times \Re)=T^*X\times \Re^2$;
$T^*(X\times \Re\times \Re)=T^*X\times \Re^4$ so that
$(\omega,t,k)\in T^*X\times \Re^2$ corresponds to a point
$(\omega,\eta)\in T^*X\times T^*\Re$, where $\eta$ is a 1-form
$kdt$ at the point $t\in \Re$; analogously, we let 
$(\omega,t_1,t_2,k_1,k_2)$ correspond to a point
$(\omega,\zeta)\in T^*X\times T^*(\Re\times \Re)$ where
$\zeta=k_1dt_1+k_2dt_2$ is a 1-form at the point 
$(t_1,t_2)\in \Re^2.$

 According to Lemma \ref{first}, We know that
$$
\mS(F_1)\cap \{(\omega,t,k)|k\leq 0\} \subset T^*_{X\times \Re}(X\times \Re). 
$$

Since $F_1\in \cD_{A_1}(X)$, we have
$$
\ms(F_1)\cap \{(\omega,t,k)|k> 0\} \subset \{(\omega,t,k)|
k>0; (x,\omega/k)\in A_1\}.
$$
Thus, 
$$
\mS(F_1)\subset \{(k\omega,t,k)| k\geq 0;\omega\in A_1\} 
$$
Analogously,
$$
\mS(F_2)\subset \{(k\omega,t,k)| k\geq 0;\omega\in A_2\}. 
$$

Therefore, $$\mS(p_1^{-1}F_1)\subset 
\{(k_1\omega_1,t_1,t_2,k_1,0)| k_1\geq 0; \omega_1\in A_1\};
$$
 $$\mS(a^!F_2)\subset 
\{(k_2\omega_2,t_1,t_2,k_2,k_2)| k_2\geq 0; \omega_2\in A_2\}.
$$

In order to estimate $\mS R\ihom(p_1^{-1}F_1;a^!F_2)$ one should first check that $\mS(p_1^{-1}F_1)\cap \mS(a^!F_2)\subset
 T_{X\times \Re\times \Re}^*(X\times \Re\times \Re)$. This is indeed
so, because every point $p$ in
 $\mS(p_1^{-1}F_1)\cap \mS(a^!F_2)$
is of the form
$$
p=(k_1\omega_1,t_1,t_2,k_1,0)=(k_2\omega_2,t_1,t_2,k_2,k_2).
$$
 which  implies $k_1=k_2=0$, hence $
k_1\omega_1=k_2\omega_2=0$. Therefore, one has
$$
\mS R\ihom(p_1^{-1}F_1;a^!F_2)\subset \{(k_2\omega_1-k_1\omega_2,t_1,t_2
;k_2-k_1;k_2)| k_1,k_2\geq 0;
\omega_1\in A_1;\omega_2\in A_2\},
$$
where it is also assumed that $\omega_1,\omega_2$ belong to
the same fiber of $T^*X$.
Let $q':X\times\Re\times\Re\to \Re\times \Re$ be the projection.
Consider an object
$$
G:= Rq'_*R\ihom(p_1^{-1}F_1;a^!F_2)
$$
so that $Rq_*H=Rq_*Rp_{2*}R\ihom(p_1^{-1}F_1;a^!F_2)=Rp'_{2*}G$,
where $p'_2:\Re\times \Re\to \Re$ is the projection along the first factor; $p'_1(t_1,t_2)=t_2$.

As the map $q'$ is proper,  the microsupport of $G$ can 
be estimated as
$$
\mS(G)\subset \{(t_1,t_2,k_2-k_1,k_2)| k_1,k_2\geq 0;
\exists \omega_i\in A_i: k_1\omega_1=k_2\omega_2\},
$$
where again it is assumed that $\omega_i$ are in the same fiber of
$T^*X$.
Denote the set on the RHS by
 $\Gamma\subset \Re^4=T^*(\Re\times\Re)$.
Let us now estimate $\mS(Rq_*H)=Rp'_{2*}G$  using Lemma \ref{improperstar}.  

Let us first prove that $(t,1)\notin \mS(q_*H)$, where we identify 
$T^*\Re=\Re\times\Re$. Assuming the opposite implies that
for any $\ve>0$ there should
exist $(t_1,t_2,k_2-k_1,k_2)\in \Gamma$ such that 
$|k_2-k_1|<\ve$; $|k_2-1|<\ve$. As $A_1,A_2$ are compact
and do not intersect, it is clear that for $\ve$ small
enough we have $k_1A_1\cap k_2A_2=\emptyset$ which contradicts
to $(t_1,t_2,k_2-k_1,k_2)\in \Gamma$.

Let us now show that $Rq_*H$ is non-singular at any point 
$(t,-1)$.  Similar to above, assuming the contrary implies
that for any $\ve>0$ there should exist
$(t_1,t_2,k_2-k_1,k_2)\in \Gamma$ such that 
$|k_2+1|<\ve$. As $k_2\geq 0$, this leads to contradiction.
\end{proof}
\subsubsection{Proof of Theorem \ref{disjoint}}
It now follows that $Rq_*H$ is a constant sheaf on $\Re$ with 
$R\Gamma(\Re,Rq_*H)=0$, i.e.
 $Rq_*H=0$. Hence $R\hom(F_1,F_2)=0$ by Lemma \ref{third} 1).

This proves Theorem \ref{disjoint}.
\subsection{Hamiltonian shifts} Let $\Phi:T^*X\to T^*X$ be a Hamiltonian symplectomorphism which is equal to identity outside of a compact. 
Let $L\subset T^*X$ be a compact subset.

\begin{Theorem}\label{hamilt} There exist:

 a collection of endofunctors $T_n:\cD(X)\to \cD(X),$,
 $1\leq n\leq N$ for some $N$,
and a collection of transformations of functors
$t_k:T_{2k}\to T_{2k+1}$ (for all $k$ with $2k+1\leq N$).
 $s_k:T_{2k+2}\to T_{2k+1}$ (for all $k$ with $2k+2\leq N$);

Such that 

1) $T_N=\Id$;

2) $T_1(\cD_L(X))\subset \cD_{\Phi(L)}(X)$;

3) For all $k$ and for all $F\in \cD(X)$, we have
$\cone(t_k(F))$ and $\cone (s_k(F))$ are torsion sheaves (see Sec. \ref{torsio})
\end{Theorem}

\subsubsection{Singular support of convolutions.} 
Let $A\in T^*X$ and
 $B\subset T^*(X\times Y)=T^*X\times T^*Y$
be compact subsets. Let $C\subset T^*Y$; 
$$
C:=A\conv B=\{p\in T^*Y|
\exists q\in A: (-q, p)\in B\}.$$
\subsubsection{Lemma}
\begin{Lemma}\label{podzhim}
 Let $A_1\supset A_2\supset \cdots\supset A_n\supset \cdots\supset
A$ be a collection of compact sets such that $\bigcap_i A_i=A$.
Let $U\supset C$ be an open neighborhood. There exists an $N>0$
such that for all $n>N$, $A_n\conv B\subset U$.
\end{Lemma}
\begin{proof}  Assume not and pick points
$b_n\in (A_n\conv B)\backslash U$. One then has points 
$a_n\in A_n$ such that $(-a_n,b_n)\in B$. As $B$ is compact, one can choose a convergent subsequence $a_{n_k}\to a$ and $b_{n_k}\to b$.
It follows that $(-a,b)\in B$.
We see that $a\in A_{n_k}$ for all $k$, hence $a\in A$.
Therefore, $b\in C$. On the other hand, as $b_{n_k}\notin U$, $b\notin U$, we have
a contradiction.

\end{proof}
\subsubsection{}
Let $A,B,C$ are compact sets as above.
\begin{Proposition}\label{svert}  Let $F\in \cD_A(X)$; $K\in \cD_B(X\times Y)$.
Then $F\conv K\in \cD_C(Y)$.
\end{Proposition}
\begin{proof} It suffices to prove: {\em let $(y_0,\eta_0)\notin C$.
 Then
$F\conv K$ is nonsingular at $(y_0,t,\eta_0,1)$ for all $t\in \Re$.}

Let us identify of $T^*(X\times Y\times \Re\times \Re)=T^*X\times T^*Y\times T^*(\Re\times\Re)=T^*X\times T^*Y\times \Re^4$,
where we identify $T^*(\Re\times \Re)=\Re^4$ in the
 same way as above: a point $(t_1,t_2,k_1,k_2)\in \Re^4$ corresponds
to a 1-form $k_1dt_1+k_2dt_2$ at the point $(t_1,t_2)\in \Re^2$.

Let us estimate the microsupport of $F\conv K:=p_{13!}(p_{12}^{-1}F\otimes p_{23}^{-1}K)$,
where $p_{ij}$ are the same as in (\ref{pij}) with $X_1=\pt$; $X_2=X$; $X_3=Y$.
We have $p_{12}^{-1}F$ is microsupported within the set $S_F$
 consisting
of all points of the form
$$
(k_1\omega_1,0_y,t_1,t_2,k_1,0),
$$
where $0_y\in T^*_YY$,
 $(x,\omega_1) \in A$; $k_1\geq 0$ (as follows from Lemma \ref{first}). Analogously,
The sheaf $p_{23}^{-1}K$ is microsupported on the set $S_K$ consisting
of all points of the form
$$
(k_2\omega_2,k_2\eta_2,t_1,t_2,0,k_2),
$$
where $k_2\geq 0$, $(\omega_2,\eta_2)\in B$.

One sees that $S_K\cap -S_F\subset 
T^*_{X\times Y\times \Re\times \Re}(X\times Y\times \Re\times \Re)
$. Therefore, $p_{12}^{-1}F\otimes p_{23}^{-1}K$ is microsupported
within the set
of all points of the form
$$
(k_1\omega_1+k_2\omega_2,k_2\eta_2,t_1,t_2,k_1,k_2),
$$
where $k_1,k_2\geq 0$; $\omega_1\in A$; $(\omega_2,\eta_2)
\in B$.

Let $Q:X\times Y\times \Re\times \Re\to Y\times \Re\times \Re$,
$a:Y\times \Re\times \Re\to Y\times \Re$ be given by
$$
Q(x,y,t_1,t_2)=(y,t_1,t_2);
$$
$$
a(y,t_1,t_2)=(y,t_1+t_2)
$$
so that $p_{13}=aQ$.

We see that the map $Q$ is proper
on the support of 
$p_{12}^{-1}F\otimes p_{23}^{-1}K$.
It then follows that 
the sheaf $\Psi:=RQ_!(p_{12}^{-1}F\otimes p_{23}^{-1}K)$
is microsupported on the set $S_Q$ of all points 
$$
(k_2\eta_2,t_1,t_2,k_1,k_2)
$$
such that $k_1,k_2\geq 0$ and there exist  $\omega_1\in A$,
$(\omega_2,\eta_2)\in B$ such that $\omega_1$ and $\omega_2$ are
in the same fiber of $T^*X$ and
$k_1\omega_1+k_2\omega_2=0$ 

Let us now  estimate the mircosupport of $Ra_!\Psi$. We will use
Lemma \ref{impropersupport}.  Let us use an isomorphism
$I:Y\times \Re\times \Re\to Y\times \Re\times \Re$, where
$$
I(y,t_1,t_2)=(y,t_1+t_2;t_2).
$$
Let $p_2:Y\times \Re\times \Re\to Y\times \Re$ be given by
$p_2(y,t_1,t_2)=(y,t_1)$ so that we have
$$
a=p_2I
$$
and $Ra_!\Psi=Rp_{2!}I_!\Psi$. 
We see that the sheaf $I_!\Psi$ is microsupported
on the set  $\Gamma_1$ consisting of all points of the form
$$
(k_2\eta_2,t_1,t_2,k_1,k_2-k_1)
$$
such that $k_1,k_2\geq 0$ and there exist $\omega_1\in A$,
$(\omega_2,\eta_2)\in B$ such that $\omega_1$ and $\omega_2$ are 
in the same fiber of $T^*X$ and
$k_1\omega_1+k_2\omega_2=0$.
  Let us now use Lemma \ref{impropersupport} in order
to estimate $\mS Rp_{2!}I_!\Psi$.
Let $\eta\in T^*Y$; $\eta\notin C$. We need to show that
$Rp_{2!}I_!\Psi$ is non-singular at any point of the form
$$
(\eta,t,1)\in T^*Y\times \Re\times \Re=T^*Y\times T^*\Re.
$$
 Assuming the contrary, for any $\delta>0$
 there should exist a point 
$(k_2\eta_2,T_1,T_2,k_1,k_2-k_1)\in \Gamma_1$ such that
$|\eta-k_2\eta_2|<\delta$ and $|k_1-1|,|k_2-k_1|<\delta$.  Given $\ve>0$, one can choose
$\delta>0$ such that under the conditions specified,
 $|1-k_1/k_2|<\ve$.
 Let $A_\ve=[1-\ve,1+\ve].A$'.
We then see that there should exist $\omega_2\in T^*X$ such that
 $(\omega_2,\eta_2)\in B$
and  $-\omega_2\in A_\ve$ (because $-\omega_2=k_1/k_2\omega_1$
and $\omega_1\in A$). Thus, $\eta_2\in A_{\ve}\conv B$.
We see that the sets $A_{1/n}$, $n=1,2,\ldots$ are compact
and $\bigcap_n A_{1/n}=A$. Let $U\supset C$ be an open neighborhood.

 By Lemma \ref{podzhim}, there exists an 
$N$ such that $A_{1/N}\conv B\subset U$ i.e. for all  $\ve\leq 1/N$
we have $\eta_2\in U$. Taking into account the inequality $|\eta-k_2\eta_2|<\delta$ and letting
$\delta$ arbitrarily small, we see that $\eta\in U$. As $U$
is any open neighborhood of $C$, we conclude $\eta\in C$.
 We  get a contradiction.
\end{proof}

\subsubsection{}\label{reduct} If $\Phi=\Phi_1\Phi_2\cdots \Phi_N$ and the statement of the Theorem is true for each $\Phi_k$, it is true for $\Phi$.
In other words, if $Z$ is the set of Hamiltonian symplectomorphisms
of $T^*X$ which are identity outside of a compact and if $Z$ generates the whole group of Hamiltonian symplectomorphisms of $T^*X$ which are
idenitity outside a compact, then it suffices to prove Theorem 
for all $\Phi\in Z$. 

Let us now choose an appopriate $Z$.  Call a symplectomorphism 
$\Phi:X\to X$ {\em small}
if

1) There exists a Darboux chart $U\subset T^*X$ with Darboux 
coordinates $x,P$, where $x^i$ are local coordinates on $x$,
 $P^i=\partial/\partial x^i$, $|x^i|<1$ and for some fixed 
\begin{equation}\label{refpi}
\pi\in \Re^n,
\end{equation}

 $|P^i-\pi^i|<1$ for all $i$. Let $p^i:=P^i-\pi^i$.
For $x\in \Re^n$ we set $|x|:=\text{max}_i |x_i|$.
We demand that  $\Phi$ should be identity
outside  a subset $V\subset U$, $|x|<1/2,|p|<1/2$.
2)  let $(x',p')=\Phi(x,p)$.  Then $(x,p')$ form a non-degenerate
coordinate system on $U$ so that $(x,p')$ map $U$ diffeomorphically
onto a domain $W\subset \Re^{2n}$. 

It is well known that the set $Z$ formed by small symplectomorphisms
satisfies the conditions. 

\subsubsection{Small symplectomorphisms in terms of generating functions}
\def\oPhi{\overline{\Phi}}
The coordinates $(x,p)$  define an embedding $U\subset \Re^{2n}$. 
Let us extend $\Phi|_U$ to a map $\oPhi:\Re^{2n}\to \Re^{2n}$ by
setting $\oPhi(x,p_0)=(x,p_0)$ for all $(x,p_0)\notin U$. We  see
that $\oPhi$ is a diffeomorphism because it maps
$U$ diffeomorphically to itself, as well as the complement
to $U$. Hence, $\oPhi:\Re^{2n}\to \Re^{2n}$ is a symplectomorphism
with respect to the standard symplectic structure. 

As above let $\oPhi(x,p)=(x',p')$. Let $\Psi:\Re^{2n}\to \Re^{2n}$
where $\Psi(x,p)=(x,p')$. 

\begin{Lemma}\label{diffeopsi}
 $\Psi$ is a diffeomorphism.
\end{Lemma}
\begin{proof}
a) $\Psi$ has a non-zero Jacobian everywhere.
Indeed, if $|x|<1,|p|<1$ this is postulated by 2); otherwise
$\Psi=\Id$ in a neighborhood of $(x,p)$.

b) $\Phi$ is an injection. 
Suppose $\Psi(x_1,p_1)=\Psi(x_2,p_2)$. Then $x_1=x_2=x$ and
$p'(x,p_1)=p'(x,p_2)$. Consider several cases:

1)    $|x|\geq 1$, then $p'(x,p_1)=p_1$; $p'(x,p_2)=p_2$ and
$p_1=p_2$; 

2) $|x|<1$; $|p_1|< 1$.  If $|p_2| <1$, then
$p_1=p_2$ by Condition 2). If $|p_2|\geq 1$, then
$p'(x,p_2)=p_2$; $|p'(x,p_2)|\geq 1$ and $|p'(x,p_1)|<1$ 
because $\oPhi$ preserves $U$, so $p'(x_1,p_1)\neq p'(x_2,p_2)$;
 
3) $|x|<1$ and $|p_2|<1$--- similar to 2);

4) $|x|<1$ and  $|p_1|,|p_2|=1$. Then 
$p'(x,p_i)=p_i$, therefore $p_1=p_2$.

c) $\Psi$ is surjective.  We know that $\Psi(x,p)=(x,p)$ if $|x|>1$ or $|p|>1$.  Assume that,
on the contrary,
$(x_0,p_0)$ does not belong to the image of $\Psi$.  It follows that $|x_0|<1$; $|p_0|<1$. 
For $R>0$ consider the sphere $S_R$ given by the equation $\sum_i (x^i)^2+\sum_i(p^i)^2=R^2.$ 
Choose $R$  so large  that $(x,p)\in S_R$ implies $|x|>1$ or $|p|>1$.  We then  have
$\Psi|_{S_R}=\Id$.  It also follows $S_R$ cannot be homotopized to a point in
$\Re^{2n}\backslash (x_0,p_0)$ (because $(x_0,p_0)$ is  inside the open ball bounded by $S_R$).
 On the other hand it can: Let $\gamma:S_R\times [0,1]\to \Re^{2n}$
be any homotopy which contracts $S_R$ to a point. Then $\Psi\circ \gamma$ is a required homotopy.
This is a contradiction.
\end{proof}

\begin{Lemma}\label{generating} There exists a   smooth function $S(x,p')$ on $\Re^{2n}$ such that

1) $(x',p')=\oPhi(x,p)$ iff for all $i$:
$$
p^i=(p')^i+\frac{\partial S}{\partial x^i};
$$
$$
(x')^i=x^i+\frac{\partial S}{\partial (p')^i};
$$

2) $S=0$ if $|x|\geq 1/2$ or $|p'|\geq 1/2$;

$$
3)\quad \max_{|x|\leq 1/2,|p'|\leq 1/2} |x^i+\partial S/\partial (p')^i|\leq 1/2
$$

\end{Lemma}
\begin{proof}
 Consider
the following 1-form  on $\Re^{2n}$:
$
\sum p^idx^i +\sum (x')^i d(p')^i.
$
This form is closed, hence  exact. So one can write
$$
\sum p^idx^i +\sum (x')^id(p')^i=d(S(x,p')+<x,p'>)
$$
by virtue of Lemma \ref{diffeopsi}.
This equation is equivalent to the part 1) of this Lemma.

We know that   $\oPhi=\Id$ if $|x|\geq 1/2$ or $|p|\geq 1/2$.  
Therefore,
$\oPhi$ , being bijective, preserves the region 
$\{(x,p)| |x|,|p|< 1/2\}$. Therefore, if
$|p'(x,p)|\geq 1/2$, then  either $|x|\geq 1/2$ or $|p|\geq 1/2$,
 hence $p'(x,p)=p$; $x'(x,p)=x$ and $dS(x,p')=0$ as soon
as $|x|\geq 1/2$ or $|p'|\geq 1/2$. As the specified region
is connected, $S$ is a constant in this region, and one can
choose $S$ to be 0 as long as $|x|\geq 1/2$ or $|p'|\geq 1/2$. This proves 2).

It also follows that if $|x|\leq 1/2$ and $|p'(x,p)|\leq 1/2$ then 
$|p|\leq 1/2$, because otherwise
$\Phi(x,p)=(x,p)$ and $p'=p$, which is a contradiction.
 This implies 3).
\end{proof} 

\subsubsection{}
Let $J$ be the set of all smooth functions $S(x,p')$ on $\Re^{2n}$
such that $S$ is supported on the set 
$\{(x,p')| |x|\leq 1/2, |p'|\leq 1/2\}$ and  the inequality 3) from Lemma \ref{generating} is satisfied.
Our ultimate goal is: given such an $S$, we would like to construct  certain kernels in $\cD(\Re^n\times \Re^n)$
and then  $\cD(X\times X)$.

Let $\pi\in \Re^n$ (this parameter has the same meaning as in (\ref{refpi}).  Let
$S\in J$.  We will start with constructing an appropriate object $\Lambda_{S,\pi}\in \cD(\Re^n\times\Re^n)$
and estimating its microsupport.

 Let $\Sigma_\pi(x_1,x_2,p'):=-S(x_1,p')-<x_1-x_2,p'+\pi>$.
We can decompose $$
d\Sigma_\pi=d_{x_1}\Sigma_\pi+d_{x_2}\Sigma_\pi+d_{p'}\Sigma_\pi.
$$ 

Let $\Gamma_\pi(S)\subset T^*\Re^n\times T^*\Re^n$ consist
of all points $(x_1,p_1,x_2,p_2)$ satisfying:
there exists $p'$ such that $d_{p'}\Sigma_\pi(x_1,x_2,p')=0$ and
$p_1=d_{x_1}\Sigma_\pi(x_1,x_2,p')$; $p_2=d_{x_2}\Sigma_\pi(x_1,x_2,p')$.

{\bf Remark.} Let us take $S$ as in Lemma \ref{generating}.  The set $\Gamma_\pi(S)$ then consists
of all points $(x_1,P_1,x_2,P_2)$ such that $\oPhi(x_1;-P_1-\pi)=(x_2,P_2-\pi)$. That is,
if $|P_1+\pi|<1$, then $(x_2,P_2)=\Phi(x_1,-P_1)$; if $|P_1+\pi|\geq 1$, then $x_2=x_1$, $P_2=-P_1$, where we use 
notation from Sec \ref{reduct}.

We are now passing to constructing an object $\Lambda_{S,\pi}\in \cD(\Gamma_\pi(S))$.
Consider the following subset
$C_{S,\pi}\subset \Re^n\times \Re^n\times \Re^n\times \Re$;
$$
\{(x_1,x_2,p',t,)| t+\Sigma_\pi\geq 0\},
$$

Let $q:\Re^n\times \Re^n\times \Re^n\times \Re\to 
\Re^n\times \Re^n\times \Re$ be given by
$$
q(x_1,x_2,p,t)=(x_1,x_2,t).
$$

Set $\Lambda_{S,\pi}:=Rq_! \gf_{C_{S,\pi}}$. 

\begin{Lemma}\label{nositelL} Assume $S\in J$. Then $\Lambda_{S,\pi}\in \cD_{\Gamma_\pi(S)}(\Re^n\times \Re^n)$.
\end{Lemma}
\begin{proof}
It is straightforward to check that $\Lambda_{S,\pi}$ is in the 
left orthogonal complement to $C_{\leq 0}(\Re^n\times \Re^n)$.

Let us now estimate the microsupport of $\Lambda_{S,\pi}$.  
Let us choose a large positive number $C$ and consider objects
$$
F_C:=Rq_! \gf_{\{(x_1,x_2,p',t)| t+\Sigma_\pi(x_1,x_2,p')\geq 0; |p'|< C\}}
$$
so that $\Lambda_{S,\pi}=L\limdir_{C\to \infty} F_C$.

We will prove: {\em
let $(x_1,x_2,t,\omega_1,\omega_2,k)\in T^*( \Re^n\times \Re^n\times \Re)$
be a singular point of $F_C$. Then  one of the following 3 statements
is true:

--- $k=\omega_1=\omega_2=0$;

--- $k>0$ and $|\omega_i/k|\geq C-|\pi|$, $i=1,2$.

---$(x_1,x_2,\omega_1,\omega_2)\in \Gamma_\pi(S)$}

This implies the Lemma as $C$ can be chosen arbitrarily large.

Let us estimate the microsupport of the sheaf
$$
\gf_{\{(x_1,x_2,p',t)| t+\Sigma_\pi(x_1,x_2,p')\geq 0; |p'|<C
\}}
$$

we see that it is contained within the set of all points of the form
$$
(x_1,x_2,p',t,kd\Sigma_\pi(x_1,x_2,p')+\sum a_i d(p')^i),
$$
where $|p'|\leq C$ and 
 $a_i\leq 0$ if $(p')^i=-C$; $a_i=0$ if $|(p')^i|<C$, and $a_i\geq 0$ if
$(p')^i=C$; also, $k\geq 0$ and if $k>0$, then $t+\Sigma_\pi(x_1,x_2,p')=0$.

Let us now estimate the singular support of the sheaf
$$
Rq_!\gf_{\{(x,x',p,t)| t+\Sigma_\pi(x_1,x_2,p')\geq 0; |p'|< C\}}.
$$
As  $q$ is proper on the support of this sheaf,
 we see
that
$$
Rq_!\gf_{\{(x,x',p,t)| t+\Sigma_\pi(x_1,x_2,p'); |p'|< C\}}
$$
is microsupported on the set of points
$$
(x_1,x_2,t,\omega_1,\omega_2,k),
$$
where there exists $p', |p'|\leq C$ such that
 
\begin{equation}\label{varomega}
\omega_i= kd_{x_i}\Sigma_\pi(x_1,x_2,p')
\end{equation}
where $k\geq 0$ and if $k>0$ then there exists  $p'$ such that
$$
\frac{\partial \Sigma_{\pi}}{\partial (p')^i}(x_1,x_2,p')\geq 0
\text{ if  } (p')^i=-C;$$

\begin{equation}\label{varpshtrih}
\frac{\partial \Sigma_{\pi}}{\partial (p')^i}(x_1,x_2,p')= 0
\text{ if  } |(p')^i|<C;
\end{equation}

$$
\frac{\partial \Sigma_{\pi}}{\partial (p')^i}(x_1,x_2,p')\leq 0
\text{ if  } (p')^i=C.
$$

Let us first consider the case $C>1/2$, $|p'|=C$, and $k>0$.
Observe that if
 $|p'|>1/2$, then $S(x,p')=0$; $\Sigma_\pi=-<x_1-x_2,p'+\pi>$. Eq. (\ref{varomega}) then implies:
If $C>1/2$ and $|p'|=C$, then 
$\omega_1=-k(p'+\pi)$; $\omega_2=k(p'+\pi)$.
Hence:  if  $k>0$ and $|p|=C$, then $|\omega_1|/k,|\omega_2|/k\geq
C-|\pi|$. 

If $k>0$ and $|p|<C$, then $(x_1,x_2,\omega_1,\omega_2)\in \Gamma_\pi(S)$ by (\ref{varomega})
and (\ref{varpshtrih}).

  If $k=0$, then $\omega_1=\omega_2=0$.  Finally, $k$ is always
non-negative. This proves the statement.
\end{proof}

\subsubsection{} 
 Let $A,B\subset \Re^n\times \Re^n\times\Re$ be the following open subsets:
$$
A=\{(x_1,x_2,t')| |x_1|>1/2 \}
$$
$$
B=\{(x_1,x_2,t')| |x_1|<3/5; |x_2|>4/5 \}
$$

\begin{Lemma} For every $S\in J$ we have:
1) $\Lambda_{S,\pi}|_A \cong \gf_{\{x_1=x_2;t\geq 0\}}[-n]$;

2) $\Lambda_{S,\pi}|_B=0$.
\end{Lemma}
\begin{proof}

1) We have $S(x_1,p')=0$ for all $x_1>1/2$. Therefore, 
$$
\Lambda_S|_A=R\pi_! \gf_{\{t-<x_1-x_2,p'+\pi>\geq 0\}}\cong 
\gf_{\{x_1=x_2;t\geq 0\}}[-n].
$$
The last isomorhpism has been established in Sec. \ref{Fourier}.

2) Let  $|x_1|<3/5,$ $|x_2|>4/5$,  and consider
the equation
$$
\partial_{p'}\Sigma_\pi(x_1,x_2,p')=0.
$$

We have
$$
\partial_{p'} (-S(x_1,p')-<x_1-x_2,p'+\pi>)=-x_1-\partial_{p'}S(x_1,p')+x_2= x_2-y,
$$
where 
$$
y=x_1+\partial_{p'}S(x_1,p')
$$

if $|p'|\leq 1/2$ then $|y|\leq 1/2$ as $S\in J$. If $|p'|\geq 1/2$, then $y=x$ and $|y|<3/5$. Thus, in any case $|y|<3/5$, therefore,$x_2-y\neq 0$ because $|x_2|>4/5$.

Thus for all $p'$, $$
\partial_{p'} \Sigma_\pi(x_1,x_2,p')\neq 0.
$$

2)Fix $(x_1,x_2)\in B$. Set $G(p'):=\Sigma_\pi(x_1,x_2,p')$.  We know that $dG(p')\neq 0$ for all $p'$.
 For $|p'|>1/2$, $G(p')=-<x_1-x_2,p'+\pi>=<c,p'>+K$ for some constants $c\neq 0$ and  $K$.

We need to show that for any such a function $G$, 
$Rq_!\gf_{\{(t,p):t+G(p)\geq 0\}}=0$ where $q:\Re^n\times \Re\to \Re$ is the projection.

Let $Y\subset \Re^n$ be the hyperplane $<c,p>+K=-M$ for $M>>0$.
Let $F_\tau$ be the flow of the gradient vector field of $G$.
We then get a map
$$
\Gamma:Y\times \Re\to \Re^n,
$$
$$
\Gamma(y,\tau)=F_\tau(y)
$$
The map $\Gamma$ is clearly a diffeomorphism and $G(\Gamma(y,\tau))
=\tau-M$.  Thus, under diffeomorphism $\Gamma$, the function
$G(p')$ gets transformed into $\tau-M$. Therefore it suffices
to show the statement for $G$ being a linear function on 
$\Re^n$, in which case the statement is clear.
\end{proof}
\subsubsection{} Using the above Lemma we will now construct a kernel in $\cD(X\times X)$
where $X$ is as in Sec. \ref{reduct}.
Observe that $A\cup B$ contains the set 
$C:=\{(x,x',t)| |x|>4/5 \text{ or } |x'|>4/5\}$ and the above Lemma
implies that 
$\Lambda_{S,\pi}|_C\cong \gf_{\{x=x',t\geq 0\}}$.

 Recall (Sec \ref{reduct}) that we have a Darboux chart $U\subset T^*X$. Let $U_1$ be the projection
of $U$ onto $X$. $U_1$ is identified with a cube $|x|<1$ in $\Re^n$. 
Let $V\subset U_1\subset  X$ be given by the equation $|x|<1$ and $K\subset V$
by the equation $|x|<4/5$.
We then 
have a sheaf  $\Lambda_{S,\pi}|_{V\times V\times \Re}$ and a compact
$K\subset V$ such that on $W:=V\times V\times \Re\backslash
(K\times  K\times \Re)$ we 
have an identification $\Lambda_{S,\pi}|_W=\gf_{\{(x_1,x_2,t)\in W| x_1=x_2; t\geq 0\}}[-n]$

One can now extend $\Lambda_{S,\pi}$ to a sheaf on $X\times\times X\times \Re$
by setting $L_S|_{X\times \Re\times X\times \Re\backslash W}=
\gf_{\{(x,x',t)|x=x';t\geq 0\}}$.
Denote thus obtained sheaf by $L_S$.  Let 
$\Gamma_\Phi=\{(-\omega,\Phi(\omega)\}\subset T^*X\times T^*X$.

\begin{Proposition}\label{nositel} We have $L_S\in \cD_{\Gamma_\Phi}(X\times X)$.
\end{Proposition}
\begin{proof} Follows easily from Lemma \ref{nositelL} and Remark before this Lemma.
\end{proof}

\subsubsection{} Let $S_+(x,p)$ be a function on $\Re^{2n}$ defined
as follows: if $S(x,p)\leq 0$, then we set $S_+(x,p)=0$; if
$S(x,p)\geq 0$, then we set $S_+(x,p)=S(x,p)$.

\begin{Lemma} For every $S\in J$ and any $\pi\in \Re^n$ we have:
1) $\Lambda_{S_+,\pi}|_A \cong \gf_{\{x_1=x_2;t\geq 0\}}[-n]$;

2) $\Lambda_{S_+,\pi}|_B=0$.
\end{Lemma}
\begin{proof} There exists a sequence of smooth functions $g_n(x)$ on $\Re$ with the following properties:
1) each function $g_n(x)$ is non-decreasing;
furthermore, $0\leq g'_n(x)\leq 1$ for all $n$ and $x$;

2) for every $x$, the sequence $g_n(x)$ is non-decreasing;

3) for $x\leq 0$, $g_n(x)=0$;

4) for $x\geq 1/n$, $g'_n(x)=1$.

Fix such a sequence of functions. 

For $S\in J$ consider functions 
$S_n(x,p)=g_n(S(x,p))$. Let us check that $S_n\in J$. 
Indeed, $S_n$ are supported on the set $|x|\leq 1/2$, $|p|\leq 1/2$
because $g_n(0)=0$.  Next, we have
$|x^i+\partial S/\partial p^i|\leq 1/2$ for all $x$ with $|x|\leq 1/2$, i.e
$$
\partial S/\partial p^i\in [-x_i-1/2;-x_i+1/2]
$$
The interval on the RHS contains zero, therefore is closed under
multiplication by any number $\lambda\in [0,1]$. 

We have $$\partial S_n/\partial p^i= g'_n(S)\partial S/\partial p^i\in 
[-x_i-1/2;-x_i+1/2]$$  precisely because $0\leq g'_n<1$.
Thus, $S_n\in J$.

Next, we see that  $S_1(x)\leq S_2(x)\leq \cdots \leq S_n(x)\leq \cdots $ and that
$S_n(x)$ converges uniformly to $S_+(x)$.
It then follows that we have induced
maps $\Lambda_{S_1,\pi}\to \Lambda_{S_2,\pi}\to \cdots \Lambda_{S_n,\pi}\to\cdots$ and we have  an isomoprhism
$$
L\limdir_n \Lambda_{S_n,\pi}\to \Lambda_{S_+,\pi}.
$$

Since the sheaves $\Lambda_{S_n,\pi}$ satisfy the Lemma,
so does $\Lambda_{S_+,\pi}$.
\end{proof}

This implies that in the same way as above, $\Lambda_{S_+,\pi}$
can be extended to $X\times X\times \Re$ in the same way 
as $\Lambda_{S,\pi}$ and we denote thus obtained sheaf by $L_{S_+,\pi}$.

\subsubsection{Proof of the Theorem \ref{hamilt}} We will prove
an equivalent statement as in Sec \ref{reduct}

 Define a functor
$T:\cD(X\times \Re)\to \cD(X\times \Re)$ by setting
$T(F)=F\conv L_S$ (see Sec. \ref{kernelconv}).  Because of Lemma \ref{nositel} and Proposition \ref{svert} we see that
if $F\in \cD_L(X)$, then $TF\in \cD_{\Phi(L)}(X)$.

Next, we have natural maps
$$
L_{S,\pi} \stackrel{i}\to L_{S_+,\pi}\stackrel{j}\leftarrow L_{0,\pi}
$$
Note that $L_{0,\pi}=\gf_{\{(x_1,x_2,t)|x_1=x_2,t\geq 0\}}$.
In order to finish the proof of the theorem, it suffices to show
that the cones of the induced maps
$F\conv L_{S,\pi}\to F\conv L_{S_+,\pi}$ and $F=F\conv L_{0,\pi}\to F\conv L_{S_+,\pi}$ are torsion sheaves for
all $F\in \cD(X)$. This easily follows from the fact that
the cones of the maps $L_{S,\pi}\to L_{S_+,\pi}$ and $L_{0,\pi}\to L_{S_+,\pi}$ are
torsion objects in $\cD(X\times X\times \Re)$.
This fact can be seen from the following:
 each of the cones in question is supported on the set
$\{(x_1,x_2,t)| m\leq t\leq M\}$ where $m$ is the minimum of $S$
and $M$ is the maximum of $S$. Any sheaf $G$ with such a property
is necessarily torsion, because the supports of $G$ and $T_{c*}G$ are disjoint for $c>>0$ and $R\hom(G,T_{c*}G)=0$. This proves theorem
\ref{hamilt}.
\subsubsection{Proof of Theorem \ref{main}}
Let $F_1,F_2\in \cD(X)$ and let $f:F_1\to F_2$. Call $f$ 
{\em an isomorphism up-to torsion}
if the  cone  of $f$ is a torsion object. Call $F_1$ and $F_2$ {\em isomorphic up-to torsion}
if they can be connected by a chain of isomorphisms up-to torsion.

It is easy to see that if $F_1$ and $F_2$ are isomorphic up-to
torsion and for some $G\in \cD(X)$,  the natural map
 $R\hom(G,F_1)\to R\hom(G,T_{c*}F_1)$ is zero for some $c>0$, then
the map $R\hom(G,F_2)\to R\hom(G,T_{d*}F_2)$ is zero for some $d>0$.

Suppose $L_1$ and $L_2$ are displaceable compact Lagrangians
in $T^*X$, i.e.
for some sympectomorphism $\Phi$ of $T^*X$ such that $\Phi$ is identity outside of a compact, we have $L_1\cap \Phi(L_2)=\emptyset$.
Let $F_i \in \cD_{L_i}(X)$. Theorem \ref{main} is equivalent to
 the statement:  {\em for some $c>0$, the natural
map $R\hom(F_1,F_2)\to R\hom(F_1,T_{c*}F_2)$ is zero. }

This statement can be proven as follows.
By Theorem  \ref{hamilt}, there exists an object
$F_3\in \cD_{\Phi(L_2)}(X)$ such that $F_3$ and $F_2$ are isomorphic up-to torsion. Therefore, it suffices to
show that the natural map
$$
R\hom(F_1,F_3)\to R\hom(F_1,T_{c*}F_3)
$$
is zero for some $c>0$. But Theorem \ref{disjoint} asserts
that $R\hom(F_1,F_3)=R\hom(F_1,T_cF_3)=0$, whence the statement.

\section{Non-dispaceability of certain Lagrangian
submanifolds  in $\CP^n$}
\label{cepeen}
Consider $\CP^N$ with the standard symplectic structure.
 We have the following standard Lagrangian subvarieties in
$\CP^N$: the Clifford torus $\cT\subset \CP^N$ consisting of all
 points with homogeneous coordinates $(z_0:z_1:z_2:\cdots:z_N)$ such that
$|z_0|=|z_1|=\cdots=|z_N|>0$. Another Lagrangian subvariety we will
consider is $\RP^N\subset \CP^N$. Our main goal is to prove
\begin{Theorem}\label{main2} 1) $\cT$ is non-displaceable from itself;

2) $\RP^N$ is non-dispalceable from itself;

3) $\cT$ and $\RP^N$ are non-displaceable from one another.
\end{Theorem}
\subsubsection{}\label{svedenie}
Let us first of all explain how
Theorem \ref{main} can be applied. 

 Let $G=\SU(N)$ Realize $\CP^N$ as a coadjoint orbit $\CP^N=\O\subset \g^*$, where $\g=\su(N)$ is the Lie algebra of $G$.  We identify $\g$
with the real vector space of $N\times N$ skew-hermitian matrices.
We have an invariant positive definite inner product on $\g$ by the 
formula $<A,B>=-\Tr (AB)$.  This way we get an identification
$\g\cong \g^*$.

The orbit    $\O\subset \g^*\cong \g$ is an
orbit of the following diagonal skew-hermitian matrix
$$
i\lambda( P_V- (1/N) I)\in \g
$$
where $\lambda\in \Re$, $\lambda\neq 0$ is a fixed real number.
For simplicity we will only work with $\lambda>0$. However, the 
case $\lambda<0$ is absolutly similar.

Consider $T^*G$. We have a diffeomorphism
$I_R:T^*G\to G\times \g^*$ where we identify $\g^*$ with right-invariant forms on $G$.  Any element $X\in \g$ gives rise to a function
on $f_X$ on $\g^*$. We have a standard Poisson structure on $\g^*$
determined by the condition $\{f_X,f_Y\}=f_{[X,Y]}$. The canonical
projection $p_R:T^*G\stackrel{I_R}\to G\times \g^*\to  \g^*$ is then a Poisson map.

Let $\g^{op}$ be the Lie algebra whose underlying vector space is
$\g$ but $[X,Y]_{\g^\op}=-[X,Y]_\g$,
            We then have an identification 
$I_L:T^*G\to G\times (\g^\op)^*$,
where we identify $(\g^{\op})^*$ with left-invariant forms on $G$.
The composition $I_RI_L^{-1}:G\times(\g^\op)^*\to G\times \g^*$
is as follows: $I_RI_L^{-1}(g,A)=(g,\Ad^*_{g^{-1}}(A))$.

Indeed, the conjugate map $(I_RI_L^{-1})^*:G\times \g\to
G\times \g^\op$  is given by $(I_RI_L^{-1})^*(g,X)=(g,\Ad_{g^{-1}}X)$.

Respectively, $I_LI_R^{-1}: G\times \g^*\to G\times (\g^\op)^*$
is given by $I_LI_R^{-1}(g,A)=(g,\Ad^*_g A)$.

 One can easily check that the product
$p_L\times p_R:T^*G\to (\g^\op)^*\times \g^*$ is a Poisson map.

We know that $\O^\op\subset \g^*$ is a symplectic leaf, hence a co-isotropic sub-variety. Therefore, so is $M:=p_R^{-1}\O\subset T^*G$.

Let $\O^\op\subset (\g^*)^\op$ be the image of $\O\subset \g^*$
under the identification of vector spaces $\g^*=(\g^\op)^*$.

We then see that $M=p_L^{-1}\O^\op=p_R^{-1} \O$.  Indeed, we know 
that $I_LI_{R}^{-1}(g,A)=(g,\Ad_{g}A)$ and $A\in \O$ iff
$\Ad_{g}A\in \O$.

Hence, we have $M=(p_L\times p_R)^{-1}(\O^\op\times \O)$. 
Given any Poisson fibration $f:X\to Y$ and a coisotropic subvariety
$N\subset Y$, the subvariety  $f^{-1}N\subset X$
 is also co-isotropic. Let $n\in f^{-1}N$ and 
 let $V\in T_nf^{-1}N$ be a co-isotropic vector ( i.e $V=X_H$ where $H$ is a function in a neighborhood of $n$
and $H|_{f^{-1}N}=0$),  we then see that $f_*V\in T_{f(n)}N$ is also a 
co-isotropic vector.

Let us apply this observation to our case. We see that
$\O^\op\times \O$ has only zero co-isotropic vectors. 
Therefore, all co-isotropic vectors in $TM$ are tangent 
to fibers of the map $p_L\times p_R:M\to \O^{\op}\times\O$.
Comparison of dimensions shows that the inverse is also true:
co-isotropic vectors in $TM$ are precisely those tangent to the 
fibers of the map $p_L\times p_R$.  Thus, co-isotropic foliation
to $M$ is the tangent foliation to $p_L\times p_R$. We know
that this implies an induced symplectic structire on 
$\O^\op\times \O$. As the map $p_L\times p_R$ is Poisson, it
follows that the induced Poisson structure coincides with that
induced by the inclusion $\O^\op\times \O\into \g^\op\times \g$.
The corresponding symplectic 2 form is equal to $(-\omega;\omega)$
where $\omega$ is Kirillov's symplectic form on $\O$
and we use the identification of manifolds $\O^\op=\O$.

Let $I:M\to T^*G$ be the inclusion and $P=p_L\times p_R:
M\to \O^\op\times \O$. It then follows that $I^{*}\omega_{T^*G}=
P^{*}\omega_{\O^\op\times \O}$. 

It  follows that if $L\subset \O^\op\times \O$ is a Lagrangian
manifold, then so is $IP^{-1}L\subset T^*G$. 

Another important observation: let $H$ be a function on $\O^\op\times \O$ and let $H'$ be a function on $T^*G$ such that
$H'|_M=P^{-1}H$. 

1) Then the Hamiltonian vector field $X_{H'}$ is tangent to $M$;

2) given any function $f$ on $\O^\op\times \O$ we have
$$
X_{H'}P^{-1}f=P^{-1}X_Hf.
$$

Let $e^{tX_{H'}}$ be the Hamiltonian flow of $H'$ and $e^{tX_H}$ 
the 
Hamiltonian flow of $H$. It then follows that for any point
$m\in M$, $Pe^{tX_{H'}}(m)=e^{tX_H}(P(m))$.

These observations imply:
\begin{Proposition} Let $L_1,L_2\subset \O^\op\times \O$
be subsets  such that $IP^{-1}L_1,IP^{-1}L_2\subset T^*G$
are non-displaceable. Then so are $L_1,L_2$.
\end{Proposition}
\begin{proof} Suppose $L_1$ and $L_2$ are displaceable. Then
there exist functions $H_1,\ldots, H_k$ on $\O^\op\times \O$
such that $e^{X_{H_1}}\cdots e^{X_{H_k}}L_1\cap L_2=\emptyset$.
Choose compactly supported functions $H'_1,\ldots,H'_k$ on
$T^*G$ such that $H'_i|_M=P^{-1}H_i$. One then has
$$
Pe^{X_{H'_1}}\cdots e^{X_{H'_k}}m=e^{X_{H_1}}\cdots e^{X_{H_k}}Pm
$$
for every $m\in M$. Therefore,
$$
IP^{-1}L_1\cap e^{X_{H'_1}}\cdots e^{X_{H'_m}}P^{-1}L_2=
\emptyset,
$$
i.e the Lagrangians $IP^{-1}L_1$ and $IP^{-1}L_2$ are displaceble,
whence the statement 
\end{proof}

let $\Delta\subset \O^\op\times\O$ be the diagonal. $\Delta$ is
clearly Lagranginan.

It then follows that Theorem \ref{main2}
follows from the following one:
\begin{Theorem}\label{main3}
 1) $IP^{-1}\Delta$ and $IP^{-1}\cT\times\cT$
are non-displaceble;

2) $IP^{-1}\Delta$ and $IP^{-1}\RP^N\times \RP^N$ are non-displaceable;

3) $IP^{-1}\RP^N\times \RP^N$ and $IP^{-1}\cT\times \cT$ are 
non-displaceable
\end{Theorem}
\subsubsection{}
We will prove Theorem \ref{main3} using Theorem \ref{main}.

Our main tool will be a certain object $u_\O\in\cD(G)$ which will be 
now introduced. 

We need a notation.
 Let $S\in \cD(G)$.
Let $F\in D(G)$. Let $m:G\times G\times \Re\to G\times \Re$ be
the map induced by the  product on $G$. Set
$F*_G S:=Rm_!(F\boxtimes S)$ (this is nothing else but a convolution). One can easily check that 
$F*_G S\in \cD(G)$ (use Proposition \ref{checkorthogonal}).

\begin{Proposition}\label{diagonal} There exists 
an object $u_\O\in \cD_{IP^{-1}\Delta}(G)$
with the following properties:

1) there exists a neigborhood of the unit $U\subset G$; $e\in U$ with the following property:

for every $g\in G$ and every object $F\in D(G)$ such that  $F$ is supported on $gU$ and $R\Gamma(G,F)=0$, the object
$F*_G u_\O$ is a torsion object;

2) The object $u_\O$ is not a torsion object.
\end{Proposition} 

The proof of this Proposition is rather long, so we will first
show how this Proposition (along with Theorem \ref{main})
implies Theorem \ref{main3}.

\subsubsection{} \begin{Lemma} Let $\h\subset \g$ be the standard
Cartan subalgebra consisting of the diagonal traceless skew-hermitian
 matrices. Let $\k:=\so(N)\subset \su(N)$. We then have
$\cT=(\g/\h)^*\cap \O$; $\RP^N=(\g/\k)^*\cap\O$.
\end{Lemma}
\begin{proof} The symplectomorphism $f:\CP^N\to \O$ is as follows. 
Given a line
$l\in \Co^N$ we  set $f(l):=i(\lambda P_L-\lambda/N I)$, where $\lambda>0$
is a fixed positive real number.
Let $v=(v_1,v_2,\ldots,v_N)\in l$; $v\neq 0$. We then have
$$
f(l)_{pq}=(i\lambda/|v|^2) v_p\overline{v_q}-
i\lambda/N \delta_{pq},
$$
where $\delta_{pq}$ is the Kronecker symbol. 

Thus, $f(l)\in \O\cap (\g/\h)^*$ iff $f(l)_{pp}=0$ for all $p$,
i.e. $|v_p|^2/|v|^2=1/N$, i.e. $|v_1|^2=|v_2|^2=\cdots |v_N|^2$,
i.e. $l\in \cT$.

Analogously, $f(l)\in (\g/\k)^*$ iff  $f(l)_{pq}\in i\Re$ for 
all $p,q$, i.e 
$
v_p\overline{v_q}\in \Re
$
for all $p,q$. Let $v_{p_0}\neq 0$. Then 
$v_{q}=t_q/\overline{v_{p_0}}
$for some $t_q\in \Re$ and for all $q$. Let $t=(t_1,t_2,\ldots,t_N)$ then $v=t/\overline{v_{p_0}}$ and $l\in \RP^N\subset \CP^N$.
The inverse can be easily checked as well.
\end{proof}

\begin{Proposition}\label{puchoknaorbite} Let $T\subset \SU(N)$ be the subgroup
of diagonal matrices and let $\SO(N)\subset \SU(N)$ be the subgroup
of special orthogonal matrices. 

We then have

1) $\gf_T*_G u_\O\in \cD_{IP^{-1}(\cT\times \cT)}(G)$;

2) $\gf_{\SO(N)}*_G u_\O\in \cD_{IP^{-1}(\RP^N\times \RP^N)}(G)$.
\end{Proposition}
\begin{proof}

Let us prove 1). First of all, one can easily check that
$\gf_T*_G u_\O\in \cD(G)$
 using Proposition \ref{checkorthogonal}.
It only remains to show that {\em $\gf_T*_G u_\O$ is microsupported on the set $\{(g,t,k\omega,k)| k\geq 0; \omega\in IP^{-1}\cT\times \cT\}$.}
We have $\gf_T*_G  u_\O=Rm_!(\gf_T\boxtimes u_\O)$,
where $m:G\times G\times \Re\to G\times \Re$ is induced by
the product on $G$. 
Let also $M:G\times G\to G$ be the product on $G$
Let $g_1,g_2\in G$. We then have  an induced map
$$
M_{g_1,g_2*}: T_{(g_1,g_2)}G\times G\to T_{g_1g_2}G
$$
Let 
$(g_1,g_2,X_1,X_2)\in G\times G\times \g\times \g=T(G\times G).$ 
One then has $M_{g_1,g_2*}(g_1,g_2,X_1,X_2)=X_1+\Ad_{g_1}X_2$. The dual map
$$
M^*_{g_1,g_2}:T^*_{g_1g_2}G\to T^*_{(g_1,g_2)}G\times G
$$
is as follows
$$
M^*_{g_1,g_2}(g_1g_2,\omega)=(g_1,g_2,\omega;\Ad^*_{g_1}\omega).
$$

Finally, the map
$$m^*_{g_1,g_2,t}: T^*_{(g_1g_2,t)}(G\times \Re)\to 
T^*_{(g_1,g_2,t)}(G\times G\times \Re)$$
is given by
\begin{equation}\label{mstar}
m^*_{g_1,g_2,t}
(g_1g_2,t,\omega,k)=(g_1,g_2,t,\omega,\Ad_{g_1}^*\omega,k).
\end{equation}

The map $m$ being proper, we know that the object  $Rm_!(\gf_T\boxtimes u_\O)$
is microsupported on the set of all
points of the form 
\begin{equation}\label{form}
(g_1g_2,t,\omega,k)
\end{equation} where
$$m^*_{g_1,g_2,t}(g_1g_2,t,\omega,k)\in  \mS(\gf_T\boxtimes u_\O),
$$
i.e
\begin{equation}\label{torus1}
(g_1,\omega)\in \mS(\gf_T);
\end{equation}
\begin{equation}\label{uuoo}
(g_2,t,\Ad_{g_1}^*\omega,k)\in \mS(u_\O).
\end{equation}
We have, 
\begin{equation}\label{mstorus}
\mS(\gf_T)\subset \{(g,\omega_1)\in G\times \g^*|
g\in T; \omega_1\in (\g/\t)^*\};
\end{equation}
\begin{equation}\label{msuo}
\mS(u_\O)\subset \{(g,t,k\omega_2,k)|k\geq 0;
 (g,\omega_2)\in IP^{-1}\Delta\},
\end{equation}
as follows from Lemma \ref{first}.  The condition $(g,\omega_2)\in IP^{-1}\Delta$ means
that $\omega_2\in \O$ and $P_L\omega_2=P_R\omega_2$,
i.e $\omega_2=\Ad^*_g\omega_2$.

Therefore$$
(g_1g_2,t,\omega,k)\in \mS Rm_!(\gf_T\boxtimes u_\O)
$$
only if  (compare (\ref{uuoo}) and (\ref{msuo})):
\begin{equation}\label{na1}
k\geq 0
\end{equation}
\begin{equation}\label{na2}
 \Ad_{g_1}^*\omega=k\omega_2,
\end{equation}
 where 
\begin{equation}\label{na5}
\omega_2\in \O
\end{equation}
and 
\begin{equation}\label{na6}
\Ad^*_{g_2}\omega_2=\omega_2.
\end{equation}
 We should also have
 (compare (\ref{torus1})
and (\ref{mstorus})):
\begin{equation}\label{na3}
g_1\in T
\end{equation}
and 
\begin{equation}\label{na4}
\omega\in (\g/\h)^*.
\end{equation}

Let us now show that  $(g_1g_2,\omega,k)$ is of the form
$
(g_1g_2,k\omega^1,k), 
$
where $k\geq 0$ and $(g_1g_2,\omega^1)\in IP^{-1}(\cT\times \cT)$.
The latter means that $(P_L\times P_R)(g_1g_2,\omega^1)\in
\O^\op\times \O$ i.e. 
{\em both $\omega^1$ and $\Ad^*_{g_1g_2}\omega^1$
belong to $\cT=\O\cap (\g/\h)^*$.}
We have $k\geq 0$ (see (\ref{na1}).
If $k=0$, then 
$\omega=\Ad_{g_1^{-1}}^*k\omega_2=0$ and $(g_1g_2,\omega,k)=(g_1g_2,0,0)$, the condition is fulfilled.

Let now $k>0$. We have $\omega =k\Ad_{g_1^{-1}}^*\omega_2$ 
(see (\ref{na2})) 
so that $\omega^1=\Ad_{g_1^{-1}}^*\omega_2$.

 As
$\omega_2\in \O$ (see (\ref{na5})), it follows that 
$\omega^1=\Ad_{g_1^{-1}}^*\omega_2\in \O.$  
We also have $\omega_2=\omega/k\in (\g/\h)^*$ (see (\ref{na4}).

Next, let us consider  $$\Ad^*_{g_1g_2}\omega^1=
\Ad^*_{g_1g_2}\Ad^*_{g_1^{-1}}\omega_2$$
$$
=\Ad^*_{g_2}\omega_2=\omega_2
$$
(the latter equality comes from (\ref{na6}), and we have already
shown that $\omega_2\in \O\cap (\g/\t)^*$.

This proves the statement 1). The statement 2) can be proven 
in precisely the same way.
\end{proof}

\subsection{} Our goal is to prove the following statements
\begin{Proposition}\label{torus} The object $\gf_T*_G u_\O\in \cD(G)$ is isomorphic up to torsion to the object $u_\O\otimes_\gf H^*(T,\gf)$.
\end{Proposition}
\begin{Proposition}\label{RPN} Suppose that $\gf$ is a field of characteristic
2.
 The object $\gf_{\SO(N)}*_G u_\O\in \cD(G)$ is isomorphic up-to
torsion to $u_\O\otimes_\gf H^*(\SO(N),\gf)$. 
\end{Proposition}
\subsubsection{}
These Propositions imply Theorem \ref{main3}.  Let $\gf$ have characteristic 2 and let each of objects $F_1$ and  $F_2$ 
be either $\gf_T*_G u_\O$
or $\gf_{\SO(N)}*_G u_\O$.  Taking into account Proposition
 \ref{puchoknaorbite} and Theorem \ref{main}, it suffices to show that
for any $c>0$, the induced map $R\hom(F_1,F_2)\to R\hom(F_1;T_{c*}F_2)$ does not vanish (for all choices of $F_1$ and $F_2$). By virtue of the  just formulated
 Propositions, this follows from $u_\O$ being non-torsion
which is promised in Proposition \ref{diagonal}.  Thus, Theorem
\ref{main3} is now reduced to  Propositions \ref{diagonal}, \ref{torus},
and \ref{RPN}. We will first deduced the last two Propositions from the
first one, and, finally, we will prove Proposition \ref{diagonal}.

\subsubsection{}

In order to prove  Propositions \ref{torus} and \ref{RPN} we need to develop
corollaries from 
 Proposition \ref{diagonal} 1).

Let $C_U$ be the full subcategory of $D(G)$ generated by all objects
$F$ as in Proposition \ref{diagonal} 1) and their finite extensions.

\begin{Lemma} let $Q:=[0,1]^{M}$, $M\geq  0$. Let 
$\pi:Q\to G$ be any continuous map. Let $F\in D(Q)$, $R\Gamma(Q,F)=0$.
Then $R\pi_!F\in C_U$.
\end{Lemma}
\begin{proof}  The case $M=0$ is obvious. Let $M>0$.
Let $Q_0:=[0,1/2]\times [0,1]^{M-1}$
and let $Q_1=[1/2,1]\times [0,1]^{M-1}$.

1) We will first prove that {\em $F$ can be obtained by a finite number
of extensions from objects $X_1,X_2,\ldots,X_m$, where each $X_i$ is supported 
on either $Q_0$ or $Q_1$ 
 and $R\Gamma(Q,X_i)=0$.} Call such objects and their
extensions {\em admissible}. Thus, we are to show that $F$ is admissible.

 Let
$I:=Q_1\cap Q_2$. Let $i_k:Q_k\to Q$ and $i:I\to Q$ be inclusions.
Realize $F$ as a complex of soft sheaves on $Q$
Let $F_k:=F|_{G_k}$ and $F_I:=F|_I$. Each of these objects is also 
a complex of soft sheaves.

 We then have an isomorphism
$$
 F\to \Cone (i_{1*}F_1\oplus i_{2*}F_2\to i_* F_I)
$$

Let $p_k:Q_k\to \pt$ and $p_I:I\to \pt$ be the natural projections.
Let $V_k:=p_{k*}F_k=p_{k!}F_k$; let $V_I:=p_{I*}F_I=p_{I!}F_I$.
$V_k$ and $V_I$ are just complexes of $\gf$-vector spaces.
We then have maps
$$
a_k:p_k^{-1}V_k\to F_k;a_I:\ p_I^{-1}V_I\to F_I;
$$
$$
b_k:i_{k*}p_k^{-1}V_k\to i_{I*} p_I^{-1}V_I
$$

We then have the following  commutative diagram of complexes of sheaves
$$\xymatrix{
i_{1*}F_1\oplus i_{2*}F_2\ar[r]& i_{I*}F_I\\
i_{1*}p_1^{-1}V_1\oplus i_{2*}p_2^{-1}V_2\ar[r]^{b_1\oplus b_2}
\ar[u]_{a_1\oplus a_2}&
 i_{I*}p_I^{-1}V_I\ar[u]}
$$

Let $\Phi$ be the total complex of this diagram. $\Phi$ can be obtained by successive extensions from the following objects
$$
\Cone( i_{k*}p_k^{-1}V_k\to i_{k*}F_k);
$$
$$
\Cone( i_{I*}p_I^{-1}V_I\to i_{I*}F_I);
$$
each of these objects is admissible. Hence $\Phi$ is admissible.

Next, we  have a natural map
$$
\Phi\to \Cone(i_{1*}p_1^{-1}V_1\oplus i_{2*}p_2^{-1}V_2\to
 i_{I*}p_I^{-1}V_I)
$$
The cone of this map is quasi-isomrophic to $F$. Thus, in order to
show that $F$ is admissible, it suffices to show that
$$\Cone(i_{1*}p_1^{-1}V_1\oplus i_{2*}p_2^{-1}V_2\to
 i_{I*}p_I^{-1}V_I)
$$
is admissible.

Let us study the arrow $i_{1*}p_1^{-1}V_1\oplus i_{2*}p_2^{-1}V_2\to
 i_{I*}p_I^{-1}V_I$. This arrow is induced by the natural maps
$V_1\to V_I$ and $V_2\to V_I$. The cone of the induced map
$f:V_1\oplus V_2\to V_I$ is quasi-isomorphic to $R\Gamma(Q,F)=0$. 
Therefore, $f$ is a quasi-isomorphism and we have an induced quasi-isomorphism
$$
\Cone(i_{1*}p_1^{-1}V_1\oplus i_{2*}p_2^{-1}V_2\to
 i_{I*}p_I^{-1}(V_1\oplus V_2))\to\Cone(i_{1*}p_1^{-1}V_1\oplus i_{2*}p_2^{-1}V_2\to
 i_{I*}p_I^{-1}V_I).
$$

The object on the left hand side is isomorphic to
$$
\Cone(i_{1*}p_1^{-1}V_1\to
 i_{I*}p_I^{-1}V_1)\oplus \Cone(i_{2*}p_2^{-1}V_2\to
 i_{I*}p_I^{-1}V_1).
$$
We see that this object is a direct sum of admissible objects,
hence is itself admissible, therefore the object
$$\Cone(i_{1*}p_1^{-1}V_1\oplus i_{2*}p_2^{-1}V_2\to
 i_{I*}p_I^{-1}V_I)$$ is also admissible, whence the statement.

2)Choose a positive integer $M$ and subdivide $Q$ into $2^M$ small cubes, denote these small
cubes by $\bfq_i$, $i=1,\ldots 2^M$. Call an object $X\in D(Q)$
{\em $M$-admissible} if either

a) $X$ is supported on one of $\bfq_i$ and $R\Gamma(Q,X)=0$
or
b) $X$ can be obtained from objects as in a) by a finite number
of extensions.

By repeatedly applying the statement from 1) we see that
{\em every object $F\in D(Q)$ such that $R\Gamma(Q,F)=0$ is
$M$-admissible.}

3) For $M$ large enough one has: for every $i$ there exists
$g_i\in G$ such that $\pi(\bfq_i)\subset g_iU$. This implies that
given any object  $X\in D(Q)$ supported on $\bfq_i$ and satisfying
$R\Gamma(Q,X)=0$, one has $R\pi_! X\in C_U$. Therefore, every
$M$-admissible object is in $C_U$, including $F$.
\end{proof}
\begin{Corollary}  Let $U$ be a neighborhood of unit in $G$ such
 that
$U$ is diffeomorphic to an open ball.  Then $C_U=C$, where $C\subset
D(G)$ is the full subcategory formed by  finite extensions 
of objects of the form $R\pi_!X$, where $\pi:Q\to G$, $X\in D(Q)$,
$R\Gamma(Q,X)=0$.
\end{Corollary}

\begin{Corollary} Let $F\in C$ and $X\in D(G)$. One then has $F*_G X\in C$; $X*_G F\in C$.
\end{Corollary}
\begin{proof}
Choose a small open  ball $U\in G$, $e\in U$, small means 
 that there 
exists another open ball $V\subset G$ such that $U\cdot U\subset V$.
 It is not hard to see that any $X\in D(G)$ can be realized as a finite extension of objects $X_i$, where each $X_i$ is supported on
$g_iU$ for some $U$. Without loss of generality, one then can assume
that $X=X_i$. Therefore, $X*_G F$ is supported on $g_iU^2\subset g_iV$.
One also sees that $R\Gamma(G,X*_G F)=R\Gamma(G,X)\otimes R\Gamma(G,F)=0$. Thus, $X*_G F\in C_V$. 

The case of $F*_G X$ can be proven in a similar way.
\end{proof}
\subsubsection{}
 Call a map $f:F\to H$ in $D(G)$ a $C$-isomorphism
if the cone of $f$ is in $C$. Call two objects $F,H\in D(G)$ {\em
$C$-isomorphic} if they can be joined by a chain of $C$-isomorphisms.
\begin{Corollary}\label{conv} if $F_1$ and $F_2$ are $C$-isomorphic and
$H_1$ and $H_2$ are $C$-isomorphic, then $F_1*_G H_1$ and $F_2*_G H_2$ are $C$-isomorphic
\end{Corollary}

\subsubsection{}
We have
\begin{Claim}\label{Ctors} If $F$ and $H$ are $C$-isomorphic, then
$F*_G u_\O$ and $H*_G u_\O$ are isomoprhic up-to torsion
\end{Claim}
\begin{proof}
Indeed, $C= C_U$, where $U$ is the same as in Proposition
\ref{diagonal}. The statement follows immediately from
part 1) of this Proposition.
\end{proof}

\subsection{Proof of Proposition \ref{torus}}

Let $S_k\subset \SU(N)$ be the one-parametric subgroup consisting
of all matrices of the form $\diag(1,1,\ldots,e^{i\phi};e^{-i\phi},1,\ldots,1)$, where $e^{i\phi}$ is at the $k$-th position.
We then have $T=S_1S_2\cdots S_{N-1}$;
$\gf_T=\gf_{S_1}*_G \gf_{S_2}*_G \cdots*_ G \gf_{S_{N-1}}$.

It is clear that the statement of Proposition follows from
\begin{Lemma}\label{Lm} For any $k$, $\gf_{S_k}$ is $C$-isomorphic to
$\gf_e\oplus \gf_e[-1]$
\end{Lemma}

Indeed,  Corollary \ref{conv}
will then imply that
$\gf_T$ is $C$-isomorphic to $(\gf_e\oplus \gf_e[-1])^{* N-1}=
\gf_e\otimes_{\gf} H^\bullet(T,\gf)$. Therefore,
by Claim \ref{Ctors}, the objects $\gf_T*u_\O$ and
$(\gf_e\otimes_\gf H^\bullet(T,\gf)))*u_\O=
u_\O\otimes_\gf H^\bullet(T,\gf)$ are isomorphic up-to torsion.

It now remains to prove Lemma
\subsubsection{Proof of Lemma \ref{Lm}}

As all subgroups $S_k$ are conjugated in $G$, it suffices
to prove Lemma for $S_1$.  One then has $S_1\subset \SU(2)\subset
\SU(N)$, where the embedding $\SU(2)\subset \SU(N)$ is induced
by the standard decomposition $\Co^N=\Co^2\oplus \Co^{N-2}$.
Let $U$ be an open neighborhood of unit in $\SU(N)$ and
let $U':=U\cap \SU(2)$. Let $\iota:\SU(2)\subset \SU(N)$ be the inclusion. It is clear that $i_*C_{U'}\subset C_U$, hence
if two objects $F_1,F_2\in D(\SU(2))$ are $C$-isomorphic, then
so are $i_*F_1$ and $i_*F_2$. Therefore, in order to prove
Lemma, it suffices to show that {\em $\gf_{S_1}$ and $\gf_e\oplus \gf_e[-1]$ viewed as objects of $D(\SU(2))$ are $C$-isomorphic.}

Let $B\subset \su(2)$ consist of all matrices of the form
$iM$, where $M$ is a Hermitian matrix whose eigenvalues
has absolute value of at most $\pi$. Let $B_\pi\subset B$ be the subset
of all matrices $iM$, where the eigenvalues of $M$ are precisely
$\pi$ and $-\pi$. It is clear the $B$ is diffeomorphic to
a 3-dimensional closed ball and $B_\pi\subset B$ is the boundary
2-sphere.

Let $I:[-\pi;\pi]\to B$ be given  by 
$I(\phi)=i\diag(\phi;{-\phi})$.

We then have a diagram
$$\xymatrix{
[-\pi;\pi]\ar[r]^{i_1}& B\ar[r]^{i_2}& \SU(2)\\
\{-\pi,\pi\}\ar[u]^{a_1}\ar[r]^{k_1}& B_\pi\ar[u]^{a_2}\ar[r]^{k_3}&
 \{-I\}\ar[u]^{a_3}}
$$
where $i_2$ is induced by the exponential map $\su(2)\to \SU(2)$;
$a_1,k_1,a_2,a_3$ are obvious inclusions; $k_3$ is the projection.
We then have 
\begin{equation}\label{for}
\gf_{S_1}=\Cone (R(i_2i_1)_!\gf_{[-\pi;-\pi]}\oplus a_{3!}\gf_{-I}\to 
R(i_2i_1a_1)_!\gf_{\{-\pi,\pi\}}).
\end{equation}
The arrow in this equation is induced by natural maps
$$
\alpha:R(i_2i_1)_!\gf_{[-\pi;-\pi]}\to R(i_2i_1a_1)_!\gf_{\{-\pi,\pi\}})
$$
and
$$
\beta:a_{3!}\gf_{-I}\to  R(i_2i_1a_1)_!\gf_{\{-\pi,\pi\}}=
R(a_3k_3k_1)_!\gf_{\{-\pi,\pi\}}
$$
where $\alpha$ is induced by the natural map
$$
\gf_{[-\pi;-\pi]}\to a_{1!}\gf_{\{-\pi,\pi\}}
$$
induced by the embedding $\{-\pi,\pi\}\subset [-\pi,\pi]$.

The map $\beta$ is induced by the natural map
$$
\gf_{-I}\to (k_3k_1)_!\gf_{\{-\pi,\pi\}}=(k_3k_1)_*(k_3k_1)^{-1}
\gf_{\{-\pi,\pi\}}.
$$
We have a $C$-isomorphism
$$
\gamma:Ri_{2!}\gf_B\to R(i_2i_1a_1)_!\gf_{[-\pi,\pi]}
$$
Therefore the object in (\ref{for})
is $C$-isomorhic to
\begin{equation}\label{vtor}
\Cone (Ri_{2!}\gf_{B}\oplus \gf_{-I}\stackrel{\alpha_1\oplus \beta}\to 
R(i_2i_1a_1)_!\gf_{\{-\pi,\pi\}})
\end{equation}
where
$\alpha_1=\alpha\gamma$.

The map $\alpha_1:Ri_{2!}\gf_B\to R(i_2i_1a_1)_!\gf_{\{-\pi,\pi\}}$
can be factored as
$$
Ri_{2!}\gf_B\to R(i_2a_2)_! \gf_{B_\pi}\to 
 R(i_2i_1a_1)_!\gf_{\{-\pi,\pi\}}.
$$

Observe that $B_\pi=\CP^1$ and that
 $ R(i_2a_2)_! \gf_{B_\pi}\cong H^*(B_\pi,\gf)\otimes_\gf \gf_{-I}$.
Next $R(i_2i_1a_1)_!\gf_{\{-\pi,\pi\}}=\gf_{-I}\oplus \gf_{-I}$.
The map $R(i_2a_2)_! \gf_{B_\pi}\to 
 R(i_2i_1a_1)_!\gf_{\{-\pi,\pi\}}$ factors as
$$
R(i_2a_2)_! \gf_{B_\pi} =a_{3!}\gf_{-I}\otimes_\gf H^*(\CP^1)\to 
a_{3!}\gf_{-I}\stackrel \beta\to a_{3!}( \gf_{-I}\oplus \gf_{-I})=
 R(i_2i_1a_1)_!\gf_{\{-\pi,\pi\}}$$.
Thus we see that $\alpha_1$ factors as $\alpha_1=\beta u$.
It is well known that in this case we have a quasi-isomorphism
$$
\Cone(\alpha_1\oplus \beta)\cong \Cone(0\oplus \beta).
$$
 meaning that the object in $(\ref{vtor})$ is isomorphic to $Ri_{2!}\gf_B\oplus \gf_{-I}[-1]$ (because $\Cone(\beta)\cong \gf_{-I}[-1]$).

Let $\ve:0\in B$ be the zero matrix. one then has a $C$-isomorphism 
$Ri_{2!}\gf_B \to Ri_{2!}\ve_!\gf_0=\gf_e$. Analogously, by choosing a
point $0'\in B_\pi$, one gets a $C$-isomorphism $Ri_{2!}\gf_B\to \gf_{-I}$.
Therefore, 
the object in (\ref{vtor}) is $C$-isomorhic
to $\gf_e\oplus \gf_{-I}[-1]$ and $\gf_{-I}$ is $C$-isomorphic
with $\gf_e$ (via $Ri_{2!}\gf_B$). Thus, the object in (\ref{vtor}),
hence $\gf_{S_1}$ is $C$-isomorphic to $\gf_e\oplus \gf_e[-1]$.
Lemma is proven.

\subsection{Proof of Proposition \ref{RPN}}.
In this subsection we fix $\char \gf=2$.

We have standard embeddings
$$
\SO(2)\subset \SO(3)\subset \cdots\subset \cdots \SO(N)\subset \SU(N)
$$
where the embedding $\SO(k)\subset \SO(N)$ is induced by the 
embedding $\Re^k\hookrightarrow \Re^N$; 
$(x_1,x_2,\ldots,x_k)\mapsto
(x_1,x_2,\ldots,x_k,0,\ldots,0)$.

We will prove the following statement.
\begin{Lemma} \label{SO} The sheaf $\gf_{\SO(k)}\in D(\SU(N))$
is $C$-isomorphic to $\gf_{\SO(k-1)}\oplus \gf_{\SO(k-1)}[1-k]$,
for all $k\geq 2$.
\end{Lemma}
It is clear that this Lemma implies the Proposition.
Let us now prove Lemma
\subsubsection{}
We have an embedding $\SO(k)\subset \SU(k)\subset \SU(N)$ and, in the same way as in the proof of Lemma \ref{Lm}, it suffices to prove
that $\gf_{\SO(k)}$ is $C$-isomorphic to $\gf_{\SO(k-1)}\oplus
\gf_{\SO(k-1)}[1-k]$ in $D(\SU(k))$. 

\subsubsection{} Let $M:=\SU(k)/\SO(k-1)$, let $\Pi:\SU(k)\to M$ be
the canonical projection.  

For any smooth manifold $Y$
Let $C(Y)\subset D(Y)$ be the full
subcategory formed by finite extensions of objects of the form
$Rp_!X$ where $p:Q\to Y$ is a continuos map, $Q=[0,1]^{M}$,
 $M\geq 0$,
$X\in D(Q)$; $R\Gamma(Q,X)=0$.

\begin{Lemma}\label{perenos} If $F\in C(M)$, then $\Pi^{-1}F\in C(\SU(k))$.
\end{Lemma}
\begin{proof} Let $p:Q\to M$ be a continuous map.
 $\Pi$ is a locally trivial fibration with fiber
$\SO(k-1)$, let $\Pi_Q:\SU(k)\times_M Q\to Q$ be the pull-back of this fibration with respect to the map $p:Q\to M$.
The fibration $\Pi_Q$ is trivial, hence we have a homeomorphism
$$
\SO(k-1)\times Q\cong \SU(k)\times_M Q.
$$
We then have   natural maps
$$
\pi:\SO(k-1)\times Q\cong \SU(k)\times_M Q\to \SU(k);
$$
Let $q':\SO(k-1)\times Q\to \SO(k-1)$,
$q:\SO(k-1)\times Q\to Q$, be  projections.
Let $X\in D(Q)$, $R\Gamma(Q,X)=0$. We then have
$\Pi^{-1}Rp_!X=R\pi_!q^{-1}X$.

Let us cover $$\SO(k-1)=\bigcup\limits_{i=1}^n Q_i,$$
 where each $Q_i\subset \SO(k-1)$ is a closed subset homeomorphic
to a  cube.
One then can represent the sheaf $\gf_{\SU(k-1)}$ (actually any
object of $D(\SU(k-1))$ as a finite extension formed by objects
$Y_i\in D(\SU(k-1))$ such that each $Y_i$ is supported on 
$Q_{l_i}$ for some $l_i$. Let $Z_i\in D(Q_{l_i})$, 
$Z_i=Y_i|_{Q_{l_i}}$.
The object
$q^{-1}X$ is then a finite extension of objects of the form
$$
q^{-1}X\otimes (q')^{-1} Y_i
$$

Let $\pi_i:Q_{l_i}\times Q\to \SO(k-1)\times Q\to \SU(k)$
be the through map. Let $q_i:Q_{l_i}\times Q\to Q$, $p_i:Q_{l_i}\times Q\to Q_{l_i}$ be projections.
 
 We then have
$
R\pi_!q^{-1}X
$
is a finite extension formed by objects 

$$R\pi_!(q^{-1}X\otimes (q')^{-1}Y_i)\cong R\pi_{i!}(q_i^{-1}X\otimes 
p_i^{-1}Z_i)\in C(\SU(k)).$$
Therefore, $\Pi^{-1}R\pi_! X\in C(\SU(k))$, whence the statement.
\end{proof}
\subsubsection{} We have an identification $\SO(k)/\SO(k-1)=S^k$.
We have the natural map $S^k=\SO(k)/\SO(k-1)\to \SU(k)/\SO(k-1)=M$.

This map is an embedding; denote the image of this embedding
$S\subset M$. 
Let $\overline{e}\in S^{k-1}$ be the image of the unit of $\SO(k)$. Fix the 
standard basis $(e^1,e^2,\ldots,e^k)$ in $\Re^k$. Then 
$S^k$ gets identified with the unit sphere in $\Re^k$ and
$\overline{e}=e^k$. The point $\overline{e}$ 
determines a point on $S$, to be also
 denoted by $\overline{e}$.

Lemma \ref{perenos} implies that Lemma \ref{SO}
follows from the following statement:
\begin{Lemma} The object $\gf_{S}$ is $C(M)$-equivalent
to $\gf_{\overline{e}}\oplus \gf_{\overline{e}}[1-k]$.
\end{Lemma}
\begin{proof}
As was explained above, $S$ is identified with the unit sphere
in $\Re^k$. Let $V\subset \Re^k$ be an orthogonal complement
to $e_k$. Let us denote $e:=e_k$ and $\ve=-e$.
 Let $B\subset V$ be the   ball of radius $\pi$.
 We have a surjective
 map $P:B\to S$: let $f=\phi n\in B$, where $0\leq \phi\leq \pi$ and
$n\in B$. Set $P(\phi n)=\cos(\phi)e+\sin(\phi)n$.
It follows that $P$ is 1-to 1 on the interior of $B$ and that
$P$ takes the boundary of $B$ to the point $\ve\in S$.
Let $c:B\stackrel P\to S\to M$  be the through map
Let $\dB\subset B$ be the boundary. We have a commutative
diagram
\begin{equation}\label{dig}
\xymatrix{B \ar[r]^c& M\\
            \dB\ar[u]^{i}\ar[r]^p&\ve\ar[u]^{\iota}}
\end{equation}
One has 
\begin{equation}\label{f0}
\gf_S\cong \cone(Rc_!\gf_B\oplus \iota_!\gf_\ve\stackrel {f_0}\to \iota_!Rp_!\gf_{\dB}),
\end{equation}
where  $f_0=\alpha\oplus \beta$;
 the map $\alpha:Rc_!\gf_B\to \iota_!Rp_!\gf_{\dB}=
Rc_!i_*\gf_{\dB}$ is induced by the canonical map
$$
\gf_B\to i_*\gf_\dB,
$$
and the map 
$$
\beta:\iota_!\gf_e\to \iota_!Rp_!\gf_\dB
$$
is induced by the canonical map
$$
\gf_\ve\to Rp_*\gf_\dB=Rp_!\gf_\dB.
$$

Let $M:B\to \SO(k)$ as follows:

---$M(\phi n)$ is identity on any vector
which is orthogonal to both $n$ and $e$;

---$M(\phi n) e=\cos(\phi)e+\sin(\phi)n$;

---$M(\phi n) n=-\sin(\phi)e+\cos(\phi)n$.
  
One then sees that the composition
$$
B\stackrel M\to \SO(k)\stackrel\Pi\to S
$$
equals $P:B\to S$. Thus, $P=\Pi M$. One can also rewrite:
$$
M(\phi n)=I+(e^{i\phi}-1)\pr_{(e+in)/\sqrt 2}+
(e^{-i\phi}-1)\pr_{(e-in)/\sqrt 2},
$$
where $\pr$ is the orthognal projector.

For $0\leq \alpha\leq \pi/4$, set
$$
\mu(\alpha,\phi n)=I+(e^{i\phi}-1)P_{(\cos\alpha e+i\sin \alpha n)
}+
(e^{-i\phi}-1)P_{(\sin \alpha e-\cos\alpha in)}
$$

One sees that:

  $\mu:[0,\pi/4]\times B\to \SU(k)$;

$$
\mu(\alpha,0)=I;
$$
$$
\mu(\alpha,\pi n)\in \SO(k); 
$$
$$
\mu(\pi/4,\phi n)=M(\phi n);
$$
$$\mu(\alpha,\pi n)e=-e.
$$
Let $\nu:[0;\pi/4]\times B\stackrel \mu\to \SU(k)\to M$ be
 the through map. It then follows that $\nu(\alpha,\pi n)=\ve$.

We have a commutative diagram
\begin{equation}\label{digdva}
\xymatrix{
B\ar[r]^{i} & [0;\pi/4]\times B \ar[r]^\nu& M\\
\dB\ar[r]^{i_0}\ar[u]^{k_0}& [0,\pi/4]\times \dB 
\ar[r]^{\pi}\ar[u]^{k_1} & \ve\ar[u]^\iota}
\end{equation}

Here $i(b)=(\pi/4,b)$ for all $b\in B$; $i_0(b)=(\pi/4,b)$
for all $b\in \dB$.

We have $c=\nu i$; $\pi i_0=p$ (where $p$ is as in diagram (\ref{dig})).

In a  way similar to above we can construct a map
$$
f:R\nu_! \gf_{[0;\pi/4]\times B}\oplus i_!\gf_\ve\to
\iota_!R\pi_!\gf_{[0;\pi/4]\times \dB}
$$
The diagram (\ref{digdva}) gives rise to a commutative diagram in
$D(M)$:
\begin{equation}\label{digtri}
\xymatrix
{Rc_!\gf_B\oplus \iota_!\gf_\ve\ar[r]^{f_0}& \iota_!Rp_!\gf_{\dB}\\
R\nu_! \gf_{[0;\pi/4]\times B}\oplus\iota_! \gf_\ve\ar[r]^f\ar[u]^a&
\iota_!R\pi_!\gf_{[0;\pi/4]\times \dB}\ar[u]}
\end{equation}
in which the right vertical arrow is an isomorhism;
the left vertical arrow is a direct sum of the identity
arrow $\iota_!\gf_e$ and the natural arrow
$$
a:R\nu_!\gf_{[0;\pi/4]\times B}\to R\nu_!Ri_{!}\gf_B=Rc_!\gf_B.
$$
This diagram defines uniquely a  map $A:\Cone(f)\to \Cone(f_0)$ (because
the rightmost arrow in diagram (\ref{digtri}) is an isomorphism)
the cone of this map is isomorphic to the cone of the map
$a$.  It easily follows that $\Cone(a)\in C(M)$, therefore,
$A$ is a $C(M)$-isomorphism.

Consider now the diagram (\ref{digtri}) where all ingredients are the same
except that  the map $i:B\to [0;\pi/4]\times B$ gets
replaced with the map $i_1:B\to [0;\pi/4]\times B$, where
$i_1(b)=(0,b)$. Let us compute $c_1:=\nu i_1:
B\to M$. We have 
\begin{equation}\label{ur}
\mu(0, \phi n)=I+(1-e^{i\phi})\pr_e + (1-e^{-i\phi})\pr_n;
\end{equation}
\begin{equation}\label{urdva}
c_1(\phi n)=P\mu(0,\phi n).
\end{equation}

We then have a commutative diagram obtained from
diagram (\ref{dig}) by replacement $c$ with $c_1$.
Hence we have a map\begin{equation}\label{f1}
\gf_S\cong \cone(Rc_{1!}\gf_B\oplus \iota_!\gf_e\stackrel {f_1}\to \iota_!Rp_!\gf_{\dB}),
\end{equation}
constructed in the same way as the map $f_0$ in (\ref{f0}).

In the same way as above one can show that $\Cone(f_1)$ is 
$C(M)$-isomorphic to $\Cone(f)$, hence to $\Cone(f_0)$, hence
to $\gf_{S}$. 

Let us now work with $\Cone(f_1)$. 

1) Eq. (\ref{ur}) and (\ref{urdva}) imply that
$c_1(rn)=c_1(-rn)$ for any $rn\in B$.  Let $B/2$ be the quotient of $B$ in which $b\in B$ gets identified with $-b$. 
Let $\delta:B\to B/2$ be the projection. We then have
a unique map $c_2:B/2\to M$ such that $c_1=\delta c_2$. 
Let $\dB/2$ is the image of $\dB$ in $B/2$. Of course,
$\dB\cong S^{k-2}$ and $\dB/2\cong \RP^{k-2}$. 
We have a natural quotient map $\delta_1:\dB\to \dB/2$. 
These maps fit into the following commutative diagram:
$$
\xymatrix{B\ar[r]^\delta& B/2\ar[r]^{c_2}& M\\
          \dB\ar[u]^{i}\ar[r]^{\delta_1}&\dB/2\ar[u]^{i_1}
\ar[r]^{p_1} & \ve\ar[u]^\iota}
 $$
One then can construct an arrow
$$
f_2: Rc_{2!}\gf_{B/2}\oplus \gf_\ve\to \iota_!(p_1\delta_1)_!\gf_{\dB}
$$
in the same way as above. Similar to above, there exists 
a natural map
$$
\Cone(f_2)\to \Cone(f_1)
$$
whose cone is isomorphic to the cone of the natural map
\begin{equation}\label{oto}
Rc_{2!}\gf_{B/2}\to Rc_{2!}R\delta_!\gf_B.
\end{equation}
Let us show that {\em the cone of this map is in $C(M)$.}

Indeed, choose a covering $\dB=\bigcup_{k=1}^m C_k$  where $C_k$, and all 
 non-empty intresections  of these sets  are closed sets 
homeomorphic
to the closed disk of the same dimension as dimension of $\dB$
 and $C_k\cap -C_k=\emptyset$. 

Consider the set of all multiple non-empty intersections of the sets $C_k$ and denote elements of this set 
by $C'_1,C'_2,\ldots,C'_M$. Each of these sets is homeomorphic
to a closed disk of the same dimension as dimension of $\dB$
and  for each $i$, 
 $C'_i\cap -C'_i=\emptyset$.  

 Let $B_k\subset B$ be the cones of $C'_k$:
$$
B_k=\{rn| 0\leq r\leq \pi; n\in C'_k\}.
$$
It is clear that $B_k$ cover $B$ and that $B_k\cap -B_k=\{0\}$.

Let $B_k/2$ be the images of $B_k$ in $B/2$. The map $\delta|_{B_k}:B_k\to B_k/2$  is a homeomorophism.
 It follows that
 $\gf_{B/2}$ is a finite extension of objects,
each of them being of the form  $\gf_{B_k/2}$. It then suffices to show that the cone of the natural map
$$
c_{2!}\delta_!\gf_{B_k\cup -B_k}=c_{2!}\delta_!\delta^{-1}\gf_{B_k/2}\to 
c_{2!}\gf_{B_k/2}\in C(M)
$$

We have $$\delta_!\gf_{B_k\cup -B_k}=\delta_!(\Cone (\gf_{B_k}\oplus
\gf_{B_k}\to \gf_0))$$
$$
=\Cone(\gf_{B_k/2}\oplus \gf_{B_k/2}\to \gf_0)
$$
 The natural map $\delta_!\gf_{B_k\cup -B_k}\to \gf_{B_k/2}$
is given by the natural map
\begin{equation}\label{shish}
\Cone(\gf_{B_k/2}\oplus \gf_{B_k/2}\to \gf_0)\to \gf_{B_k/2}
\end{equation}
induced by
$$
\Id\oplus \Id:\gf_{B_k/2}\oplus \gf_{B_k/2}\to \gf_{B_k/2}
$$
Therefore, the cone of the map (\ref{shish}) is isomorphic
to the cone of the natural map
$$
\gf_{B_k/2}\to \gf_0
$$
Denote this cone by $F'$ and let  $F:=F'|_{B_k/2}$. 
It follows that $R\Gamma(B_k/2,F)=0$.
Let $P:B_k/2\to B/2\to M$ be the trough map. Our task is now 
reduced to showing that $RP_!F\in C(M)$. This follows from the
fact that $B_k/2$ is homeomorphic to a unit cube.

Thus, the cone of the map (\ref{oto}) is in $C(M)$, therefore
$\Cone(f_1)$ and $\Cone(f_2)$ are $C(M)$ isomorphic.

Let us now study $\Cone(f_2)$.
The  map $f_2$ is a direct sum of two maps: one of them is the natural map
 $g: Rc_{2!}\gf_{B/2}\to \iota_!(p_1\delta_1)_!\gf_{\dB}=
Rc_{2!}i_{1!}\delta_{1!}\gf_{\dB}$
and the other is the natural map
\begin{equation}\label{h}
h:\gf_{\ve}\to \iota_!(p_1\delta_1)_!\gf_{\dB}
\end{equation}

The map $g$ factors as
\begin{equation}\label{l}
Rc_{2!}\gf_{B/2}\stackrel{g_1}\to Rc_{2!}i_{1!}\gf_{\dB/2}\stackrel l\to
Rc_{2!}i_{1!}\delta_{1!}\gf_{\dB}
\end{equation}
We have $Rc_{2!}\gf_{\dB/2}=\iota_!Rp_{1!}\gf_{\dB/2}\cong
H^*(\dB/2,\gf)\otimes_\gf \iota_! \gf_\ve$;
$$
Rc_{2!}i_{1!}\delta_!\gf_{\dB}=\iota_!Rp_!\gf_{\dB}=
H^*(\dB;\gf)\otimes_\gf \iota_!\gf_\ve.
$$
The map $l$ in (\ref{l}) is induced by the map
$$
\delta_1^*:H^*(\dB/2;\gf)\to H^*(\dB;\gf).
$$
Recall that $\dB\cong S^{k-2}$; $\dB/2\cong \RP^{k-2}$ and
$\delta_1$ is the quotient map. As $\char \gf=2$, it follows
that the map $\delta_1^*$ factors as
$$
H^*(\dB/2;\gf)\stackrel{n_1}\to \gf\stackrel{n_2}\to H^*(\dB;\gf),
$$
where the  arrow  $n_1$ is induced by any embedding
$\text{pt}\to \dB/2$ and the  arrow $n_2$
is induced by the projection $\dB\to \text{pt}$.
This means that $l=l_2l_1$, where 
$$
l_1:Rc_{2!}\gf_{\dB/2}\to \iota_!\gf_\ve
$$
is induced by $n_1$, and 
$$
l_2:\gf_\ve\to Rc_{2!}i_{1!}\delta_{1!}\gf_{\dB}
$$
is induced by $n_2$.
Let us now consider the map $h$ in (\ref{h}).
As was explained above, $\iota_!(p_1\delta_1)_!\gf_{\dB}\cong
H^*(\dB;\gf)\otimes_\gf \iota_!\gf_\ve$  and the 
map $h$ is induced by the map $\gf\to H^*(\dB;\gf)$ induced
by the projection $\dB\to \pt$. That is $h=l_2$

These observations show that the map $g=l_2l_1g_1=hl_1g_1$ factors through $h$. This implies that 
$$
\Cone(f_2)=\Cone(g\oplus h)=\Cone(0\oplus h)=Rc_{2!}\gf_{B/2}\oplus
\Cone(h)=Rc_{2!}\gf_{B/2}\oplus \iota_!\gf_{\ve}[1-k]
$$

As was explained above, $Rc_{2!}\gf_{B/2}$ is $C$-isomorphic
to $Rc_!\gf_B$.  Let $x\in \dB$. We then have
natural $C$-isomorphisms
$$
Rc_!\gf_B\to Rc_!\gf_0=\gf_e
$$
and
$$
Rc_!\gf_B\to Rc_!\gf_{0'}=\gf_\ve
$$
hence, $Rc_{2!}\gf_{B/2}$ is $C$-isomorphic with
both $\gf_e$ and $\gf_\ve$, as well as with $Rc_{2!}\gf_{B/2}$.

 Thus,
$$
Rc_{2!}\gf_{B/2}\oplus \iota_!\gf_{\ve}[1-k]
$$
is $C$-isomorphic with $\gf_e\oplus \gf_e[1-k]$, hence so 
is $\Cone(f_2)$. This proves Lemma
\end{proof}
\section{Proof of Proposition \ref{diagonal}: constructing
$u_\O$}\label{const:uo}
The rest of this paper will be devoted to proving Proposition 
\ref{diagonal}. In this section we will construct the object 
$u_\O$. In the subsequent sections we will check it satisfies
all the required properties. 


\subsection{Constructing $u_\O$} 
Our construction is based on a certain object $\sp\in D(G\times \h)$.
This object is introduced and studied in the subsequent Sec. \ref{specp}
It is defined as any object satisfying the conditions in Theorem \ref{specth}.

\subsubsection{Convolution on $\h$}
Let $X,Y$ are manifolds. Let $a:X\times \h\times 
Y\times \h\to X\times Y\times \h $ be given
by $a(x,A_1,y,A_2)=(x,y,A_1+A_2)$. 
Let $F\in D(X\times \h)$ and $G\times D(Y\times \h)$. Set $F*_\h G:=Ra_!(F\times G)$.
\subsubsection{}\label{odindvatri}
 Let $L:=\O\cap C_+$. We have
 $L=\lambda e_1$, where $\lambda>0$.

Let $\gamma_L\in D(\h\times \Re)$ be given by
$\gamma_L=\gf_{\{(A,t)|t+<A,L>\geq 0\}}$. 
Let $I_0:G\times \Re\to G\times\h\times \Re$ be given by
$I_0(g,t)=(g,0,t)$.  Set
$$
u_\O=I_0^{-1}(\sp*_\h \gamma_L).
$$

Let us first of all prove that $u_\O\in \cD_{IP^{-1}\Delta}(G)$.
Using proposition \ref{checkorthogonal} it is easy to
show that $u_\O$ is in the left orthogonal complement to
$C_{\leq 0}(G)$. 
Let us now estimate $\mS(u_\O)$.

Let $p_3:G\times \h\times \Re\to G\times \h$;
$p_1:G\times \h\times \Re\to \h\times \Re
$;
$
p_2:G\times \h\times \Re\to G\times \Re
$
be the projections.

One can show that
$$
u_\O=Rp_{2!}(p_1^{-1}\gf_{\{(A,t)|t\geq <A,L>\}}\otimes p_3^{-1}\sp)$$

As usual let us identify $$T^*(G\times \h \times \Re)=
G\times \h\times \Re\times \g^*\times \h^*\times \Re$$
We see that $p_1^{-1}\gf_{\{(A,t)|t\geq <A,L>\}}$ is microsupported
on  the set
$$
\Omega_1:=\{(g,A,t,0,-kL;k)| k\geq 0\}.
$$

The object $p_2^{-1}\sp$ is microsupported on the set
$$
\Omega_2:=\{(g,A,t,\omega,\eta,0)\}, 
$$
where $(g,A,\omega,\eta)\in \Omega_\sp$ (See Sec. \ref{omegasps}
for the definition of$\Omega_\sp$).

One sees that if $\zeta_j \in \Omega_j\cap T^*_{(g,A,t)}(G\times 
\h\times \Re)$ and $\zeta_1+\zeta_2=0$, then $k=0$ and $\zeta_1=0$,
hence $\zeta_2=0$. Therefore,
the object
$$
\Psi:=p_1^{-1}\gf_{\{(A,t)|t\geq <A,L>\}}\otimes p_3^{-1}\sp
$$

is microsupported on the set
$$
\Omega_3:=\{(g,A,t,\omega_1+\omega_2;\eta_1+\omega_2;k_1+k_2)| 
(g,A,t,\omega_j;\eta_j;k_j)\in \Omega_j\}
$$

We have
$$
\Omega_3=\{(g,A,t,\omega;\eta-kL;k)|k\geq 0;(g,A,\omega,\eta)\in
\Omega\}
$$

Let us now apply Corollary  \ref{impropersupport} to the projection
$p_2$ (so that $E=\h$).

Let $$\pi:G\times \h\times\Re\times\g^*\times \h^*\times \Re\to
G\times\Re\times \g^*\times \h^*\times \Re.
$$

Let us find $\pi(\Omega_3)$  We see that
$$
\pi(\Omega_3)\subset \{(g,t,\omega,\eta-kL;k)|k\geq 0,\Ad_g \omega=
\omega;\eta=|\omega|\}=:\Omega_4.
$$

The set $\Omega_4$ is closed.  Therefore,
$\mS(Rp_{2!}\Psi)$  is confined within the set of all  points 
of the form $\{(g,t,\omega,k)|(g,t,\omega,0,k)\in\Omega_4\}$
Thus $\|\omega\|=kL$,$\Ad_g\omega=\omega$ and $k\geq 0$. If $k=0$. then $\omega=0$ and we have $(g,t,0,0)\in \mS(Rp_{2!}\Psi)$.
If $k>0$, then set $\omega=k\zeta$. We then have
$|\zeta|=L$ (which means that $\zeta\in \O$) and $\Ad_g\zeta=\zeta$. This is the same as to say
$(g,\zeta)\in IP^{-1}\O$.
This proves the statement

\subsection{Proof of Proposition \ref{diagonal} 1)}

\subsubsection{The map $\tau_c:u_\O\to T_{c*}u_\O$}  We will rewrite this map in a way more convenient to us.

Let $c>0$. We then have an obvious map 
$\tau^\gamma_c:\gamma_L\to T_{c*}\gamma_L$;
$$
T_{c*}u_\O=I_0^{-1}(\sp*_\h T_{c*}\gamma_L)
$$

The natural map $\tau_c:u_\O\to T_{c*}u_\O$ (coming from the fact
that $u_\O\in \cD(G)$), in terms of the above identifications,
is given by the map
$$
I_0^{-1}(\sp*_\h \gamma_L)\to I_0^{-1}(\sp*_\h T_{c*}\gamma_L)
$$
 which is induced by the map $\tau^\gamma_L$.

Let $A_1,A_2\in \h$. For $A\in \h$ set $U_A=\{A_1\in \h| A_1<<A\}$.
Set $V_A:=\gf_{U_A}[\dim \h]$.  We then have
a natural map 
\begin{equation}\label{nuochennat}
e_A:\gf_A\to V_A.
\end{equation}
It is defined as follows.
Let us identify $\Re^{N-1}=\h$, where 
$$
(x_1,x_2,\ldots,x_{N-1})\mapsto \sum x_kf_k.
$$
  Let $A=\sum t_kf_k$ Upon this identification,
$U_A=\{(x_1,x_2,\ldots,x_{N-1})| x_k<t_k\}$ and 
$V_A=\boxtimes_k (\gf_{(-\infty,t_k)}[1])$; $\gf_A=\boxtimes_k
\gf_{t_k}$. The map  $e_A$ is defined as a product of maps
$\ve_k:\gf_{t_k}\to \gf_{(-\infty;t_k)}[1]$ which represents the class of the extension 
$$
0\to \gf_{(-\infty;t_k)}\to \gf_{(-\infty,t_k]}\to \gf_{t_k}\to 0.
$$

Let $A\in C_+$, we then have $c=<A,L>\geq 0$ because $A,L\in C_+$.

\begin{Lemma}  Let $A\in \h$ be  such that $<A,L>=c$
  
The natural map $$ \gf_A*_\h \gamma_L\stackrel{e_A}\to 
V_A*_\h\gamma_L
$$
is an isomorphism.
\end{Lemma}
\begin{proof}
Clear
\end{proof}

Let now $A\in C_+$. Since $A,L\in C_+$, it follows that
$c=<A,L>\geq 0$. We also have a natural isomorphism
$$
\gf_A*_\h\gamma_L\cong T_{c*}\gamma_L.
$$
Let us combine this isomorphism with that of the Lemma, 
we will get
an isomorhpism
$$
V_A*\gamma_L\cong T_{c*}\gamma_L
$$
By substituting $A=0$, we get an isomorphism
$$
V_0*\gamma_L\cong \gamma_L.
$$

Upon these identifications,
the map $\tau_c^\gamma$ corresponds to a map
$$
\tau^V_A:V_0\to V_A
$$
induced by the inclusion $U_0\subset U_A$.

Thus, the map $\tau_c:u_\O\to T_cu_\O$ is isomorphic to the map
$$
I_0^{-1}(\sp*_\h V_0*_\h \gamma_L)\to I_0^{-1}(\sp*_\h V_A*_\h \gamma_L)$$
induced by the natural map $\tau^V_A:V_0\to V_A$.
As $\h$ is an abelian Lie group, we can rewrite  the above map
as
\begin{equation}\label{tauc}
I_0^{-1}(V_0*_\h\sp*_\h \gamma_L)\to I_0^{-1}(V_A*_\h\sp*_\h \gamma_L).\end{equation}

\subsubsection{} Let $*_{G\times \h}$ denote the convolution
on $D(G\times \h)$. 

 Taking into account the expression (\ref{tauc}) for $\tau_c$,
 the Proposition \ref{diagonal} 1) can be deduced from the 
following Proposition:

\begin{Proposition}\label{htors} Let $U$ and $F\in D(G)$ be as in Proposition \ref{diagonal} 1). Then there exists $A\in C_+$ such that the natural map
\begin{equation}\label{kruch}
(F\boxtimes V_0)*_{G\times \h}\sp\to (F\boxtimes V_A)*_{G\times\h}\sp
\end{equation}
induced by the map $\tau^V_A:V_0\to V_A$, is zero in $D(G\times \h)$
\end{Proposition}

Thus, Proposition \ref{diagonal} 1) is now reduced to Proposition
\ref{htors}

\subsection{Proof of Proposition \ref{htors}.}
Let $H$ be any sheaf on $\h$.  Let $\alpha:\h\to\h$ be the antipode map.
We then have $H*_\h \sp =Rp_{2!}(p_1^{-1}\alpha_*H\otimes a^{-1}\sp)$, where as usual $p_1:G\times \h\times \h\to \h$ is given by
$$
p_1(g,A_1,A_2)=A_1;
$$ 
and $p_2:G\times \h\times \h\to G\times \h$
is given by $p_2(g,A_1,A_2)=(g,A_2)$.
Set $H^\alpha:=\alpha_* H$.
We then have$$
Rp_{2!}(p_1^{-1}H^\alpha\otimes a^{-1}\sp)=
Rp_{2!}((p_1^{-1}H^\alpha)\otimes (\sp*_G \sp)),
$$
where we have used the isomorphism (\ref{convsp}).
Next,
$$
Rp_{2!}((p_1^{-1}H^\alpha)\otimes (\sp*_G \sp))
\cong
[Rp_!(\pi^{-1}H^\alpha\otimes \sp)]*_G \sp,
$$
where $\pi:G\times \h\to\h$; $p:G\times \h\to G$ are projections.
 
One then has
$$
Rp_!(\pi^{-1}H^\alpha\otimes \sp)=I_0^{-1}( H*_\h \sp).
$$

Let $S_A:=I_0^{-1}( V_A*_\h \sp)$. 
We then have a natural map $\tau_A^S:S_0\to S_A$.

We have $V_A*_\h \sp\cong S_A*_G\sp$ and
$$
(F\boxtimes V_A)*_{G\times \h}\sp\cong F*_G(S_A*_G\sp)=
(F*_G S_A)*_G \sp
$$

The map \ref{kruch} is then induced by the map $\tau_A^S$.

Thus, Proposition \ref{htors} is  now reduced to
\begin{Proposition}\label{htors1} There exist: a neighborhood $U\subset G$ of the unit $e\in G$ and
 $A\in C_+$
such that the natural map
$$
F*_G S_0 \to F*_G S_A
$$
induced by $\tau^S_A$ is zero for any  $F\in D(G)$ which 
is supported on $gU$ for some $g\in G$ and satisfies $R\Gamma(G,F)=0$.
\end{Proposition}
\begin{proof}
We have a natural map $\gf_A\to V_A$, as in (\ref{nuochennat}).
 Hence, we have an induced map

\begin{equation}\label{SA}
I_0^{-1}(\gf_{A}*_\h \sp)\to I_0^{-1}(V_A*_\h\sp)=: S_A.
\end{equation}
One sees that  {\em this map is actually an isomorphism.}
Indeed, one can easily show that for any object $F\in D(G\times \h)$
such that $\mS(F)\subset T^*G\times \h\times C_+\subset T^*G\times
T^*\h$,  the map
$$
I_0^{-1}(\gf_{A}*_\h F)\to I_0^{-1}(V_A*_\h F)
$$
induced by the map (\ref{nuochennat}), is an isomorphism, and $\sp$ is of this
type by virtue of Theorem \ref{specth}.

 One also
sees that 
$I_0^{-1}(\gf_{A}*_\h\sp)=I_{-A}^{-1}\sp$, where $I_{-A}:G\to G\times \h$;
$I_{-A}g =(G,-A)$.   Taking into account (\ref{SA}),
we obtain an isomorphism
$$
S_A\cong I_{-A}^{-1}\sp.
$$
\def\cU{\mathcal{U}}
Let us choose a small $A$, $A>>0$. 

As was shown in the course of proving Theorem \ref{specth}, for  $0<<A<<b$ we have
$$
S_A=I_{-A}^{-1}\sp\cong \gf_{\cU_A}.
$$

where $\cU_A=\{e^{X}| \|X\|<< A\}\subset G$. We also know that $S_0=\gf_e$.

Without loss of generality one can assume that for some $A\in C_+$; $A<<b$,
$U\subset \cU_{A/10}$. Let  $h\in U$ so that $h=e^X$, where $\|X\|<< A/10$.  We have
$(F*_G S_A)|_{gh}=R\Gamma_c(\{ghr^{-1}| r\in \cU_A\};F)[\dim G]$. 
It follows
that $gU\subset  \{ghr^{-1}| r\in \cU_A\}$ (Indeed, let $gh'\in gU$ so that $h'=e^{X'},\|X'\|<<A/10$.
We have $h'=hr^{-1}$, $r=(h')^{-1}h$. By Lemma \ref{Klyachko}, $r=e^Z$, where $\|Z\|\leq \|-X'\|+\|X\|<<A$.
So $r\in \cU_A$).
 Therefore,
$(F*_G S_A)|_{gh}=R\Gamma(gU,F)[\dim G]=0$.
Thus, $F*_G S_A$ is supported away from  $gU$.
But $F*_G S_0=F$ is supported on $gU$.  Therefore, 
$R\hom(F*_G S_0;F*_G S_A)=0$ which proves the statement.
\end{proof}

Thus, we have proven Proposition \ref{diagonal} 1). The rest
of the paper is devoted to proving the second part of the Proposition. 
\subsection{} Recall that we have a sheaf $\gamma_L:=\gf_{\{(A,t)|
t+<A,L>\geq 0\}}$ on $\h\times \Re$.  Let
 $\iota:\Re\to \h\times \Re$
be given by $\iota(t)=(0,t)$. We have a natural isomorphism
$$
\gf_{[0,\infty]}[-\dim \h]=\iota^!\gamma_L
$$

hence  a natural map
\begin{equation}
\label{odindva}
\gf_{0\times[0,\infty)}[-\dim \h]\to \gamma_L.
\end{equation}
This map induces a map
\begin{equation}\label{restr1}
I_0^{-1}(\sp*_\h\gf_{0\times [0,\infty)})[-\dim\h]\to 
\gf_{I_0^{-1}(\sp*_\h \gamma_L)}=u_\O
\end{equation}

where $I_0:G\times \Re\to G\times \h\times \Re$, $I_0(g,t)=(g,0,t)$
 Next, one has$$
I_0^{-1}(\sp*_\h\gf_{0\times [0,\infty)})=
i_0^{-1}\sp \boxtimes \gf_{[0,\infty)}
$$
where $i_0:G\to G\times \h$, $i_0(g)=(g,0)$.
We know that $i_0^{-1}\sp=\gf_e$, thus we have an isomoprhism
$$
I_0^{-1}(\sp*_\h\gf_{0\times [0,\infty)})=\gf_{e\times [0,\infty)}
$$
The map (\ref{restr1})
then can be rewritten as:
\begin{equation}\label{restr}
\gf_{e\times [0,\infty]}[-\dim \h]\to u_\O
\end{equation}
\begin{Proposition} Let $\Phi\in \cD_{G\times \O}(G)$
 The natural map
$$
 \hom_{G\times \Re}
(u_\O;\Phi)\to R\hom_{G\times \Re}(\gf_{(e,0)}[-\dim \h];\Phi)
$$
induced by the map (\ref{restr}) is an isomorphism.
\end{Proposition}
\begin{proof}

We have 
$$
u_\O=Rp_{2!}(p_3^{-1}\sp\otimes p_1^{-1}
\gf_{\{(A,t)|t\geq (A,L)\}});
$$
$$\gf_{e\times [0,\infty)}=
I_0^{-1}(\sp*_\h \gf_{0\times [0,\infty)})
$$
$$
=Rp_{2!}(p_3^{-1}\sp\otimes p_1^{-1}
\gf_{0\times [0,\infty)})
$$
where  $p_1:G\times \h\times \Re\to \h\times \Re$;
$p_2:G\times \h\times \Re\to G\times \Re$; $p_3:G\times \h\times \Re\to G\times \h$ are projections.

Let $X\in D(\h\times \Re)$. We then have

$$R\hom(Rp_{2!}(p_3^{-1}\sp\otimes p_1^{-1}
X);\Phi)$$
$$
=R\hom(p_3^{-1}\sp\otimes p_1^{-1}X;
p_2^!\Phi)$$
$$
=R\hom( p_1^{-1}X;R\ihom(p_3^{-1}\sp;p_2^!\Phi)
$$
\begin{equation}\label{XXX}
=R\hom_{\h\times \Re}(X;Rp_{1*}R\ihom(p_3^{-1}\sp;p_2^!\Phi)).
\end{equation}
Let us estimate the microsupport of the sheaf
$$Rp_{1*}R\ihom(p_3^{-1}\sp;p_2^!\Phi).
$$

We know that $\ms(\sp)\subset\{(g,A,\omega,|\omega|)\in G\times\h\times \g^*\times \h^*\}$.
Therefore,
$$
\ms(p_3^{-1}\sp)\subset \Omega_1:=
\{(g,A,t,\omega_1,|\omega_1|,0)\}\subset G\times \h\times \Re\times \g^*\times \h^*\times \Re.
$$

Analogously,
$$
\ms(p_2^{!}\Phi)\subset \Omega_2=\{(g,A,t,k\omega,0,k)|k\geq 0,\omega\in \O\}
$$
One sees that if $(g,A,t,\omega_i,\eta_i,k_i)\in \Omega_i$
and $\omega_2=\omega_1,\eta_2=\eta_1,k_1=k_2$, then
$0=k_1=k_2$, hence $\omega_2=\omega_1=0$; also $0=\eta_2=\eta_1$.
Therefore,
$$
\mS(R\ihom(p_3^{-1}\sp;p_2^!\Phi))\subset\Omega_3:=\{(g,A,t,\omega_2-\omega_1;\eta_2-\eta_1;k_2-k_1)|(g,A,t,\omega_i,\eta_i,k_i)\in \Omega_i\}
$$
$$
=\{(g,A,t,k\omega-\omega_1;-|\omega_1|;k)|k\geq 0;\omega\in \O\}
$$
As the map $p_1$ is proper, one has
$$
\mS(Rp_{1*}R\ihom(p_3^{-1}\sp;p_2^!\Phi))\subset\Omega_4:=
\{(A,t,\eta,k)|\exists g\in G: (g,A,t,0,\eta,k)\in \Omega_3\}.
$$
We see that
$$
\Omega_4=\{(A,t,-k|\omega|,k)\}=\{(A,t,-kL,k)\}.
$$

Let $\pi:\h\times \Re\to \Re$; $\pi(A,t)=t-<A,L>$.
It then follows that $Rp_{1*}R\ihom(p_3^{-1}\sp;p_2^!\Phi)$
is locally constant along the fibers of $\pi$ i.e.
there exists a sheaf $\Gamma$ on $\Re$ such that
$$
Rp_{1*}R\ihom(p_3^{-1}\sp;p_2^!\Phi)=\pi^!\Gamma
$$

Taking into account (\ref{XXX})
the statement  is reduced to showing that
the natural map
$$
R\hom_{\h\times \Re}(\gf_{\{(A,t)|t\geq <A,L>\}};\pi^!\Gamma)
\to
R\hom_{\h\times \Re}(\gf_{0\times [0,\infty)};\pi^!\Gamma)
$$
is an isomorphism for any sheaf $\Gamma\in D(\Re)$.
This is equivalent to showing that the  map
$$
 R\pi_!\gf_{0\times[0,\infty)}[-\dim \h]\to
 R\pi_!\gf_{\{(A,t)|t\geq <A,L>\}}
$$
induced by the map (\ref{odindva})
is an isomorphism, which is easy.
\end{proof}

It then follows that for all $c\in \Re$, we have an isomorphism
$$
R\hom(u_\O;T_{c*}u_\O)\cong R\hom(\gf_{e\times [0,\infty)}[-\dim \h]
;T_{c*}u_\O)$$

Let $i:\Re\to G\times \Re$; $i(t)=(e,t)$. We then have
$$
R\hom(\gf_{e\times [0,\infty)};T_{c*}u_\O)=
R\hom(\gf_{[0,\infty)};i^!T_{c*}u_\O).
$$

One sees that the submanifold $i(\Re)\subset G\times \Re$ is non-characteristic for $T_{c*}u_\O$ (because $\mS(T_{c*}u_\O)\subset
\{(g,t,k\omega,k),k\geq 0; \omega\in \O\}$).
Therefore, according to  Proposition \ref{ks:inverse}, we have an isomorphism
$$
i^{-1}T_{c*}u_\O[-\dim G]\cong i^!T_{c*}u_\O.
$$
Thus, we have an isomorphism
$$
\rho:R\hom(u_\O;T_{c*}u_\O)\cong R\hom(\gf_{[0,\infty)}[-\dim \h];
i^{-1}T_{c*}u_\O[-\dim G])
$$
For $c>0$ the natual maps
$$
R\hom(u_\O;u_\O)\to R\hom(u_\O;T_{c*}u_\O)
$$

and 
$$
R\hom(\gf_{[0,\infty)}[-\dim \h];i^{-1}u_\O[-\dim G])\to
 R\hom(\gf_{[0,\infty)}[-\dim \h];i^{-1}T_{c*}u_\O[-\dim G])
$$
commute with our isomorphism.

Proposition (\ref{diagonal}) 2) reduces to 

\begin{Proposition}\label{torsion} For any $c>0$, the natural map
\begin{equation}\label{sdvignat}
R\hom(\gf_{[0,\infty)}[-\dim \h];i^{-1}u_\O[-\dim G])\to 
R\hom(\gf_{[0,\infty)}[-\dim \h];
i^{-1}T_cu_\O[-\dim G])
\end{equation}
is non-zero
\end{Proposition}

\subsubsection{} Let $I:\h \to G\times \h$ be given by
$I(A)=(e,A)$.  Let $\bfS_e:=I^{!}\sp=I^{-1}\sp[-\dim G]$. We then have
\begin{equation}\label{bfS-uO}
i^{-1}u_\O=I_0^{-1}(\bfS_e*_\h \gamma_L)[\dim G]
\end{equation}

This equation dictates us to find an explicit expression for
$\bfS_e$.  It turns out to be more convenient to work with a 
slightly
different object. Namely, let $\Zentrum\subset G$ be the center
of $G$. Let $I_\Zentrum:\Zentrum\times \h\to G\times \h$
be the obvious embedding. Set $\bfS:=I_\Zentrum^!\sp=I_\Zentrum^{-1}\sp[-\dim G]$.
We will identify this object up-to an isomorphism. 

\subsection{Identifying $\bfS$} 
We will now give an explicit description
of  the object $\bfS$ up-to isomorphism. The proof of this
result will be given in the subsequent sections of the paper.
\subsubsection{ Object $\cY$}
We first define an object $\cY\in D(\Zentrum\times\h)$ 
as follows.  Let $\bL\subset \h$  be the lattice consisting of all $A\in \h$ such that $e^A\in \Zentrum$.

 For a subset $J\subset \{1,2,\ldots,N-1\}$ 
and $l\in \bL$ let $K(J,l)\subset e^l\times\h\subset \Zentrum\times 
h$ be defined as follows:
$$
K(J,l):=\{(e^{l},x)\in \Zentrum\times \h|\forall j\in J: 
<x-l,e_j>\geq 0\}.
$$
 Let $V(J,l):=\gf_{K(J,l)}[D(l)],
$
where $D(l)$ is an integer defined  in (\ref{deem}). That is, decompose $l=\sum l_ke_k$, where
$e^1,e^2,\ldots,e^n$ is a basis in $\h$ as in (\ref{ebasis}). Then $D(l)=\sum l_kD_k$, where
$D_k=k(N-k)$ and  $N=\dim \h +1$.
Let $\bL_J=\{l\in \bL|\forall i\notin J: <l,f_j>\leq 0\}$
Let
$\Psi^J:=\bigoplus\limits_{l\in \bL_J} V(J,l).$

Let $J_1\subset J_2$. Construct a map 
$$
I_{J_1J_2}:\Psi^{J_1}\to \Psi^{J_2}.
$$
It is defined as the  direct sum of the natural maps
$$
V(J_1,l)\to V(J_2,l)
$$
for all $l\in \bL_{J_1}\subset \bL_{J_2}$. These maps
come from the closed embeddings  $K(J_2,l)\subset K(J_1,l)$.

Let $\Subsets$ be the poset (hence the category) of all subsets
of $\{1,2,\ldots,N-1\}$. We then see that $\Psi$ is a functor
from $\Subsets$ to the category of sheaves on
 $\Zentrum\times \h$.  We then construct the standard complex
$\cY^\bullet$ such that
\begin{equation}
\cY^k:=\bigoplus\limits_{I,|I|=k} \Psi^I
\end{equation}
and the differential $d_k:\cY^k\to \cY^{k+1}$ is given by
\begin{equation}
d_k=\sum (-1)^{\vs(J_1,J_2)}I _{J_1J_2},
\end{equation}
where the sum is taken over all pairs $J_1\subset J_2$ such that
$|J_1|=k$ and $|J_2|=k+1$. The set $J_2\backslash J_1$ then consists
of a single element $e$ and $\vs(J_1J_2)$ is defined as the number
of  elements in $J_2$ which are less than $e$.

\subsubsection{Object $\bfS$}  Let $I\subset \{1,2,\ldots,N-1\}$ be a subset. 
Denote $e_I:=\sum\limits_{i\in I} e_i\in \h$.  Let also $G(I)$
be a graded vector space as in Lemma \ref{lemmagi}.

For any $l\in \h$, let $T_l:\Zentrum\times \h\to \Zentrum\times \h$ be the shift by $l$: $T_l(z,A):=(z,A+l)$

\begin{Theorem}\label{mainbfs}   We have
an isomoprhism
\begin{equation}\label{Y-bfS1}
\bfS\cong \bigoplus_I   G_I[D(-2\pi e_I)]\otimes T_{-2\pi e_I*}\cY,
\end{equation}
\end{Theorem}
Proof of this theorem is obtained as a result of a study of the obect $\bfS$ in Sec. \ref{specp}-
\ref{lastbfs}.

Given this description of  $\bfS$, we 
can now compute $i^{-1}u_\O$.

\subsection{Computing $i^{-1}u_\O$} Let $\O$ be the orbit
of $L\in \g^*$, where $L=\lambda e_1$, $\lambda>0$.
For each $z\in \Zentrum$,
let us define  objects $\cV_z\in D(\Re)$ by the formula:
\begin{equation}\label{defcv}
\cV_z:=\bigoplus\limits_{l\in \bL^z;\forall j\neq 1:<l,f_j>\leq 0}
\gf_{[<l,L>;\infty)}[D(l)-\dim \h],
\end{equation}
where $\bL^z:=\{l\in \bL|e^l=z\}$.
For every $d>0$ we have natural maps
$\tau_d:\cV_z\to T_{d*} \cV_z$, where $T_d$ is the shift by $d$. The map $\tau_d$ is induced by the obvious maps

$$
\gf_{[<l,L>;\infty)}\to \gf_{[<l,L>+d;\infty)}=
T_{d*}\gf_{[<l,L>,\infty)}.
$$
\begin{Theorem} 1)  We have an isomrophism
\begin{equation}\label{theoruo}
i^{-1}u_\O\cong \bigoplus\limits_I G_I[D(-2\pi e_I)]
\otimes T_{<-2\pi e_I,L>*}\cV_{e^{2 \pi e_I}}[\dim G]
\end{equation}

2) The natural map $i^{-1}u_\O\to i^{-1}T_{d*}u_\O$ 
 is induced
by the maps $\tau_d$.
\end{Theorem}
\begin{proof}
Let $\bL^c=\{l\in L;e^{l}=c\}$.
Let $\bL^c_J=\bL^c\cap \bL_J$. Let $\cY_c\in D(\h)$; 
$\cY_c=\cY|_{c\times \h}$.

It follows from  (\ref{Y-bfS1}) and  (\ref{bfS-uO})
that we have an isomorphism
$$
i^{-1}u_\O\cong \bigoplus_I G_I[D(-2\pi e_I)]\otimes
 I_0^{-1}(T_{-2\pi e_I*}\cY|_{e\times \h}*_\h\gamma_L)[\dim G].
$$
 Let $\cU_z:=I_0^{-1}\cY|_{z\times \h}*_\h \gamma_L$. We then have
$$I_0^{-1}[T_{-2\pi e_I*}\cY|_{e\times \h}*_\h\gamma_l]=
I_0^{-1}[\cY_{e^{2\pi e_I}}*_\h T_{-2\pi e_I*}\gamma_l]
$$
$$
=
I_0^{-1}[\cY_{e^{2\pi e_I}}*_\h T_{<-2\pi e_I,L>*}\gamma_L]
$$
$$
=T_{<-2\pi e_I,L>*}\cU_{e^{2\pi e_I}},
$$ 
where  for a real number $t$, we define a map $T_t:G\times \Re\to G\times \Re$ to be the shift along $\Re$ by $t$, whereas for $A\in \h$, $T_A$ is the shift
by $A$ along $\h$ in $G\times \h$.

We then have an isomorphism
\begin{equation}\label{V-uO}
i^{-1}u_\O\cong \bigoplus\limits_I  G_I[D(-2\pi e_I)]\otimes 
T_{<-2\pi e_I,L>*}\cU_{e^{2\pi e_I}}[\dim G]
\end{equation}

One also sees that the natural map
$$
i^{-1}u_\O\to i^{-1}T_{d*}u_\O
$$
for $d>0$ corresponds under this isomorphism to the natural map
induced by the maps
\begin{equation}\label{taud}
 \tau_d:\cU_c\to T_{d*}\cU_c,
\end{equation}
in  turn induced by the 
natural map $\gamma_L\to T_{d*}\gamma_L$
coming from the embedding $$\{(t,A)|t\geq -<A,L>+d\}\subset
\{(t,A)|t\geq -<A,L>\}$$ (we have
$\gamma_L=\gf_{\{t\geq -<A,L>\}}$ and $T_{d*}\gamma_L=
\gf_{\{(t,A)|t\geq -<A,L>+d\}})$).

Let us compute $\cU_z$ for $z\in \Zentrum$. We will actually
see that $\cU_z\cong \cV_z$.

\begin{Lemma} 
 We have $I_0^{-1}((V(J,l)|_{e^{l}\times \h})*_\h \gamma_L)=0$ for all
$J\neq \{1\}$.
\end{Lemma}
\begin{proof} Let $V'(J,l):=V(J,l)|_{e^{l}\times \h}$.

 We have $\gamma_L=\gf_{\{(A,t)| t+<A,L>\geq 0\}}$.
The inequality $t+<A,L>\geq 0$ is equivalent to  $t/\lambda+<A,e_1>\geq 0$.  Set $T=t/\lambda$. 
Then our statement can be reformulated as:
$$
V'(J,l)*_\h \gf_{\{(A,T)|T+<A,e_1>\geq 0\}}=
RP_!((V'(J,l)\boxtimes \gf_\Re)\otimes \gf_{\{(A,T)|T\geq <A,e_1>\}})=0,
$$
where  $P:\h\times \Re\to \Re$ is the projection. This is equivalent
to showing that for any $T\in \Re$,
$$
R\Gamma_c(\h;V'(J,l)\otimes \gf_{\{A\in \h| T\geq (A,e_1)\}})=0.
$$
 Let $x_j:\h\to \Re$; $x_j=<A,e_j>$.
We then have 
$$
V'(J,l)\otimes \gf_{\{A\in \h| T\geq <A,e_1>\}}=\gf_{S}[D(l)],
$$
where $S=\{A\in\h| x_1(A)\leq T; \forall j\in J: x_j(A)\geq x_j(l)\}$.

Suppose  there exists $j\in J$, $j\neq 1$.  Decompose
$\h=\Re.f_j\times  E$, where $E$ is the
 span of all $f_i$, $i\neq j$ (recall that $f_j$ form  the basis dual to $e_1,e_2,\ldots,e_{N-1}$).
Thus, $\h=\Re\times E$. Then $\gf_{S}[D(l)]=\gf_{[0,\infty)}\boxtimes A$ for some $A\in D(E)$. 
Let $\pi:\h\to E$ be the projection. Then 
$R\pi_!\gf_S[D(l)]=0$ because $R\Gamma_c(\Re,\gf_{[0,\infty)})=0$.
If $J=\emptyset$, then $S=\{A\in \h|x_1(A)\leq T\}$. It is 
easy to see that $R\Gamma_c(\h,\gf_S[D(l)])=0$. 
This exhausts all subsets $J\neq \{1\}$.
\end{proof}
It now follows that $I_0^{-1}(\Psi^J*_\h \gamma_L)=0$
for all $J\neq \{1\}$  Therefore,  we have an isomorphism
$$
\cU_z=I_0^{-1}(\Phi_z*_\h \gamma_L)[\dim G]
\cong I_0^{-1}(\Psi^{\{1\}}_z*_\h \gamma_L)[-1][\dim G]
$$

$$
\cong\bigoplus\limits_{l\in \bL^z_{\{1\}}}
 I_0^{-1}[V(\{1\};l))_z*_\h\gamma_L][-1][\dim G],
$$
where the subscript $z$ hear and below means the restriction onto
$z\times \h\subset \Zentrum\times \h$.
Let us compute 
$$
 I_0^{-1}[V(\{1\};l)_z*_\h\gamma_L]= 
RP_{!}(\gf_{(A,t); x_1(A)\geq x_1(l)}\otimes
 \gf_{\{(A,t)|\lambda x_1(A)\leq t\}})[D(l)],
$$
where $P: \h\times \Re\to \Zentrum\times \Re$
is the projection.
We have
$$
RP_{!}(\gf_{\{(A,t); x_{1}(A)\geq x_1(l)\}}\otimes \gf_{\{(A,t)|
\lambda x_1(A)\leq t\}})
$$

$$
=RP_!(\gf_{\{(A,t); x_1(l)\leq x_1(A)\leq t/\lambda\}})=
\gf_{[\lambda x_1(l),\infty)}[1-\dim \h]
$$
Thus,
$$
 I_0^{-1}[V(\{1\};l)_c*_\h\gamma_L]\cong  \gf_{[\lambda x_1(l),\infty)}[1-\dim \h][D(l)]
$$
 
Let $d\geq 0$. We need to compute the map
$$
\tau_d: I_0^{-1}[V(\{1\};l)_c*_\h\gamma_L]\to T_{d*} I_0^{-1}[V(\{1\};l)_c*_\h\gamma_L]
$$
induced by the natural map
$$
\gamma_L\to T_{d*}\gamma_L.
$$

It is easy to see that the map $\tau_d$
is isomorphic to the natural map
$$
 \gf_{[\lambda x_1(l),\infty)}[1-\dim \h]\to T_{d*} \gf_{[\lambda x_1(l),\infty)}[1-\dim \h]
$$
$$
= \gf_{[\lambda x_1(l)+d,\infty)}[1-\dim \h],
$$
induced by the embedding

$$ [\lambda x_1(l)+d,\infty)\subset[\lambda x_1(l),\infty).$$
Thus, we have,
$$
\cU_z=\bigoplus\limits_{l\in \bL_{\{1\}}^z }
\gf_{[\lambda x_1(l),\infty)\}}[D(l)][-\dim \h]
$$
$$
=\bigoplus\limits_{l\in \bL^z; \forall j\neq 1:<l,f_j>\; \leq 0}
\gf_{[\lambda <l,e_1>,\infty)}[D(l)-\dim \h].
$$

Thus, we see that $\cU_z\cong \cV_z$.  It is now 
 straightforward to check that the maps $\tau_d$ on both sides 
do match
\end{proof}

Let us substitute (\ref{defcv}) into (\ref{theoruo}).
We will get
$$
i^{-1}u_\O\cong \bigoplus\limits_I G(I)\otimes \upsilon(I)[-\dim \h+\dim G],
$$
where
$$
\upsilon(I)=\bigoplus\limits_{l\in \bL^{{e^{2\pi e_I}}};\forall j\neq 1:<l,f_j>\leq 0} \gf_{[<l-2\pi e_I,L>;\infty)}[D(l-2\pi e_I)].
$$

Let us replace $l$ with $l+2\pi e_I$. We will get  an ultimate formula
\begin{equation}\label{upsilon}
\upsilon_I=\bigoplus\limits_{l\in \bL^0;\forall j\neq 1:<l+2\pi e_I,f_j>\leq 0}\gf_{[<l,L>;\infty)}[D(l)].
\end{equation}

The map $\tau_d:i^{-1}u_\O\to T_{d*}i^{-1}u_\O$, $d\leq 0$ is induced
by natural maps $\tau_d:\upsilon_I\to T_{d*}\upsilon_I$ which are produced
by the embeddings $T_d[<l,L>;\infty))\subset [<l,L>;\infty).$
\subsubsection{Proof of Proposition \ref{torsion} }
We have
$$
R\hom(\gf_{[0,\infty)}[-\dim \h]; T_{d*}i^{-1}u_\O[-\dim G])=
\bigoplus\limits_I G(I)\otimes H_I(d),
$$ 
where
$$
H_I(d):=R\hom(\gf_{[0,\infty)};T_{d*}\upsilon_I)
\cong
\bigoplus\limits_{l\in S_I(d)}\gf[D(l)],
$$
and
$$S_I(d):=\{l\in L^0|\forall j\neq 1:<l+2\pi e_I,f_j>\leq 0; <l,L>+d\geq 0\}.
$$

The map  (\ref{sdvignat})  is induced by  maps
$\tau_d: H_I(0)\to H_i(d),
$
which are in turn induced by the maps $\tau_d:\upsilon\to T_{d*}\upsilon$.
It is not hard to see that the map $\tau_d:H_I(0)\to H_I(d)$ is
induced by the inclusion $S_I(0)\subset S_I(d)$. As $S_I(0)$ is not empty,
the maps $\tau_d:H_I(0)\to H_I(d)$ do not vanish for any $d\geq 0$,
which proves the Proposition.

\section{An object $\sp$}\label{specp}We will freely use notations 
from Sec. \ref{gnot}.

The object $\sp$  will
be characterized microlocally. Let us first define a subset
\begin{equation}\label{omegasps}
\Omega_\sp\in T^*(G\times \h)
\end{equation}
 which will serve as a microsupport
of $\sp$. Define $\Omega_\sp$ as a set of all points 
$$(g,A,\omega,\eta)\in G\times \h\times \g\times \h=T^*(G\times \h)$$
satisfying: 

1)$ g(V_k(\omega))\subset V_k(\omega)$, that is $\Ad_g\omega=\omega$;

2) $\det g|_{V_k(\omega)}=e^{-i<e_k,A>}$;

3) $\eta=\|\omega\|$.
The notation $V_k(\omega)$ is defined in the beginning of Sec. \ref{gnot},
see (\ref{flagomega}).

Finally, let us denote for $A\in \h$, $I_A:G\to G\times \h$ the embedding
$I_A(g)=(g,A)$. 

We now formulate

\begin{Theorem}\label{specth}  There exists an object $\sp\in D(G\times \h)$
such that 

1) $\mS(\sp)\subset \Omega_\sp$; 

2) $I_0^{-1}\sp=\gf_{e_G}$.

\end{Theorem} 

\subsection{Proof of Theorem \ref{specth}}\label{proofspecth}
\subsubsection{} Let $U_1,U_2\subset \h$ be open convex sets.
Let $a:\h\times \h\to \h$ be addition. The map $a$
induces a map $U_1\times U_2\to U_1+U_2$ which is well known to
be a trivial  smooth fibration whose fiber and base
 are diffeomorphic to $\h$.

Let  $F_k \in D(G\times U_k)$, $k=1,2$. Let $M:G\times U_1\times
G\times U_2\to G\times U_1\times U_2$ be the map induced by the
product on $G$. Set $F_1*_G F_2:=RM_!(F_1\boxtimes F_2)$.

Let $a:G\times U_1\times U_2\to G\times (U_1+U_2)$ be induced by the addition
on $\h$.

\begin{Lemma}\label{specnositel}
 Suppose that $\mS(F_k)\subset \Omega_\sp\cap T^*(G\times U_k)$.
Then
1) The natural map
$$
a^{-1}Ra_*(F_1*_G F_2)\to F_1*_G F_2
$$
is an isomorphism;

2)  $\mS(Ra_*(F_1*_G F_2))\subset \Omega_\sp\cap 
T^*(G\times (U_1+U_2))$.
\end{Lemma}
\begin{proof}

Let us first estimate the microsupport of $F_1*_G F_2=
RM_!(F_1\boxtimes F_2)$.  Since the map $M$ is proper, we know that
a point
$$
\zeta:=(g,A_1,A_2,\omega,\eta_1,\eta_2)\in G\times U_1\times U_2\times
\g^*\times \h^*\times \h^*=T^*(G\times U_1\times U_2)
$$
belongs to $\mS RM_!(F_1\boxtimes F_2)$ only if there exist
$g_1,g_2\in G$ such that $M(g_1,A_1,g_2,A_2)=(g,A_1,A_2)$
(i.e. $g=g_1g_2$) and 
$$
M^*\zeta|_{(g_1,A_1,g_2,A_2)}\in \mS(F_1\boxtimes F_2).
$$
We have
$$
M^*\zeta|_{(g_1,A_1,g_2,A_2)}=
(g_1,A_1,\omega,\eta_1,g_2,A_2,\Ad_{g_1}^*\omega,\eta_2).
$$

We then have $(g_1,A_1,\omega,\eta_1),(g_2,A_2,\Ad^*_{g_1}\omega,\eta_2)\in \Omega_\sp$. Therefore, $\Ad^*_{g_1}\omega=\omega$, and we have
$$
(g_k,A_k,\omega,\eta_k)\in \Omega_\sp.
$$
This implies $\eta_1=\eta_2=\|\omega\|$. This means that
any 1-form in $\mS(RM_!(F_1\boxtimes F_2))$ vanishes on
fibers  of $a$. This proves part 1). 

Let us now estimate $\mS Ra_*(F_1*_G F_2)$. We know
that $\zeta\in \mS Ra_*(F_1*_G F_2)$, where $\zeta\in T^*_{(g,A)}
(G\times (U_1+U_2))$, iff 
for every point $(g,A_1,A_2)\in G\times U_1\times U_2$ such that
$A_1+A_2=A$, we have
$$
 a^*\zeta|_{(g,A_1,A_2)}\in \mS(a^{-1}Ra_*(F_1*_G F_2)).
$$

Let $\zeta=(g,A,\omega,\eta)$, then 
$a^*\zeta|_{(g,A_1,A_2)}=(g,A_1,A_2,\omega,\eta,\eta)$.
Using the  isomorphism $ a^{-1}Ra_*(F_1*_G F_2)\to F_1*_G F_2$.
and the above estimate for $\mS(F_1*_G F_2)$, we get:
there exist $g_1,g_2\in G$ such that $g=g_1g_2$ and
$$
(g_k,A_k,\omega,\eta)\in \Omega_{\sp}.
$$

It remains to show that $(g_1g_2,A_1+A_2,\omega,\eta)\in\Omega_\sp$.
Indeed, we have $\eta=\|\omega\|$. Next, $Ad^*_{g_k}\omega=\omega$,
therefore, $\Ad^*_{g_1g_2}\omega=\omega$. 

Finally,
$$
\det g_1g_2|_{V_k(\omega)}=\det g_1|_{V_k(\omega)}
\det g_2|_{V_k(\omega)}
$$
$$
=e^{-i<A_1,e_{d_k(\omega)}>}e^{-i<A_2,e_{d_k(\omega)}>}
$$
$$
e^{-i<A_1+A_2,e_{d_k(\omega)}>}.
$$
\end{proof}

\subsubsection{} Let $b\in C_+^\circ$; $b\leq e_1/100$.
Let $V_b^{-}:=\{A\in C_-^\circ| -A<<b\}$, where $C_+^\circ$ is the interior
of the positive Weyl chamber and $C_-^\circ=-C_+^\circ$,  see Sec. \ref{gnot}.
let $W_b^{-}\subset G\times V_b^{-}$;
$$
W_b^{-}:=\{(e^X,A); A\in V_b^{-}; \|X\|<<-A\}.
$$
Set $F^{-}\in D(G\times V_b^-)$; 
$$
F^{-}:=\gf_{W_b^{-}}[\dim G].
$$

\subsubsection{}
We will identify $TG=G\times \g$; $T^*G=G\times \g^*=
G\times \g$ via identifying $\g$ with the space of all right invariant
vector fields on $G$ and $\g^*=\g$ with the space of all
right invariant 1-forms on $G$.
Analogously, we will identify $T(G\times \h)=G\times \h\times \g\times \h$ and $T^*(G\times \h)=G\times \h\times \g^*\times \h^*=
=G\times \h\times \g\times \h$.

\begin{Lemma}\label{supportF-} The microsupport of  $F^{-}$ is contained in the
set of all points 
$(e^X,A,\omega,\eta)\in G\times V_b^{-1}\times\g^*\times \h^*$, where

1) $\|X\|\leq -A$;

2) $[X,\omega]=0$;

3) $\Tr X|_{V_{k}(\omega)}=-i<A,e_{d_k}>$;

4) $\eta=\|\omega\|$
\end{Lemma}
\begin{proof}

Let $U\subset \g\times V_b^{-} $; $U=\{(X,A)|\|X\|<<-A\}$. 
Let $$\exp:\g\times V_b^{-1}\to G\times V_b^{-1}$$ 
be the exponential map. We see that $\exp$ maps $U$ diffeomorphically onto 
$W_b^{-}$, hence we have an induced diffeomorhphism
$\exp:T^*U\to T^*W_b^-$.
It also follows that  $F^-=\exp_* \gf_{U}[\dim G]$ and
that $\mS(F^-)=\exp(\mS\gf_U)$.

Let us estimate $\mS(\gf_U)$. $U\subset \g\times V_b^{-}$ is an 
open convex subset. It follows that a point
$(X,A,\omega,\eta)\in \g\times V_b^-\times \g^*\times \h^*$
is in the microsupport of $\gf_U$ iff 
1) $\|X\|\leq -A$;

2) for all $(X',A')\in U$, 
$<X',\omega>+<A',\eta>\;< \;<X,\omega>+<A,\eta>$.

Fix $A'$, then $X'\in \g$ is an arbitrary element
such that $\|X'\|<<-A$. 
Lemma \ref{ineq} implies  that
$$
\sup <X',\omega> =<-A';\|\omega\|>
$$
Thus, Condition 2) is equivalent to
\begin{equation}\label{inegalite}
<-A',\|\omega\|>+<A',\eta>\leq \; <X,\omega> +<A,\eta>
\end{equation}
for all $A'\in V_b^{-}$. Plug $A'=A$. We will get
$$
<-A,\|\omega\|>\leq <X,\omega>.
$$
On the other hand $<X,\omega>\leq <\|X\|,\|\omega\|>\leq
<-A,\|\omega\|>$. This implies that
\begin{equation}\label{egalite}
<-A,\|\omega\|>=<X,\omega>.
\end{equation}
According to Lemma \ref{ineq},
for all $k$,
$$
\Tr X|_{V_k(\omega)}= -i<A,e_{d_k(\omega)}>.
$$

Let us plug (\ref{egalite}) into (\ref{inegalite}).
We will get
$$
<-A',\|\omega\|>+<A',\eta>\leq \;<-A,|\omega|> +<A,\eta>
$$
for all $A'\in V_b^-$. As $A\in V_b^-$ and $V_b^-$ is open, 
this is only possible if $\eta=\|\omega\|$.
\end{proof} 

\begin{Corollary} We have $\mS(F^-)\subset \Omega_\sp\cap
T^*(G\times V_b^-)$.
\end{Corollary}
\subsubsection{}  Let $U\subset G\times V_b^-\times V_b^-$
be given by
$$
U:=\{(e^X,A_1,A_2)| \|X\|<<-A_1-A_2\}
$$
\begin{Lemma}\label{dvoinaya} We have an isomorphism
$$
F^-*_G F^-\cong \gf_{U}[\dim G].
$$
\end{Lemma}
\begin{proof}  Let $j_{U_1}:U_1\into  G\times V_b^-\times V_b^-$ be an open
set defined by $U_1=M(W_b^-\times W_b^-)$. 
 It follows that
we have an isomorphism
$$
j_{U_1!}((F^-*_G F^-)|_{U_1})\to F^-*_G F^-.
$$

We have
 $$
U_1=\{(e^{X_1}e^{X_2},A_1,A_2)| A_k\in V_b^-; \|X_k\|<<-A_k\}.
$$
According to Lemma \ref{Klyachko}, we have
$e^{X_1}e^{X_2}=e^Y$, where $\|Y\|\leq \|X_1\|+\|X_2\|<<
-A_1-A_2$. Thus
$
U_1\subset U.
$

Let $j_U:U\to G\times V_b^-\times V_b^-$ be the open embedding.
We then have an isomorphism
\begin{equation}\label{ogranichenienaU}
j_{U!}((F^-*_G F^-)|_U)\to F^-*_G F^-
\end{equation}

Let us now study $F^-*_G F^-|_U$.  Let us estimate the microsupport
of this object. Similar to proof of Lemma \ref{specnositel},
we see that a point 
\begin{equation}\label{tochkanositelya}
(g,A_1,A_2,\omega,\eta_1,\eta_2)\in
G\times V_b^-\times V_b^-\times \g^*\times \h^*\times \h^*=
T^*(G\times V_b^-\times V_b^-)
\end{equation}
 is in $\mS(F^-*_G F^-|_U)$
iff

1) $(g,A_1,A_2)\in U$;

2) there exist $X_1,X_2\in \g$ such that $g=e^{X_1}e^{X_2}$ 
and $(e^{X_k},A_k,\omega,\eta_k)\in \mS(F^-)$ for $k=1,2$.

According to Lemma \ref{supportF-}, we have
$$
\|X_k\|\leq -A_k;
$$
$$
\Tr X_k|_{V_l(\omega)}=-i<A_l,e_{d_l}>.
$$

Hence, $e^Y=e^{X_1}e^{X_2}$ preserves
the spaces $V_l(\omega)$. As $\|Y\|<<\|X_1\|+\|X_2\|\leq e_1/(50N)$,
it follows that all eigenvalues of $-iY$ have absolute value
of  less than $1/(50N)$. It then follows that $Y$ does preserve
the spaces $V_l(\omega)$ as well, and $\Tr Y|_{V_l(\omega)}$
has absolute value of at most $1/50$.

We also have
$$
\det e^Y|_{V_l(\omega)}=\det e^{X_1}|_{V_l(\omega)}e^{X_2}|_{V_l(\omega)}=e^{-i<A_1+A_2,e_{d_l(\omega)}>}.
$$
We have $|<A_1+A_2,e_{d_l(\omega)}>|\leq 1/50$, therefore,
\begin{equation}\label{nagranicu}
\Tr Y|_{V_l(\omega)}=-i<A_1+A_2,e_{d_l(\omega)}>.
\end{equation}

Assume $\omega\neq 0$. Then there exists a  subspace
$V_l(\omega)$ which is proper, i.e. $0<d_l(\omega)<N$.
On the other hand,  we have $(e^Y,A_1,A_2)\in U$, meaning
that $e^Y=e^{Y'}$, where $\|Y'\|<<-A_1-A_2$.  We then have
$\|Y\|,\|Y'\|<e_1/(50N)$ which implies $Y=Y'$ and
$\|Y\|<<A_1+A_2$. This clearly contradicts to (\ref{nagranicu}).
Therefore, it is impossible that $\omega\neq 0$, hence $\omega=0$.
It then follows that in (\ref{tochkanositelya}), $\eta_1=\eta_2=
\|\omega\|=0$.  Thus, we have proven that
$F^-*_G F^-|_U$ is microsupported on the zero-section, hence is
 locally constant. However, under the exponential map, $U$ is a
diffeomorphic image of an open convex set
$\{(X,A_1,A_2)| A_k\in V_b^{-}; \|X\|<<-A_1+A_2\}\subset
\g\times V_b^-\times V_b^-$. Therefore, $U$ is diffeomorphic
to $\Re^{\dim U}$ and $F^-*_G F^-$ is constant on $U$.

Let $Z:=R\Gamma_c(U; F^-*_G F^-)$. We then have 
a natural isomorphism  $F^-*_G F^-|_U \cong Z_U[\dim U]$.
Because of an isomorphism (\ref{ogranichenienaU}), we have
an induced isomorphism
$$
R\Gamma_c(U; F^-*_G F^-)\to R\Gamma_c(G\times V_b^-\times V_b^-;
F^-*_G F^-)
$$
$$
\cong R\Gamma_c(G\times V_b;F^-)\otimes R\Gamma_c(G\times
V_b;F^-)
$$
$$
\cong \gf[-\dim G\times V_b^-]\otimes \gf[-\dim G\times
 V_b^-][2\dim G]=\gf[-\dim U+\dim G].
$$
This implies the statement.
\end{proof}

Let $a:G\times V_b^-\times V_b^-\to G\times 2V_b^-$
be the addition map. The just proven  Lemma as well as
Lemma \ref{specnositel} imply that the natural map
$a^{-1}Ra_*(F^-*_G F^-)\to F^-*_G F^-$ is an isomorphism
and that 
$$
Ra_*(F^-*_G F^-)\cong
 \gf_{\{(e^X,A)| A\in 2V_b^-; \|X\|<<-A\}}[\dim G].
$$
We then have an induced isomophism
\begin{equation}
\label{double-isomorphism}
\iota: Ra_*F^-*_G F^-|_{G\times V_b^-}\cong F^-.
\end{equation}
\subsubsection{} Let $M>0$ and let $F^-_M\in D(G\times (V_b^-)^M)$;
$$
F^-_M:=F^-*_GF^-*_G\cdots *_G F^-,
$$
where $F^-$ occurs  $M$ times.

Let $a_M:G\times (V_b^-)^M)\to G\times MV_b^-$ be the addition map.
Lemma \ref{specnositel} implies that the natural map
$$
a_M^{-1}Ra_{M*}F_M^-\to F_M^-
$$
is an isomorphism. 

Let $\Phi_M^-:=Ra_{M*}F_M^-$. 

Let us construct a map 
$$
I_M:\Phi_M^-|_{G\times (M-1)V_b^-}\to \Phi_{M-1}^-,
$$
where $M\geq 2$,
as follows.

Let $W\subset (V_b^-)^2$ be an open convex subset 
consisting of all points of the form $(v_1,v_2)$,
where $v_{1}+v_{2}\in V_b^-$.
Let $W_M:=(V_b^{-})^{M-2}\times W\subset (V_b^-)^M$.

Let us decompose
$$
\alpha_M:=a_M|_{G\times W_M}:G\times W_M=G\times (V_b^-)^{M-2}
\times W
\stackrel{a_2}\to G\times (V_b^-)^{M-2}\times V_b^-
$$
$$
\stackrel{a_{M-1}}\to(M-1)V_b^{-}.
$$
It follows that $\alpha_M(W_M)=G\times (M-1)V_b^-$.
We have a natural isomorphism
$$
Ra_{M*}F^-_M|_{G\times (V_b^-)^{M-1}}=R\alpha_{M*} F^-_M|_{G\times W_M}\cong 
Ra_{M-1*}Ra_{2*}F^-_M|_{G\times W_M}.
$$

We have 
$$
Ra_{2*}F^-_M|_{G\times W_M}\cong F^-_{M-2}*_G(R\alpha_*F^-*_GF^-|_{W}),
$$
where $\alpha:G\times W\to G\times V_b^-$ is the addition
map. 
We have an isomorphism (see (\ref{double-isomorphism}))
$$
R\alpha_*(F^-*_GF^-|_{W})\cong 
(Ra_*(F^-*_G F^-*))|_{V_b^{-}}\stackrel\iota\cong F^-.
$$

Hence, we have  isomoprhisms
$$
Ra_{2*}(F^-_M|_{G\times W_M})\cong F^-_{M-1}
$$
$$
I_M:Ra_{M*}F^-_M|_{G\times (V_b^-)^{M-1}}\cong Ra_{M-1*}F^-_{M-1}.
$$

Thus, we have objects $\Phi^-_M\in D(G\times MV_b^-)$
and isomorhhisms
$$ I_M:\Phi^-_M|_{(M-1)V_b^-}\to \Phi^-_{M-1}.
$$
It then follows that there exists 
an object
$\Phi^-\in D(G\times C_-^\circ)$ (note that $C_-^\circ=\bigcup_M
MV_b^-$) along with isomorphisms
$$
J_M:\Phi^-|_{MV_b^-}\to \Phi^-_M
$$
which are compatible with $I_M$ in the obvious way.

Let $\Psi^-\in D(G\times C_-^\circ)$ be another object
endowed with isomorphisms $J'_{M}:\Psi^-|_{MV_b^-}\to \Phi^-_M$
so that $J'_M$ are compatible with $I_M$. Then there exists
a (non-canonical) isomorphism $\Phi^-\to \Psi^-$ which is compatible
with the isomorphisms $J_M,J'_M$.

Lemma \ref{specnositel} implies that
$\mS(\Phi_M^-)\subset \Omega_\sp\cap T^*(G\times MV_b^-)$.
Therefore,
 $$
\mS(\Phi^-)\subset \Omega_\sp\cap T^*(G\times C_-^\circ).
$$
\subsubsection{}
Lemma \ref{specnositel} implies that
we have an isomorphism
$$
A^{-1}RA_*(\Phi^-*_G\Phi^-)\to \Phi^-*_G\Phi^-
$$
where $A:G\times C_-^\circ\times C_-^\circ\to G\times C_-^\circ$
is the addition.  

Let us restrict this isomorphism to $G\times MV_b^-\times V_b^-$.
We will then get an isomoprhism
$$
A^{-1}(RA_*\Phi^-*_G\Phi^-|_{(M+1)V_b^-})\to \Phi^-_M*_G F^-=
A^{-1}\Phi^-_{M+1}.
$$ 

Thus, we have an isomorphism
$$
J'_{M+1}:RA_*\Phi^-*_G \Phi^-|_{(M+1)V_b^-}\cong \Phi_{M+1}^-
$$
One can easily check that these isomoprhisms are compatible
with  $I_M$ hence,  there exists  an isomorphism
$$
J:RA_*(\Phi^-*_G\Phi^-)\cong \Phi^-
$$
which is compatible with isomorphisms $J_M,J'_M$.
Therefore, we  have an isomorphism
\begin{equation}\label{isoilneiso}
I:\Phi^-*_G\Phi^-\cong  A^{-1}\Phi^-.
\end{equation}
\subsubsection{}  Let $X^\pmin\in D(G)$, 
$X^+:=\gf_{\{e^{-X}|\|X\|\leq b/2 \}}$; 
$X^-:=\gf_{\{e^X|\|X\|<<b/2\}}[\dim G]$. 
We have an isomorphism $X^-\cong \Phi^-|_{G\times (-b/2)}$.
\begin{Lemma} We have an isomorphism  $X^-*_G X^{+}\cong \gf_e$.
\end{Lemma}
\begin{proof}
Let us first  compute the microsupport of $X^-*_G X^+$.
We will prove the following: {\em $\mS(X^-*_G X^+)$
is contained in the set of all points of the form $(e^Y,\omega)\in
G\times \g^*$, where $Y\in \g$; $\|Y\|\leq (e_1+e_{N-1})/200$,
$[Y,\omega]=0$
 and $<Y,\omega>=0$.}

Let us first estimate $\mS(X^-)$. 
Let $\exp:\g\times G$ be the exponential map.
We then see that  $X^-=\exp_*\gf_{U}[\dim G]$, where $U\subset \g$
is an open convex subset $U=\{X| \|X\|<<b/2\}$. 

We know that $\mS(\gf_{U})$ consists of all points 
of the form $(X,\omega)\subset \g\times \g^*$, where 
 $\|X\|\leq b/2$ and 
$$
<X',\omega>\;<\;<X,\omega>
$$
for all $X'<<b/2$.  Lemma \ref{ineq} implies that this is equivalent
to $<b/2,\|\omega\|>\leq <X,\omega>$ 
(because $$\sup_{X'<<b/2}<X',\omega>=<b/2,\|\omega\|>);$$
 on the other hand
$$
<X,\omega>\leq < \|b/2\|,\|\omega\|>
$$
by the same Lemma \ref{ineq}. Therefore $<X,\omega>=
<\|b/2\|,\|\omega\|>$. As $\|X\|\leq b/2$ this implies 
$[X,\omega]=0$; 
\begin{equation}\label{sledXminus}
\Tr X|_{V_k(\omega)}=i<b,e_{d_k(\omega)}>/2.
\end{equation}

As $[X,\omega]=0$, we see that
$
\mS(X^-) 
$
consists of all points $(e^X,\omega)$, where $\|X\|\leq b/2$ and
 $[X,\omega]=0$
and we have (\ref{sledXminus}).

Analogously,  $X^+=\exp_*\gf_{K}$, where $K\subset\g$ is a convex compact $K=\{X|\|-X\|\leq b/2\}$. Therefore,
$\mS(\gf_K)
$
consists of all points $(X_1,\omega_1)$, where 
$\|-X_1\|\leq b/2$ and $<X',\omega_1>\geq <X_1,\omega_1>$ for  
all $X'\in K$. I.e. $<-X',\omega_1>\leq <-X_1,\omega_1>$. In the
 same way as above, we conclude that this is equivalent
to $<-X_1,\omega_1>=<b/2,\|\omega_1\|>$ which in turn is equivalent
to $\|-X_1\|\leq b/2$; $[X_1,\omega_1]=0$; 
\begin{equation}\label{sledXplus}
\Tr(-X_1)|_{V_k(\omega_1)}=
i<b/2,e_{d_k(\omega_1)}>.
\end{equation}

Thus, $\mS(X^+)$ consists of all points of the form
$(e^{X_1},\omega_1)$, where 
$[X_1,\omega_1]=0$; $\|-X_1\|\leq b/2$ and (\ref{sledXplus})
is the case. Observe that we have $\|-X_1\|\leq e_1/200$ which
means $\|X_1\|\leq e_{N-1}/200$.

We know that  the microsupport of
$X^-*_G X^+=Rm_!(X^-\boxtimes X^+)$  is contained in the set
of all points of the form $(g_1g_2,\omega)$ where
$g_1,g_2\in G$; $(g_1,\omega)\in \mS(X^-)$;
 $(g_2,\Ad^*_{g_1}\omega)\in \mS(X^+)$. This means
that $\mS(X^-*_G X^+)$ consists of all points of the form
$$
(e^{X}e^{X_1},\omega),
$$
where $(e^X,\omega)\in \mS(X^-)$ 
and $(e^{X_1},\omega)\in \mS(X^+)$ 
(because $[X,\omega]=[X_1,\omega]=0$).
According to Lemma \ref{Klyachko},  $e^Xe^{X_1}=e^Y$, where
$\|Y\|\leq \|X\|+\|X_1\|\leq (e_1+e_{N-1})/200$.  It follows that $e^YV_k(\omega)\subset V_k(\omega)$ and 
$$
\det e^Y|_{V_k(\omega)}=e^{i<b/2-b/2,\|\omega\|>}=1,
$$
see (\ref{sledXminus}), (\ref{sledXplus}). As $\|Y\|\leq b$, this
implies $\Tr Y|_{V_k(\omega)}=0$. This in turn implies
that $<Y,\omega>=0$, which we wanted.

Let $c:=(e_1+e_{N-1})/200$. Let $W:=\{X\in g;\|X\|<<2c\}$.
The exponential map gives rise to an open empedding
$\exp:W\to G$. The object $X^-*_G X^+$ is supported 
within $\exp(W)$. Consider $E\in D(W)$; 
$E:=\exp^{-1}(X^-*_G)
X^+)$.  It suffices to show that $E\cong \gf_0$. 

We see that $E$ is microsupported within  the set $(X,\omega)$,
where $\|X\|\leq c$,  $[\omega,X]=0$, $(\omega,X)=0$.
Let $D$ be the dilation vector field on $\g$. That is
$D$ is a section of $T\g=\g\times \g$ ; $D:\g\to \g\times \g$;
$D(X)=(X,X)$. It then follows that every point $(x,\omega)\in
\mS(E)$ satisfies $i_{D_x}(\omega)=0$. Let
$X\in \g$; $X\neq 0$. Let $R_X:=(\Re_{>0}.X)\cap W$
be an open segment. It then follows that $E|_{R_X}$ is a constant sheaf. However, $R_X$ does necessarily contain points $Y\in R_X$
such that $\|Y\|>>c$, meaning that $E|_Y=0$, and $E|_{R_X}=0$.
Hence $E$ is supported at 0 and it suffices to show
that $E|_0\cong \gf$. 

We have $$E|_0=(X^-*_G X^+)|_e=R\Gamma_c(G; \gf_{\{e^X|\|X\|<<b/2\}}\otimes\gf_{\{e^X|\|X\|\leq b/2\}})[\dim G]
$$
$$
=R\Gamma_c(G;\gf_{\{e^X|\|X\|<<b/2\}})[\dim G]=\gf,
$$
because the open set $\{e^X|\|X\|<<b/2\}$ is diffeomorphic to
an open ball.
\end{proof}

\subsubsection{} Let $T:C_-^\circ\to C_-^\circ$ be the shift by $-b/2$.
$T(l)=l-b/2$.  

 Let us restrict
the isomorphism  (\ref{isoilneiso}) onto $G\times V_b^-\times (-b/2)$.
We will get an isomorphism
$$
\Phi^-*_G X^-\cong A^{-1}\Phi^{-}|_{G\times V_b^-\times (-b/2)}
$$
$$
=T^{-1}\Phi^-.
$$
Taking the convolution with $X^+$  and using the previous Lemma,
we will get an isomorphism
\begin{equation}\label{kleitel}
\Phi^-\cong (T^{-1}\Phi)*_G X^+.
\end{equation}

Let $\bT_M:(C_-^\circ +Mb/2)\to C_-^\circ$ be the shift by $-Mb/2$. 
Set
$$
 \Psi_M:=\bT_M^{-1} \Phi*_G (X^+)^{*_G^M},
$$
$$
\Psi_M\in D(G\times (C_-^\circ +Mb/2)).
$$ 
We have  an isomorphism
$$
i_M:\Psi_M|_{C_-^\circ +(M-1)b/2}\cong  \bT_{M-1}^{-1}[
(T^{-1}\Phi^{-}*_G X^+)*_G (X^+)^{*_G^{M-1}} ]
$$
$$
\cong
\bT_{M-1}^{-1}(
(\Phi^-*_G (X^+)^{*_G^{M-1}})=\Psi_{M-1}
$$
where on the last step we have used the isomorphism (\ref{kleitel}).
Similar to above, there exists an object
$\sp\in D(G\times \h)$ and isomorphisms
$\sp|_{C_- +Mb/2}\to \Psi_M$ which are compatible with
isomorphisms $i_M$. 
Lemma \ref{specnositel} readily implies that $\mS(\sp)\subset \Omega_\sp$. Let us compute $\sp|_{G\times 0}$. We have
$0\in C_-^\circ-b/2$. Therefore,
we have an isomorphism
$$
\sp|_{G\times 0}\cong \Psi^{-}_1|_{G\times 0}
\cong 
\Phi^-|_{G\times -b/2}*_G X^+\cong X^-*_G X^+\cong \gf_e.
$$
This proves that the object $\sp$ satisfies all the conditions of
Theorem \ref{specth}.

\subsubsection{Uniqueness}\begin{Theorem}\label{specuniq} Let $\sp_1,\sp_2$ satisfy the conditions
of Theorem \ref{specth}. Then $\sp_1$ and $\sp_2$ are canonically isomorphic.
\end{Theorem} 
\begin{proof}
According to Lemma \ref{specnositel} (take $U_1=U_2=\h$), we have an isomorphism
\begin{equation}\label{secad}
a^{-1}Ra_*(\sp_1*_G\sp_2)\to \sp_1*_G \sp_2.
\end{equation}

Let $I_1,I_2:\h\to \h\times \h$ be as follows: $I_1(A)=(A,0)$; $I_2(A)=(0,A)$. aApplying  functors $I_1^{-1},I_2^{-1}$ to (\ref{secad}) and taking into account 
the isomorphisms
 $I_0^{-1}\sp_i\cong \gf_{e_G}$, we will get the following isomorphisms
\begin{equation}\label{secad1}
Ra_*(\sp_1*_G \sp_2)\to \sp_2;\quad Ra_*(\sp_1*_G \sp_2)\to \sp_1,
\end{equation}
whence an isomorphism $\sp_1\to \sp_2$.
\end{proof}
From now on we will denote by $\sp$ any object satisfying Theorem \ref{specth} (they are all
canonically isomorphic). 

Equations (\ref{secad}), (\ref{secad1}) imply that we have an isomorphism
\begin{equation}\label{convsp}
\sp*_G \sp \to a^{-1}\sp.
\end{equation}
\def\Up{\Upsilon}
\subsubsection{} One can prove even more general result.
Let $\Upsilon\subset T^*(G\times \h)$ consist of all points
$(G,A,\omega,\|\omega\|)\in G\times \h\times \g^\times \h=T^*(G\times \h)$.
Of course, $\Omega_\sp\subset \Upsilon$.
Let $C_\Up\subset D(G\times \h)$ be the full subcategory consisting
of all objects $F$ microsupported on $\Up$.  Let $i_0:G\to G\times \h$
be the embedding $i_0(g)=(g,0)$. We have a functor
$i_0^{-1}:C_\Up\to D(G)$.  We also have a functor $\Sigma:D(G)\to C_\Up$;
$\Sigma(F)=F*_G \sp$ (it is easy to show that $\mS(F*_G\sp)\subset \Up$).

\begin{Theorem}\label{equiv:restr} The functors $i^{-1}_0$ anb $\Sigma$ are mutually quasi-inverse
equivalences.
\end{Theorem}
\begin{proof}
Let $F\in C_\Up$ and consider $F*_G\sp\in D(G\times\h\times \h)$.
As above, let  $a:G\times \h\times \h\to G\times \h$ be the addition. 
Similar to above, one can show that the natural map
\begin{equation}\label{a:iso}
a^{-1}Ra_* (F*_G\sp)\to F*_G\sp
\end{equation}
is an isomorphism.  Let $i_1,i_2:G\times \h\to G\times \h\times\h$ be
given by $i_1(g,A)=(g,A,0)$; $i_2(g,A)=(g,0,A)$. In the same spirit as above,
we can apply $i_1^{-1},i_2^{-1}$ to (\ref{a:iso}). We will get  functorial isomorphisms
$$
Ra_*(F*_G\sp)\to F=\Id(F),\quad Ra_*(F*_G \sp)\to i_0^{-1}F*_G\sp=
\Sigma i_0^{-1}F,
$$
whence an isomorhism of functors $\Id_{D(G)}\cong \Sigma i_0^{-1}$.

Let us consider the composition in the opposite order:
$$
i_0^{-1}\Sigma(F)=F*_G (\sp|_{G\times 0})=F*_G \gf_{e_G}=F.
$$
This way we get an isomorphism $\Id_{D(G\times \h)}\cong i_0^{-1}\Sigma$.
\end{proof}
\subsubsection{Lemma} These Lemma  will be used in the sequel.  Let $A\in \h$. Let $I_A:G\to G\times \h$
be given by $I_A(g)=(g,A)$. Let $S_A:=I_A^{-1}\sp$. Let $T_A:\h\to \h$; $T_A(A_1)=A+A_1$ be the shift
by $A$.
\begin{Lemma}\label{sdvigsp} We have an isomorphism $T_A^{-1}\sp\cong S_A*_G\sp$.
\end{Lemma}
\begin{proof}
Apply the functor $I_A^{-1}$ to (\ref{secad}).
\end{proof}

\section{Study of $\bfS|_{\Zentrum\times C_-^\circ}$}
\label{restrC-}
We denote by $j_{C_-^\circ}:C_-^\circ \to \h$ the open embedding.
We will denote by the same symbol the induced embeddings
$\Zentrum\times C_-^\circ\to \Zentrum\times \h$; $G\times C_-^\circ\to G\times\h$.

We start with studying the object 
$j_{C_-^\circ}^{-1}\sp$.

\subsection{Microsupport of $j_{C_-^\circ}^{-1}\sp$} 
\begin{Lemma} \label{supportogran} The object $j_{C_-^\circ}^{-1}\sp$ is microsupported
within the set of points of the form
$(g,A,\omega,\eta)\in G\times
C_-^\circ \times \g^*\times \h^*$ such that there exists an $X\in \g$ satisfying:

1) $g=e^X$;

2) $\|X\|\leq -A$;

3)$[X,\omega]=0$;

4) $\Tr X|_{V_k(\omega)}=-i<A,e_{d_k}>$;

5)$ \eta=\|\omega\|$.
\end{Lemma}

\begin{proof}
As was shown in the proof of Theorem \ref{specth},  we  have 
$C_-^\circ =\bigcup\limits_{M}MV_b$ and
$$
\sp|_{G\times MV_b}\cong \Phi^-_M.
$$
The object $\Phi^-_M$ is defined by 
$$
a_M^{-1}\Phi^-_M \cong F^-*_G F^-*_G\cdots *_G F^-
$$
(total $M$ copies of $F^-$ and we use the same notation 
as in Sec. \ref{proofspecth}.)

The object $F^-*_G F^-*_G\cdots *_G F^-$ ($M$ times) is the same
as 
$$
Rm_! (F^-)^{\boxtimes M},
$$
where $m:(G\times V_b^-)^M\to G\times (V_b^-)^M$ is induced
by the product on $G$. 
The map $m$ is proper, so we can estimate the microsupport
of $
Rm_! (F^-)^{\boxtimes M}
$
in the standard way. Using
Lemma \ref{supportF-}, we conclude that $
Rm_! (F^-)^{\boxtimes M}
$
is microsupported within the set of points
of the form
$$(e^{X_1}e^{X_2}\cdots e^{X_M};A_1,A_2,\ldots, A_M;
\omega,\eta_1,\eta_2,\ldots,\eta_M),
$$
where $(X_k,A_k,\omega,\eta_k)\in \mS(F^-)$ (we use the
the equality $[X_k,\omega]=0$).

By Lemma \ref{supportF-},   $\eta_1=\eta_2=\cdots
=\eta_M=\|\omega\|$. This implies 
that 
the object  $\Phi_M^-$ is microsupported within the set
of all points of the form

$$
(e^{X_1}e^{X_2}\cdots e^{X_M};A_1+A_2+\cdots+ A_M;
\omega,\|\omega\|),
$$
where $(e^{X_k};A_k;\omega;\|\omega\|)\in \mS(F^-)$ for 
all $k$. 

Lemma \ref{supportF-} says that $[X_k,\omega]=0$.
 By Lemma \ref{spryamlenie}, there exists $X\in \g$ such that
$$e^X=e^{X_1}e^{X_2}\cdots e^{X_M};\quad [X,\omega]=0;
$$
$$
\Tr X|_{V_k(\omega)}=\sum_k \Tr X_k|_{V_k(\omega)};
$$
$$
\|X\|\leq \|X_1\|+\|X_2\|+\cdots+\|X_M\|
$$
According to Lemma \ref{supportF-},
$$
\sum_k \Tr X_k|_{V_k(\omega)}=-i<\sum_k A_k;e_{d_k}>;
$$
$$
\sum _k\|X_k\| \leq -\sum_k A_k.
$$
This implies the statement.
\end{proof}

Using this Lemma we can easily estimate the microsupport
of $j_{C_-^\circ}^{-1}\bfS$.
\begin{Proposition} \label{supportrestriction} The object $j_{C_-^\circ}^{-1}\bfS$
is microsupported within the set of all points of the form
$
(z;A;\eta)\subset \Zentrum\times C_-^\circ\times \h^*,
$
where there exists $B\in C_-$ such that

1) $e^{-B}=z$;

2) $B\geq A$ (i.e. $\forall k:<B-A,e_k>\geq 0$);

3) $\eta\in C_+$ (i.e. $\forall k:<\eta,f_k>\geq 0$);
 if $<\eta,f_k>\;>0$, then $<B-A,e_{k}>=0$.

\end{Proposition}
\begin{proof} Let $I:Z\times C_-^\circ \into G\times C_-^\circ$
be the embedding.  We have
$j_{C_-^\circ}^{-1}\bfS =I^{-1}j_{C_-^\circ}^{-1}\sp[-\dim G]$.
The just proven Lemma implies that $j_{C_-^\circ}^{-1}\sp$
is non-singular with respect to the embedding $I$ (i.e.
given a point $\zeta\in \mS(j_{C_-^\circ}^{-1}\sp)$ 
where $\zeta \in T^*_x (G\times C_-^\circ)$, $x\in \Zentrum\times C_-^\circ$, and $I^*\zeta=0$, one then has $\zeta=0$).

Therefore, the microsupport $I^{-1}j_{C_-^\circ}^{-1}\sp[-\dim G]$
consists of all points of the form $I^*\zeta$,  where 
$\zeta\in \mS(j_{C_-^\circ}^{-1}\sp)$,
 $\zeta \in T^*_x (G\times C_-^\circ)$, $x\in \Zentrum\times C_-^\circ$.

Thus the microsupport $I^{-1}j_{C_-^\circ}^{-1}\sp$ is contained
in the set of all points of the form
$$
(e^X,A,\eta)\in \Zentrum\times C_-^\circ \times \h^*,
$$
where there exists $\omega\in \g^*$ such that
$(e^X,A,\omega,\eta)\in \mS j_{C_-^\circ}^{-1}\sp$.
According to the previous Lemma, this implies
that
$\|X\|\leq -A$;  $\|\omega\|=\eta$ (so $\eta\in C_+$). 

This means that $\eta=i(\lambda^1(\omega),\lambda^2(\omega),
\ldots,\lambda^N(\omega))$, where
$\lambda^1(\omega)\geq \lambda^2(\omega)\geq \cdots\geq
\lambda^N(\omega)$ is the spectrum of $-i\omega$ (with 
multiplicities).

It is clear  that the flag $V_\bullet(\omega)$ 
contains a $k$-dimensional subspace iff $\lambda^k(\omega)>
\lambda^{k+1}(\omega)$ which is the same as 
$<\eta,f_k>\; >0$. 

Denote this $k$-dimensional subspace by $V^k$. We then know
that $XV^k\subset V^k$ and $\Tr X|_{V^k}=-i<A,e_k>$. 
On the other hand we know that $$\Tr -iX|_{V^k}\leq  <\|X\|,e_k>,
$$ for any $X\in\g$.
Hence,
$$
<-A,e_k>\leq <\|X\|,e_k>.
$$
As $\|X\|\leq -A$, this means that
$<-A,e_k>=<\|X\|,e_k>$.  

Let us now set $B:=-\|X\|$. We see that thus defined $B$
satisfies all the conditions.
\end{proof}

\subsubsection{} Let us reformulate the just proven Proposition.

Let $\Lambda\subset C_-$ be a discrete subset. 

Let $X(\Lambda)\subset C_-^\circ\times C_+\subset T^*C_-^\circ$
consist of all points $(A,\eta)$ such that there exists
a $B\in \Lambda$ satisfying:

1) $B\geq A$;

2) If $<\eta,f_k>\; >0$, then $<B-A,e_k>=0$.

For $z\in \Zentrum$ let $S_z\in D(C_-^\circ)$ be the restriction
$$S_z:=
j_{C_-^\circ}^{-1}(\bfS|_{z\times C_-^\circ})
=\sp|_{z\times C_-^\circ}.
$$ 

Let $\bL_z^{-}:=\{B\in C_-| e^{-B}=z\}$.  $\bL_z^{-}$ is an intersection
of a discrete lattice in $\h$ with $C_-$, hence is itself discrete.

 Proposition \ref{supportrestriction} can be now reformulated as:

\begin{Proposition}\label{ms:bfs} We have $\mS(\bfS_z)\subset X(\bL_z^{-})$
\end{Proposition}

\subsection{Sheaves with microsupport of the form $X(\Lambda)$}

Fix a discrete subset $\Lambda\subset C_-$. One can number elements
of $\Lambda$ in such a way that $\Lambda=\{m_1,m_2,\ldots,m_n,\ldots\}$ and $m_n$ is a maximum of
 $\Lambda \backslash \{m_1,m_2,\ldots,m_{n-1}\}$ with respect to the partial order on $C_-$.

For $x\in C_-$ we set $U_x^{-}\subset C_-^\circ$ to consist of
all $y\in C_-^\circ$ such that  $y<<x$.

\begin{Proposition}\label{emen} Let $F\in D(C_-^\circ)$ be microsupported
on the set $X(\Lambda)$. Then there exists an inductive system of
objects in $D(C_-^\circ)$:
$$
F=F_0\to F_1\to F_2\to \cdots F_n\to\cdots,
$$
such that
1) $ L\limdir_n  F_n=0;$

2) We have isomorphisms 
$$
M_n\otimes_\gf  \gf_{U_{m_n}^-}\to \cone(F_{n-1}\to F_n),
$$
for certain graded vector spaces $M_n$.
\end{Proposition}

\subsubsection{Lemma}\begin{Lemma}\label{gammatop}
Let $U\subset V\subset \Re^n$ be open convex sets. Let $\gamma\subset \Re^n$be an open proper cone. Let $\gamma^\circ\subset \Re^n$ be the dual
closed cone; $\gamma^\circ =\{v|<\gamma,v>\geq 0\}$. Suppose that
$V\subset U-\gamma$.
Let
$F\in D(V)$ be such that $\mS(F)\subset V\times \gamma^\circ$. 
Then the restriction map $R\Gamma(V,F)\to R\Gamma(U,F)$
is an isomorphism
\end{Lemma}
\begin{proof} Let $X\subset U\times V$ to consists of all
pairs $(u,v)\in U\times V$ such that $v-u\in -\gamma$. 

Let $\phi:X\times (0,1)\to V$; $F(u,v)=(1-t)u+tv$.
We see that $\phi$ is a smooth fibration with contractible fiber
of dimension $n+1$.
Therefore, the object $\phi^{-1}F$ is microsupported
on the set of those 1-forms which are $\phi$-pullbacks of 1-forms
in the microsupport of $F$. Let $E:=\Re^n$. Identify
$T^*V=V\times E^*$;
$$T^*(X\times (0,1))=X\times (0,1)\times E^*\times
E^*\times \Re.$$

We then have $\mS(F)\subset V\times \gamma^\circ$;
$$\mS(\phi^{-1}F)\subset
 \{(u,v,t,(1-t)\eta; t\eta; <v-u,\eta>)\},
$$
where $\eta\in \gamma^\circ$.

Here we have used the formula
$$
<\eta,d((1-t)u +tv)=(1-t)<\eta,du>+t<\eta,dv>+<\eta,(v-u)>dt.
$$

As $v-u\in -\gamma,\eta\in \gamma^\circ$, we see
that 
$$
\mS(\phi^{-1}F)\subset \{(u,v,t,\eta_1,\eta_2,k)|k\leq 0\}.
$$

Let $S\subset X\times (0,1)$ be any open  subset such that
for any $(u,v)\in X$, the set of all $t\in (0,1)$ such that
$(u,v,t)\in S$ is of the form $(0,T(u,v))$ for some $T(u,v)>0$.
It then follows that the restriction map
$$
R\Gamma(X,\phi^{-1}F)\to R\Gamma(S;\phi^{-1}F)
$$
is an isomorphism.

Let now $S:=\phi^{-1}U$. It is easy to see that all the conditions are satisfied.  It also follows that the projection 
$\phi_U:S\to U$ induced by $\phi$ is a smooth fibration with contractible 
fiber. 

We have a commutative diagram
\begin{equation}\label{odnadiag}
\xymatrix{ R\Gamma(V,F)\ar[r]\ar[d]& R\Gamma(U,F)\ar[d]\\
R\Gamma(X\times (0,1);\phi^{-1}F)\ar[r]& R\Gamma(S;\phi^{-1}F)}
\end{equation}
Coming from the Cartesian square
$$\xymatrix{
S\ar[r]\ar[d]& X\times(0,1)\ar[d]\\
U\ar[r]& V}.
$$
As the fibrations $S\to U$ and $X\times (0,1)\to V$ have contractible fibers, the vertical arrows in (\ref{odnadiag}) are isomorphisms. So is the low horizontal arrow. Hence the upper vertical arrow is also
an isomporphism.
\end{proof}

\subsubsection{}\begin{Lemma} \label{UX}We have 
$$R\hom(\gf_{U_x^-};\gf_{U_y^-})\cong \gf$$
if $x\leq y$. Othewise $R\hom(\gf_{U_x^-};\gf_{U_y^-})=0.$
\end{Lemma}
\begin{proof} If $x\leq y$, we have  an isomorphism
$$R\hom(\gf_{U_x^-};\gf_{U_y^-})=R\hom(\gf_{U_x^-};\gf_{U_x^-})=\gf
$$
because $U_x^-$ is a convex hence contractible set.

If it is not true that $x\leq y$, then $x$ does not belong to the closure
of $U_y^-$ and there exists a convex neighborhood $W$ of $x$ in $\h$
such that $W$ still does not intersect the closure of $U_y^-$. 
Let $V:=U_x^-\cap W$. According to the previous Lemma,
we have an isomorphism
$$
R\hom(\gf_{U_x^-};\gf_{U_y^-})\to R\hom(\gf_V;\gf_{U_y^-})=0.
$$
Indeed, $\gf_{U_y^-}$ is microsupported within the set
$C_-^\circ\times C_+$. The dual cone to $C_+$ is 
$\gamma:=\{x|x\geq 0\}$ and $U_x^-=V-\gamma$. 
\end{proof}

\subsubsection{Lemma} Let $E_1,E$ be real finitely-dimensional
vector spaces and let $U\subset E_1\times E$ be an open convex
set. Let $\gamma\subset E^*$ be a closed  proper cone such that
$\gamma$ is the closure of its interior $\Int \gamma$. Let
$\delta\subset E$ be the dual closed cone. Let 
$x,y\in E$, $y-x\in \Int \delta$. Let $V\subset E_1$ be an
 open subset such that $V\times ((x+\delta)\cap( y- \Int \delta))\subset U$. 
Let $H:=V\times ((x+\delta)\cap( y- \Int \delta))$.

Let us identify $T^*U=U\times E_1^*\times E^*$.
Let $F\in D(U)$ be such 
that $\mS(F)\subset U\times E_1^*\times
\gamma$.

\begin{Lemma}\label{parallelogram} We have $Rhom(\gf_H;F)=0$.
\end{Lemma}
\begin{proof} 

Choose vectors $e\in \Int \gamma$ and $f\in \Int \delta$.
We have $<e,f>\;>\;0$.  Let $E':=\Ker e$.  We have $E=\Re.f \oplus
E'$; $E^*=\Re.e \oplus (E')^*$. Let $\ve>0$. Let
$T_\ve: E\to E$ be given by
$T_\ve|_{E'}=\Id$; $T_\ve(f)=\ve f$. Let $\delta_\ve:=T_\ve\delta$.

There exists a sequence of points $y_n\in E_2$, $\ve_n\in (0,1)$
such that 
$$
(x+\delta)\cap( y_n-\delta_{\ve_n})\subset( x+\delta)\cap( y_m-
\Int \delta_{\ve_m})
$$
for all $n<m$ and
$$
\bigcup_n( x+\delta)\cap( y_n-\Int \delta_{\ve_n})=(x+\delta)\cap (y-
\Int\delta).
$$

We then have 
$$
\gf_{(x+\delta)\cap (y-\Int\delta)}=\limdir_n\gf_{(x+\delta)\cap (y_n-\Int \delta_{\ve_n})}.
$$

Therefore, it suffices to show that
$$
R\hom(\gf_{V\times((x+\delta)\cap (y_n-\Int \delta_{\ve_n}))};F)=0.
$$
More precisely,
given $z\in E$, $\ve\in (0,1)$, and 
any open $W_1\subset V$ such that the closure of $W_1$ is contained
in $V$ and $(x+\delta)\cap (z-\delta_\ve)\subset (x+\delta)\cap
(y-\Int \delta)$,
we will show
$$
R\hom(\gf_{W_1\times(x+\delta\cap z-\Int \delta_\ve)};F)=0
$$

It follows that  there exists an open convex
$W_2\subset E$ such that $W_1\times W_2\subset U$ and
$$
(x+\delta)\cap (z-\delta_\ve)\subset W_2.
$$
Indeed, let $\overline{V}$ be the closure of $V$. Then
$\overline{V}\times ((x+\delta)\cap ( z-\delta_\ve))
\subset U$. As both sets in this product are compact and
$U$ is open, there exists a neighborhood $W_2$ of
$(x+\delta)\cap (z-\delta_\ve)$ such that
$\overline{V}\times W_2\subset U$.

There exists $z'\in (z+\Int \delta_\ve)\cap W_2$ such that
$(x+\delta)\cap (z'-\delta_\ve)\subset W_2$
Let $Z:=(z'-\Int \delta_\ve)\cap W_2  $ so that $z\in Z$ and
for any $u\in Z$, $(x+\delta)\cap (u-\Int\delta_\ve)\subset
W_2$ (because $(u-\Int \delta_\ve)\subset (z'-\Int \delta_\ve)$).

Let $G\subset  W_2\times Z$ be the following locally closed
subset:
$$
G=\{(w,u)|w\in x+\delta\cap u-\Int\delta_\ve\}.
$$

Let $p:W_1\times W_2\times Z\to W_1\times W_2$ and
$q:W_1\times W_2\times Z\to W_1\times Z$.

Let $\Phi:=F|_{W_1\times W_2}$. 

We will show $Rq_*\ihom(\gf_{W_1\times G};p^!\Phi)=0$.
by computing microsupports.

Let us first study $\mS(\gf_G)$, where 
$\gf_G\in D(W_2\times Z)$. 

We have
$$
\gf_G=\gf_{G_1}\otimes \gf_{G_2},
$$
where 
$G_1,G_2\subset W_2\times Z$,
$G_1=(x+\delta\cap W_2)\times Z$;
$G_2=\{(w,u)| w-u\in -\Int \delta_\ve\}$.

We have
$
\mS(\gf_{G_1})
$
is contained within the set of all points
$(w,u,\eta,0)\in W_2\times Z\times E_2^*\times E_2^*$,
where $\eta\in \gamma$. 

Similarly, $\mS(\gf_{G_2})$ is contained within the set
of all points
$(w,u,\zeta,-\zeta)$, where $\zeta\in \gamma_{1/\ve}$ (
$\gamma_{1/\ve}:=T_{1/ve}\gamma$ is the dual cone
to $\delta_\ve$).

Therefore, $\gf_G$ is microsupported within the set of all points
 of the form
$$
(w,u,\eta+\zeta,-\zeta),
$$
where $w,u,\eta,\zeta$ are as above.

Hence $\mS(\gf_{W_1\times G})$ is contained within the set
of all points
of the form
$$
(w_1,w_2,u,0,\eta+\zeta,-\zeta)\in W_1\times W_2\times Z\times
E_1^*\times E_2^*\times E_2^*.
$$

The object $p^!\Phi$ is microsupported within the set of all
points of the form
$$
(w_1,w_2,u,\alpha,\kappa,0)\in W_1\times W_2\times Z\times
E_1^*\times E_2^*\times E_2^*,
$$
where $\alpha\in E_1^*$ is arbitrary and 
$\kappa\notin \Int \gamma$. 

It follows that $\ihom(\gf_{W_1\times G};p^!\Phi)$ is microsuported
within the set of all points of the form
$$
(w_1,w_2,u,\alpha, \kappa-\eta-\zeta;\zeta),
$$
where $\eta,\zeta,\kappa$ are same as before.

The map $q$ is proper
on the support of $\ihom(\gf_{W_1\times G};p^!\Phi)$,
because the latter is contained within the set
$$
W_1\times ((x+\delta)\cap( z'-\delta_\ve))\times Z,
$$
and $(x+\delta)\cap( z'-\delta_\ve)\subset W_2$ is compact.
Therefore, $Rq_*\ihom(\gf_{W_1\times G};p^!\Phi)$ is contained
within the set of all points of the form
$$
(w,u,\alpha,\zeta)\in W_1\times Z\times E_1^*\times E_2^*,
$$
where $\alpha$ is arbitrary, $\zeta\in \gamma_{1/\ve}$, and
there exist $\kappa,\eta$ as above, such that $\kappa-\eta-\zeta=0$. The latter is only possible if $\zeta=0$ (otherwise
$\zeta+\eta\in \Int \gamma$ because $$\gamma_{1/\ve}\subset
\{0\}\cup \Int \gamma.$$
Thus, $Rq_*\ihom(\gf_{W_1\times G};p^!\Phi)$ is microsupported
within the set of all points of the form
$
(w,u,\alpha,0),
$
i.e. is locally constant along $Z$.  There exists a convex open
subset $U_0\subset Z$, $U_0\subset x-\delta$. It follows
that 
$G\cap (W_2\times U_0)=\emptyset$. Therefore, 
$$
Rq_*\ihom(\gf_{W_1\times G};p^!\Phi)|_{W_1\times U_0}=0.
$$

This implies
that 
$$
Rq_*\ihom(\gf_{W_1\times G};p^!\Phi)=0,
$$
because $Z$ is convex, and our object is locally constant along
$Z$.

Therefore,

$$
0=R\hom(\gf_{W_1\times z};Rq_*\ihom(\gf_{W_1\times G};p^!\Phi)
$$
$$
=R\hom(\gf_{W_1\times G}\otimes \gf_{W_1\times W_2\times z};
p^!\Phi)
$$
$$
=
R\hom(
\gf_{W_1\times (x+\delta\cap z-\Int\delta_{\ve})\times z};p^!\Phi)
$$
$$
=R\hom(\gf_{W_1\times (x+\delta\cap z-\Int\delta_{\ve})};\Phi),
$$
as was required.
\end{proof}
\subsubsection{Lemma} \begin{Lemma} \label{shevel}
Let $x,y\in C_-$, $y>x$.
 Let
$I_x:=\{k|<x,f_k><0\}$.  There exists $k\in I_x$ such that
$<y-x,e_k>>0$.
\end{Lemma}
\begin{proof} Assume the contrary, i.e. $<y-x,e_k>=0$ for
all $k\in I_x$. Let $z=y-x$ and let $z_k=<z,e_k>$ so that
$z_k=0$ for all $k\in I_x$. If $l\notin I_x$, then $z_l\geq 0$
and $<z,f_l>=<y,f_l>\leq 0$. On the other hand,
$<z,f_l>=2z_l-z_{l-1}-z_{l+1}$ (we set $z_0=z_N=0$).
For $l\notin I_x$ let $[a,b]\subset [1,N-1]$ be the largest interval containing $l$ and not intersecting with $I_x$.
We then have $z_{A-1}=z_{B+1}=0$; 
$$
0\geq-z_A \geq z_A-z_{A+1}\geq z_{A+1}-z_{A+2}\geq\cdots\geq
z_B\geq 0
$$
(because for any $l\notin I_x$, $2z_l-z_{l-1}-z_{l+1}\leq 0$).
This implies that $z_A=z_{A+1}=\cdots=z_B=0$. Hence, $z_l=0$
for all $l$, $z=0$, and $y=x$, which contradicts to
$y>x$.
\end{proof}

\subsubsection{Lemma} \begin{Lemma}\label{nol} Let $F\in D(C_-^\circ)$
be such that $\mS(F)\subset X(\Lambda)$ and assume
that for all $l\in \Lambda$, $R\Gamma(U_l^-;F)=0$. Then $F=0$.
\end{Lemma}
\begin{proof} Consider open subsets of $C_-^\circ$ of the form
 $U\cap U_x^-$ where $U$ is open and convex and $x\in C_-$.
These sets form a base of topology of $C_-^\circ$. Thus, it
suffices to show $R\Gamma(U\cap U_x^-;F)=0$ for all such
$U,U_x^-$.
By Lemma
\ref{gammatop}, we have an isomorphism
$$
R\Gamma(U_x^-;F)\to R\Gamma(U\cap U_x^-;F).
$$
Thus, it suffices to show that $R\Gamma(U_x^-;F)=0$ for all
$x$. 

Given $x\in C_-$, let $\Lambda_x:=\{l\in \Lambda| l\geq x\}$.
Let $N_x=|\Lambda_x|$. Let us prove the statement
by induction with respect to $N_x$.

If $N_x=0$, then there are no points in $X(\Lambda)$ which
project to $x$. Hence $x\notin \text{Supp}F$. Therefore,
there exists a convex neighborhood of $U$ of $x$ such that
$F|_U=0$. Therefore, we have an isomorphism
$$
R\Gamma(U_x^-;F)\stackrel\sim\to R\Gamma(U\cap U_x^-;F)=0.
$$

Suppose now that $R\Gamma(U_x^-;F)$ for all $x$ with
$N_x<n$. Prove that the same is true for all 
$x$ with $N_x\leq n$.  Let $S\subset C_-$ be the set of
all points $y$ such that $\Lambda_y=\Lambda_x$. 
Let $t_k:=\sup_{y\in S} <y,e_k>$.  As $S\in C_-$, $t_k\geq 0$.
Let $$x':=\sum\limits_{k=1}^{N-1} t_kf_k.$$
Let us show $x'\in C_-$. This is equivalent to
$<x';f_l>\leq 0$ for all $l$. We have
$<x',f_l>=2<x',e_l>-<x',e_{l-1}>-<x';e_{l+1}>$.
We have
$$
2<x',e_l>=\sup_{y\in S} 2<y,e_l>\leq
\sup_{y\in S} <y,e_{l-1}>+<y,e_{l+1}>
$$
$$
\leq <x',e_{l-1}>+<x',e_{l+1}>.
$$
Thus, $x'\in C_-$. It then easily follows that $x'\in S$.

It is clear that $x'\geq x$. Let us show that the restriction
map $R\Gamma(U_{x'}^-;F)\to R\Gamma(U_x;F)$ is an isomorphism. If $x'=x$, there is nothing to prove, so assume $x'>x$.
Let $I:=\{tx+(1-t)x'|0\leq t<1\}$.  Let
 $K:=\{k|<x'-x,e_k>>0\}$.
Ler  $U'$ be a convex neighborhood
of 0 in $\h$. Let $U:=C_-^\circ\cap U'$. We then see that

1)  
$I+U\subset C_-^\circ$ is convex and open;

2) For $U'$ small enough the following is true. Given any
$y\in I+U$, we have $\Lambda_y=\Lambda_x$;
for any $l\in \Lambda_y$ and for any $k\in K$,
$<l-y,e_k>>0$. 

The restriction maps $$R\Gamma(U_{x'}^-;F)\to R\Gamma(I+U;F);
$$
$$R\Gamma(U_x^-;F)\to R\Gamma(x+U;F)
$$
 are isomorphisms
by Lemma \ref{gammatop}. Hence it suffices to show
that the restriction map
\begin{equation}\label{iures}
R\Gamma(I+U;F)\to R\Gamma(x+U;F)
\end{equation}
is an isomorphism.

It  follows from the definition of $X(\Lambda)$ that $F|_{I+U}$ is microsupported
within the set of all points $(y,\eta)\in (I+U)\times \h^*$
such that $<\eta,f_k>=0$ for all $k\in K$. Hence,
$<\eta,x'-x>=0$. This implies that
that ( \ref{iures}) is an isomorphism.

We can now assume $x=x'$. By the construction of $x=x'$, given
any point $y\in C_-$, $y>x$, the set
$\Lambda_y$ is a proper subset of $\Lambda_x$.
If $x\in \Lambda$ there is nothing to prove. Assume
$x\notin \Lambda$. Let $I_x:=\{k|<x,f_k><0\}$. 
By Lemma \ref{shevel} for any $l\in \Lambda_x$ there exists
$k\in I_x$ such that $<l-x,e_k>>0$. It follows that
there exists a neighborhood $U'$ of $x$ in $C_-$ such that
for all $y\in U'$, $\Lambda_y\subset\Lambda_x$ and for all
$l\in \Lambda_y$ there exists $k\in I_x$ such that
$<l-y,e_k>>0$. Let $U=U'\cap C_-^\circ$.  It follows that
$F|_U$ is microsupported within the set of all points
of the form
$$
(u,\eta)\in U\times \h^*,
$$
where $<\eta,f_k>=0$ for some $k\in I_x$.

Let $\cV\subset \h$ be the $\Re$-span of all $f_k$, $k\in K$.

It follows that there exists $\ve>0$ such that 
$x+\sum_{k\in I_x} t_kf_k\in U'$ if for all $k\in I_x$, $t_k\in [0,\ve]$. Indeed, let $U'=W\cap C_-$, where $W$ is a neighborhood
of $x$ in $\h$. It is clear that for $\ve$ small enough,
$x+\sum_{k\in I_x} t_kf_k\in W$. 
As $<x,f_k><0$ for all $k\in I_x$, for all $\ve$ small enough
and for all $k'\in I_x$ we have:
$<x+\sum_{k\in I_x} t_kf_k,f_{k'}><0$. If $\lambda\notin I_x$, then
$<x+\sum_{k\in I_x} t_kf_k,f_\lambda>=\sum_{k\in I_x} t_k<f_k,f_\lambda>\leq 0$,
because $<f_k,f_\lambda>\leq 0$ for all $k\neq \lambda$. Thus, 
$$
x+\sum_{k\in I_x} t_kf_k\in C_-.
$$

Fix $\ve>0$ as above.  There also exists $\ve_1>0$ such that
$$
x+\sum_{k\in I_x} t_kf_k+\sum_{\lambda=1}^{N-1}a_\lambda e_\lambda \in C_-^\circ
$$
as long as $t_k\in [0,\ve]$ and $0>a_\lambda>-\ve_1$.

Let $U_{\ve_1}:=\{x+\sum a_\lambda e_\lambda|0>a_\lambda>-\ve_1\}\subset \h$;
$$
M_{\ve}:=\{\sum_{k\in I_x } t_kf_k|0<t_k<\ve\}\subset \cV.
$$

Let $A:\h\times \cV\to \h$ be the addition map. 
There exists an open convex neighborhood $\cU\in \h\times \cV$
of $U_{\ve_1}\times M_{\ve}$ such that
$A(\cU)\subset U$.  Let $\alpha:\cU\to A(\cU)\subset U$ be the map induced by
$A$. As $\cU$ is convex, $\alpha:\cU\to A(\cU)$ is a smooth
fibration. Let $\Phi:=\alpha^!(F|_{A(\cU)})$. It follows that 
$\mS(\Phi)$ consists of pull-backs of 1-forms
from $\mS(F)$. Thus,
$\mS(\Phi)$ is contained in the set of all points of the form
$$
(A,u,\eta,\kappa)\in \h\times \cV\times \h^*\times \cV^*,
$$
where $(A,u)\in \cU$ and there exists $k\in I_x$ such that
$<\kappa,f_k>=0$. By Lemma \ref{parallelogram}, we have
$$
R\hom(\gf_{U_{\ve_1}\times G};\Phi)=0,
$$
where $G=\{\sum_{k\in K} t_kf_k|0\leq t_k<\ve\}$.

For $L\subset I_x$, let 
$G_L:=\{\sum_{l\in L} t_lf_l|0< t_l<\ve\}$.
Set $G_\emptyset:=\{0\}$.
We have a natural map $$\gf_{U_{\ve_1}\times G}\to
\gf_{U_{\ve_1}\times G_\emptyset}.$$
The cone of this map is obtained from
sheaves $\gf_{U_{\ve_1}\times G_L}$, $L\neq \emptyset$,
by means of succesive extensions. 

We also have
$$
R\hom(\gf_{U_{\ve_1}\times G_L};\Phi)=
R\Gamma(A(U_{\ve_1}\times G_L);\Phi).
$$
We have $$A(U_{\ve_1}\times G_L)\subset 
U_{x+\sum_{l\in L}\ve f_l}^-.
$$
By Lemma \ref{gammatop} the restriction map
$$
R\Gamma(U_{x+\sum_{l\in L}\ve f_l}^-;F)\to 
R\Gamma(A(U_{\ve_1},G_L);F)
$$
is an isomorphism.  As $ x+\sum_{l\in L}\ve f_l>x$
for $L\neq \emptyset$, we have 
$$
R\Gamma(U_{x+\sum_{l\in L}\ve f_l}^-;F)=0
$$
and
$$
R\hom(\gf_{U_{\ve_1}\times G_L};\Phi)=0
$$
for all $L\neq \emptyset$.
Therefore 
$$
R\hom(\cone(\gf_{U_{\ve_1}\times G}\to
\gf_{U_{\ve_1}\times G_\emptyset});\Phi)=0.
$$
Therefore
$$
0=R\hom(\gf_{U_{\ve_1}\times G_\emptyset};\Phi)=R\hom(U_x^-;F)
$$
\end{proof}
\subsubsection{Proof of Proposition \ref{emen}}
Let us construct objects $F_n\in D(C_-^\circ)$ by induction.
Set $F_0=F$. Set $M_n:=R\Gamma(U_{m_n}^-;F_{n-1})$ and
$$
F_n:=\Cone(\alpha_n:M_n\otimes \gf_{U_{m_n}^-}\to F_{n-1}),
$$
where $\alpha_n$ is the natural map. 

We have  structure maps $i_n:F_{n-1}\to F_n$ so that
the sheaves $F_n$ form an inductive system. This system 
stabilizes on any compact $K\subset C_-^\circ$ because
for $n$ large enough, $K\cap U^-_{m_n}=\emptyset$. 

Let $G:=L\limdir_n F_n$. It follows that
$\mS(G)\subset X(\Lambda)$ 
(because $\mS(F_n)\subset X(\Lambda)$).

Let $U_n$ be a neighborhood of $m_n$ in $C_-^\circ$ 
such that the closure of $U_n$ in $C_-^\circ$ is compact.
We have
$$
R\Gamma(U_{m_n}^-;G)\cong
 R\Gamma(U_{m_n}^-\cap U_n;G)
$$
$$
\cong R\Gamma(U_{m_n}^-\cap U_n;F_N)\cong
R\Gamma(U_{m_n}^-;F_N)
$$
for $N$ large enough.

Let $S^i:=R\Gamma(U_{m_n}^-;F_i)$
As follows from Lemma \ref{UX},
$S^i=S^{i+1}$ for all $i\geq n$; also, by construction,
$S^n=0$. Thus $S^N=0$ for $N\geq n$. Therefore,
$$
R\Gamma(U_{m_n}^-;G)=0
$$
for all $n$ and $G=0$ by Lemma \ref{nol}.

Next, $\Cone(F_n\to F_{n-1})$ is isomorphic
to $M_n\otimes \gf_{U_{m_n}^-}$. This proves the proposition.

\subsection{Invariant definition of the spaces
$M_n$} The goal of this section is
to define spaces $M_n$ from Proposition \ref{emen}
in  a more invariant way.

\subsubsection{Lemma}\label{vx} As in the previous Lemma, let $x\in C_-$ and 
let $I_x:=\{k|<x,f_k><0\}$. As was shown in the previous Lemma, there exists
$\ve>0$ such that  $x+\sum_{k\in I_x} t_kf_k\in C_-$ as long as
all $t_k\in [0,\ve]$. Fix such a $\ve>0$.

Set
$$V:=V(x,\ve):=\{y\in C_-^\circ|\forall k\in I_x: <y-x,e_k>\in
 [0,\ve); \forall l\notin I_x:<y-x,e_l><0.\}
$$

\begin{Lemma}\label{KL}
1) We have $$
R\hom(\gf_V;\gf_{U_x^-})\cong \gf[-|I_x|].
$$

2) Let $y\in C_-$. Suppose there exists $k\in \{1,2,\ldots,N-1\}$
 such that either $k\in I_x$ and
$<y-x,e_k>\notin [0,\ve]$ or $k\notin I_x$
and  $<y,e_k>\;<\;<x,e_k>$. Then
$$ R\hom(\gf_V;\gf_{U_y^-})=0.$$

\end{Lemma}
\begin{proof} For $L\subset I_x$ set
$$f_L:=\ve\sum_{l\in L} f_l.$$
 For every $k\in I_x$ we have a natural map
$$
\gf_{U_{x+f_{I_x-\{k\}}}^-}\to \gf_{U_{x+f_{I_x}}^-}.
$$
Let $C_k$ be the corresponding  2-term complex, we put
$\gf_{U_{x+f_{I_x}}^-}$ into degree 0. 

Consider the complex
$$
D:=\bigotimes\limits_{k\in I_x} C_k
$$

 We have
$$
D^{-i}=\bigoplus_L \gf_{U_{x+ f_L}^-},
$$
where the sum is taken over all $|I_x|-i$-element 
subsets $L$ of
$I_x$. 

In particular $D^0=\gf_{U_{x+f_{I_x}}^-}$.  As 
$V\subset U_{x+f_{I_x}}^-$ is a
closed subset, we have a natural map 
$$
\gf_{U_{x+f_{I_x}}^-}\to \gf_V.
$$
This map defines a map of complexes
$D\to \gf_V$ which is a quasi-isomrorphism.
Therefore, we have an isomorphism
$$
R\hom(\gf_V;\gf_{U_y^-})\to R\hom(D;\gf_{U_y^-}).
$$

Let $y=x$, then, according to Lemma \ref{UX},
$R\hom(\gf_{U_{x+f_L}^-};\gf_{U_x^-})=0$ for all $L\neq \emptyset$. For $L=\emptyset$, we have 
$$R\hom(\gf_{U_x^-};\gf_{U_x^-})=
\gf.$$

 Therefore, we have an isomorphism
$$
R\hom(D,\gf_{U_x^-})\cong \gf[-|I_x|].
$$

Let now $y\in C_-$ and $k\in I_x$ be such that
$<y-x,e_k>\notin [0,\ve]$.

Let $D_k:=\bigotimes\limits_{l\neq k} C_l$ so that we have
$D=D_k\otimes C_k$. I.e 
\begin{equation}\label{opatcon}
D\cong \Cone( D_k\otimes \gf_{U_{x+f_{I_x-\{k\}}}^-}\to 
D_k\otimes \gf_{U_{x+f_{I_x}}^-}),
\end{equation}
where the map is induced by the natural map
$$
\gf_{U_{x+f_{I_x-\{k\}}}^-}\to \gf_{U_{x+f_{I_x}}^-}.
$$

We have
$$
D^{-i}_k\otimes \gf_{U_{x+f_{I_x-\{k\}}}^-}=
\bigoplus_{L} \gf_{U_{x+f_L}^-},
$$
where the sum is taken over all $|I_x|-i-1$-element subsets
$L\subset I_x-\{k\}$. 

Analogously,
$$
D^{-i}_k\otimes \gf_{U_{x+f_{I_x}}^-}=
\bigoplus_{L} \gf_{U_{x+f_L}^-}
$$
where the sum is taken over all $|I_x|-i$-element 
subsets $L\subset I_x$ such that $k\in L$.

In view of these identifications, 
the  $-i$-th degree component of the map in (\ref{opatcon}) 
 is induced by the natural maps
$$
\gf_{U_{x+f_L}}\to \gf_{U_{x+f_{L\cup\{k\}}}}.
$$

If $<y-x,e_k>\notin [0,\ve]$, then 
these maps induce isomorphism

$$
R\hom( \gf_{U_{x+f_{L\cup\{k\}}}};\gf_{U_y^-})
\to R\hom(\gf_{U_{x+f_L}};\gf_{U_y^-})
$$.

Hence, the map in (\ref{opatcon}) induces an isomorphism
$$
R\hom(D_k\otimes \gf_{U_{x+f_{I_x}}^-};\gf_{U_y^-})\to 
R\hom(D_k\otimes \gf_{U_{x+f_{I_x-\{k\}}}^-};\gf_{U_y^-})
$$
Therefore, 
$$
R\hom(D,\gf_{U_y^-})=0,
$$
as was stated.

If there exists $k\notin I_x$ such that 
$<y,e_k>\;<\;<x,e_k>$, then it follows that
$R\hom(\gf_{U_{x+f_L}^-};\gf_y)=0$ for all $L$ (because
it is not true that $x+f_L\leq y$).
\end{proof}

\subsubsection{} \begin{Lemma}\label{epsi}
Let  $l\in \Lambda$.  There exists $\ve>0$ such that for
any $l'\in \Lambda$, $l'\neq l$:

--- either there exists $k\in I_l$ such that
$<l'-l,e_k>\notin [0,\ve]$

--- or there exists $k\notin I_l$ such that
$<l',e_k> < <l,e_k>$.
\end{Lemma}
\begin{proof} If there exists $k\in \{1,2,\ldots,N-1\}$ such that $<l'-l,e_k><0$, then one of the conditions is satisfied.
If such a $k$ does not exist, then $l'\geq l$. There are only finitely many $l'\in \Lambda$ with this property. Hence, 
the statement follows from Lemma \ref{shevel}.
\end{proof}
\subsubsection{} Let $m_n$ be a numbering of $\Lambda$ as in
Proposition \ref{emen}. Let $\ve$ be as in the proof of the previous Lemma.

\begin{Lemma}\label{vyrazheniemn} Let $\ve'\in (0,\ve)$.
 We have $$M_n\cong R\hom(\gf_{V(m_n,\ve')};F)[|I_{m_n}|]$$
\end{Lemma}
\begin{proof} Follows from Proposition \ref{emen} and two
previous Lemmas.
\end{proof}
\subsection{The sheaf $\bfS_z$}  Proposition \ref{emen}
and Lemma \ref{vyrazheniemn}
applies to $j_{C_-^\circ}^{-1}\bfS_z$ with 
$\Lambda=\bL_z^-$.  We would like to rewrite the expression
from Lemma \ref{vyrazheniemn} in a more convenient way.

Let $x\in \h$ and $I\subset \{1,2,\ldots,N-1\}$. 
let $W(I,x)\subset \h$ be given by
$$
W(I,x,\ve)=\{y:\forall k\in I: <y-x,e_k>\in [0,\ve);\forall k\notin
I: <y-x,e_k><0\}.
$$

For  $x\in C_-$ and $\ve$ as in Sec. \ref{vx}, we have 

$$
V:=V(x,\ve)=W(I_x,x,\ve)\cap C_-^\circ,
$$
Set $W:=W(I_x,x,\ve)$. Set $I:=I_x$.

For any $F\in D(\h)$ we have an induced map of sheaves
\begin{equation}\label{vtow}
R\hom_\h(\gf_W;F)\to R\hom_\h(\gf_V;F)=
R\hom_{C_-^\circ}(\gf_V;j_{C_-^\circ}^{-1}F).
\end{equation}
\begin{Lemma} Suppose that $\mS(F)\subset \h\times C_+$.
Then the map (\ref{vtow}) is an isomorphism
\end{Lemma}
\begin{proof} For $z\in \h$ set $U_z=\{y\in \h|y<<z\}$.
Lemma \ref{gammatop} implies that for any $z\in C_-$,
the restriction map
$$
R\hom(\gf_{U_z};F)\to R\hom(\gf_{U_z^-};F)
$$
is an isomorphism.

For $k\in I$ consider the following 2-term complex
$C'_k$
$$
\gf_{U_{x+f_{I-\{k\}}}}\to \gf_{U_{x+f_I}},
$$
where we use the notation from proof of Lemma \ref{KL}.
Let 
\begin{equation}\label{def:D'}
D':=\bigotimes_{k\in I} C'_k.
\end{equation} 
Similar to $D$, we have  a quasi-isomorphism
$$
D'\to \gf_{W}.
$$
We also have 
$$
(D')^{-i}=\bigoplus_L \gf_{U_{x+f_L}},
$$
where the sum is taken over all $|I|-i$-element subsets of $I$.
We have natural maps $C_k\to C'_k$ which induce
maps $D\to D'$. The latter map
is induced by maps
$$
\gf_{U_{x+f_L}^-}\to \gf_{U_{x+f_L}}
$$

According to Lemma \ref{gammatop}, the induced map
$$
R\hom(\gf_{U_{x+f_L}};F)\to R\hom(\gf_{U_{x+f_L}^-};F)
$$
is an isomorphism for all $F$ such that $\mS(F)\subset \h\times C_+$. This implies the statement.
\end{proof}
\subsubsection{} 
\begin{Lemma}\label{pryzhok:fk} Let $F\in D(\h)$ be constant along fibers of
projection $\h\to \h/\Re.f_k$ for some $k$. Then
for all $I\subset \{1,2,\ldots,N-1\}$ such that $k\in I$ and
for all $\ve>0$, we have
$$
R\hom(\gf_{W(I,x,\ve)};F)=0
$$
\end{Lemma}
\begin{proof}
Follows easily from the quasi-isomorphism
$D'\to \gf_{W(I,x,\ve)}$.
\end{proof}

\subsection{Periodicity} Let us get back to 
the object $j^{-1}_{C_-^\circ}\bfS_z$.
In this case $\Lambda=\bL_z^-$. There exists $\ve>0$ such that the condition of Lemma \ref{epsi} is satisfied for all $l\in \bL_z^-$.
Fix such a $\ve$ 
throughout.
 Proposition \ref{emen}
applies to $F=\bfS_z$. by Lemmas \ref{vyrazheniemn}
and \ref{vtow}
we have an isomorphism
$$M_n=R\hom(\gf_{V(m_n,\ve)};\bfS_z)[-|I_{m_n}|]
=R\hom(\gf_{W(I_{m_n},m_n,\ve)};\bfS_z)[-|I_{m_n}|].
$$
For $z\in \h$ and $I\subset \{1,2,\ldots,N-1\}$ and 
$F\in D(\h)$
$$
\Delta_{I;z}(F):=R\hom(\gf_{W(I,z,\ve)};F)[|I|]
$$

Our goal is to prove the following theorem
\begin{Theorem}\label{period} For any $m\in \h$, any $I\subset
\{1,2,\ldots, N-1\}$ and any
$k\in I$
there exists
a  quasi-isomorphism
$$
\Delta_{I;m}\bfS_{z}\to 
\Delta_{I;m-2\pi e_k}\bfS_{ze^{-2\pi e_k}}[-D_k]
$$
where  $D_k=2k(N-k)$.
\end{Theorem}
The rest of the current subsection will be devoted
to proving this Theorem.

In the next two subsections we will prove the main auxiliary result towards the proof.
\subsubsection{Sheaves $\sp|_{G\times -2\pi e_k}$} 
Recall that $\sp\in D(G\times \h)$.
Let $\sph_k:=\sp|_{G\times -2\pi e_k}$, so that
 $\sph_k\in D(G)$. 

\begin{Lemma} We have an isomorphism
$\sph_k=\gf_{W_k}$, where $W_k\subset G$ is an open subset
consisting of all points of the form
$$
W_k=\{e^{-Y}|\|Y\|<2\pi e_k\}.
$$
\end{Lemma} 

\begin{proof}
As follows from the proof of Theorem \ref{specth}
$\sph_k$ can be constructed as follows. Let
us decompose $-2\pi e_k =A_1+A_2+\cdots A_M$,
where $A_i\in V_b^-$. For $A\in C_-^\circ$ set
$U(A)\subset G$; $U(A):=\{e^X| X\in \g;\|X\|<<-A\}$.
 One then has
$$
\sph_k\cong 
\gf_{U_{A_1}}*_G\gf_{U_{A_2}}*_G\cdots*_G \gf_{U_{A_M}}[M\dim \g]
$$ 

Let $g\in G$. It follows that $\sph_k|_g\neq 0$ only if
there exist $X_k\in \g$; $\|X_k\|<<-A_k$ such that
$g=e^{X_1}e^{X_2}\cdots e^{X_M}$. According to
Lemma \ref{Klyachko}, this implies that
$g=e^Y$, where $\|Y\|<<-(A_1+\cdots +A_M)=2\pi e_k$.
Thus, fibers of $\sph_k$ at any point outside of
$W_k$ are zeros.

Let $H:=\sph_k|_{W_k}$. It then suffices to
prove that $H\cong \gf[\dim \g]$.

Let us find $\mS(H)$.
 Observe that the exponential map
identifies $W_k$ with $\{X\in\g|\|X\|<<2\pi e_k\}$.
Lemma \ref{supportogran} implies that $(g,\omega)\in \mS(H)$ only if
there exists $X\in \g$ such that $g=e^X$; $\|X\|\leq 2\pi e_k$,
$[X,\omega]=0$; $<\|X\|-2\pi e_k,e_{d_r(\omega)}>=0$ for all $r$.
As $g\in U_{-2\pi e_k}$ and $\|X\|\leq 2\pi e_k$ we must have
$\|X\|<< 2\pi e_k$, so that $<\|X\|-2\pi e_k,e_l><0$ for all $l$. This  means that $\omega=0$.

Thus, {\em $H$ is a constant sheaf.}

Let us now find $H|_e$. We have $H|_e=\bfS_e|_{-2\pi e_k}$.

However, as follows from  Proposition \ref{supportrestriction}, 
$\bfS_e$  is  constant in the domain consisting of all $A\in C_-$ such that there is no $l\in \bL_0^-$, $l\neq 0$, $A\geq l$. 
Both $-2\pi e_k$ and $-e_1/100$ lie in this domain.
Thus we have an isomorphism
$$
\bfS_e|_{-2\pi e_k}=\bfS_e|_{-e_1/100}=\gf[\dim \g].
$$
This finishes the proof.
\end{proof}

Let us  compute $H_k:=R\hom(\gf_{e^{-2\pi e_k}};\sph_k)$.

Let us choose a small neighborhood  $U$ of $e^{-2\pi e_k}$
 in $G$
so that $U=\{e^{-X}e^{-2\pi e_k}| \|X\|<<b\}$. Let us describe
 the set $U_k:=U\cap W_k$. Let $g\in U\cap W_k$. As $g\in W_k$,
we have $g=e^{-Y}$ where 
$\|Y\|<<2\pi e_k$ which  simply means that
 $\lambda_1(Y)<2\pi(N-k)/N$; $\lambda_N(Y)>-2\pi k/N$, where
$$\lambda_1(Y)\geq \lambda_2(Y)\geq \cdots\geq\lambda_N(Y)$$ is the spectrum
of a Hermitian matrix $Y/i$.

 As $g\in U$, there must exist $X$, $\|X\|<<b$ such that
$e^{-Y}=e^{-X}e^{-2\pi e_k}$, or
$$
e^Y=e^{2\pi e_k}e^X.
$$
Observe that $e^{2\pi e_k}=e^{-2\pi k/N}\Id$.
Therefore, one can number the spectrum of 
$X/i$ in such a way that $\lambda^j(X)-2\pi k/N-\lambda_j(Y)\in 2\pi \mathbb{Z}$, $j=1\ldots N$. In other words, there exist integers $m_i$
such that
$\lambda_j(Y)=-2\pi k/N+\lambda^j(X)+2\pi m_i$,
where $m_i$ are integers.

As $-2\pi k/N<\lambda_j(Y)<2\pi (N-k/N)$ and 
$\lambda^j(X)$ are small
we see that $m_j=0$ or $m_j=1$. Since $\Tr(Y)=\Tr(X)=0$,
$\sum m_j=k$. Since $\lambda_1(Y)\geq \lambda_2(Y)\geq \cdots$, we conclude that $m_1=\cdots=m_k=1$; $m_{k+1}=m_{k+2}=\cdots m_N=0$.  We then see that $$
0>\lambda^1(X)\geq \lambda^2(X)\geq \cdots\geq \lambda^k(X);$$
$$
\lambda^{k+1}(X)\geq \cdots \geq \lambda^N(X)>0.$$

In other words, the set $W_k$
consists of all elements of the form $e^{-2\pi e_k}e^{-X}$ where$\|X\|<<b$ and $X/i$ has $k$ negative eigenvalues and 
$N-k$ positive eigenvalues (and no 0 eigenvalues). 
Let $H_k\subset \g$ be an open subset consisting
of all matrices $A$ such that $A/i$ has $k$ negaitive
 and $N-k$
positive eigenvalues. It now follows that
$$
R\hom_G(\gf_{e^{-2\pi e_k}};\sph_k)\cong
 R\hom_\g(\gf_0;\gf_{H_k}[\dim \g]).
$$

Let $M^\circ\subset M\subset E\subset G(k,N)\times \g$ be  defined as follows:

$$
E=\{(V,X)|XV\subset V\};
$$
$$
M=\{(V,X)| XV\subset V; X/i|_V\geq 0;X/i|_{V^\perp}\leq 0\};
$$
$$
M^\circ=\{(V,X)| XV\subset V; X/i|_V> 0;X/i|_{V^\perp}< 0\}.
$$
It follows that $M\subset E\subset G(k,N)\times\g$ are  closed embeddings
and that $M^\circ\subset M$ is an open embedding.
The projection $\pi: E\to \g$ is proper. The natural
projection $p_E:E\to G(k,N)$ is a complex unitary bundle;
$E=S\otimes \overline{S}\oplus S^\perp\otimes \overline{S^\perp}$, where $S$ is the k-dimensional 
tautological bundle over $G(k,N)$.

Let $j:M^\circ\to E$ be the open inclusion. Then
$k_{H_k}=R\pi_!j_!\gf_{M^\circ}=R\pi_*j_!\gf_{M^\circ}$.
Therefore,
$$
R\hom_\g(\gf_0;\gf_{H_k}[\dim \g])=
R\hom_\g(\gf_0;R\pi_*j_!\gf_{M^\circ}[\dim \g])
$$
$$
=R\hom_M(\pi^{-1}\gf_0;j_!\gf_{M^\circ}[\dim \g])
$$

Let $i:G(k,N)\to E$; $i(V)=(V,0)$ be the zero section. We then have
$$
R\hom_\g(\gf_0;\gf_{H_k}[\dim \g])
 =R\hom_M(i_*\gf_{G(k,N)};j_!\gf_{M^\circ}[\dim \g]).
$$

It is easy to see that the natural map
$$
R\hom_M(i_*\gf_{G(k,N)};j_!\gf_{M^\circ}[\dim \g])\to
R\hom_M(i_*\gf_{G(k,N)};\gf_E[\dim \g])=
R\Gamma(G(k,N);i^!\gf_E)[\dim \g]
$$
is a quasi-isomorphism. We have a natural isomorphism
$i^!\gf_E\cong\orient_E[-\dim_\Re E]$ where $\orient_E$ is the sheaf of
orientations on $E$ which is canonically trivial on every
complex bundle. Thus $i^!k_E[\dim \g]=\gf_{G(k,N)}[-\dim E+\dim \g]=\gf_{G(k,N)}[\dim G(k,N)]\cong D_{G(k,N)}$, where 
$D_{G(k,N)}$ is the dualizing sheaf on $G(k,N)$. Finally
we have $R\Gamma(G(k,N);D)\cong H_*(G(k,N);k)$. Thus we have established
\begin{Proposition} There is
 a natural isomorphism
$$
R^{-\bullet}\hom(\gf_{e^{-2\pi e_k}};\sph_k)\cong H_\bullet
G(k,N).
$$
\end{Proposition}

\subsubsection{} Let $D_k:=\dim_\Re G(k,N)=2k(N-k)$. Let
 $\beta\in H_{D_k}(G(k,N))$ be the fundamental class.

According to the previous Proposition, the element $\beta$ defines a map $B_k: \gf_{e^{-2\pi e_k}}\to \sph_k[-D_k]$ in $D(G)$.
Let $C_k:=\cone B_k$.
\begin{Proposition}\label{efka} The singular support of the sheaf $C_k$
is confined within the set
$$
\{(g,\omega)|<|\omega|,f_k>=0\}
$$
\end{Proposition}
\begin{proof} First, consider the case $g\neq e^{-2\pi e_k}$. 

Then 
$(g,\omega)\in \mS(C_k)$ iff $(g,\omega)\in \mS(\sph_k)$.
The sheaf 
$\sph_k$ is microsupported within the set
$$
(e^{X},\omega),
$$
where  $\|-X\|\leq 2\pi e_k$ and  if $<\omega,f_j>\neq 0$, then
$<\|-X\|,e_j>=
2\pi<e_k,e_j>$ for all $j$. 

Therefore, it suffices to show
that $<\|-X\|/2\pi,e_k>\;<\;<e_k,e_k>$.
Assume the contrary and
let $\eta:=\|-X\|/2\pi$. 
 Let $\eta_l:=<\eta,e_l>$; $\ve_l=<e_k,e_l>$. Set
$\eta_0=\eta_N=\ve_0=\ve_N=0$.
We have $0\leq <\eta,f_l>=2 \eta_l-\eta_{l-1}-\eta_{l+1}$.
Therefore, $\eta_l-\eta_{l-1}\geq \eta_{l+1}-\eta_l$.  These convexity inequalities imply
$$
\eta_l\geq l/k\eta_k
$$
for all $l\leq k$;
$$
\eta_l\geq (N-l)/(N-k) \eta_k,
$$
for all  $l\geq k$.  

If $\eta_k=\ve_k$, these inequalities mean that
$\eta_l\geq \ve_l$ for all $l$. However, we know that
$\eta\leq \ve$. Hence, $\eta=\ve$ and $\|-X\|=2\pi e_k$,
hence $e^{X}=e^{-2\pi e_k}$ which is a contradiction.

Thus, $<\|-X\|,e_k>\;<\;<2\pi e_k,e_k>$, therefore, 
$<\eta,f_k>=0$.

Let us now consider the case $g=e^{-2\pi e_k}$. It  suffices
to consider the restriction $\sph_k|_{U\cap W_k}$ as in the previous
theorem.  Let $V:=\{X\in \g | |X|<<b\}$ We then have
an identification $I:V\to U$; 
$X\mapsto e^{-X}e^{-2\pi e_k}$.
we know  that  $I^{-1}\sph_k\cong R\pi_*\gf_{M^\circ}[\dim \g]|_{V_k}$  an the map $\gf_{e^{-2\pi e_k}}\to \sph_k[d_k]$ is induced
by a certain map $\gf_0\to R\pi_*\gf_{M^\circ}[\dim \g]$. 
Namely, this map comes from the identification
$$
\hom(\gf_0;R\pi_*\gf_{M^\circ}[\dim \g])\to \hom(\gf_{G(k,N)};
\gf_{M^\circ}[\dim \g])\to \hom(\gf_{G(k,N)};\gf_E[\dim\g])
$$
$$
=
H_*(G(k,N)).
$$

Note that the sheaves $\gf_0$ and  $R\pi_*\gf_{M^\circ}$ are
 dilation invariant, so we may study their
 Fourier-Sato transforms. Let us find $(R\pi_*\gf_{M^\circ})^\lor$. Let $E^*\cong E$ be the dual bundle over $G(k,N)$;
let $M^*\subset E^*$ be the closed cone dual to the open convex
cone $M^\circ \subset E$. 
Upon the identification $E^*=E$ by means of the scalar
product, we identify $M^*$ with the set of all pairs
$(X,V)\in \g\times G(k,N)$ such that $XV=V$ and
the smallest eigenvalue of $X/i|_V$ is greater or equal to
the largest eigenvalues of $X/i|_{V^\perp}$.

  Let
$P:\g\times G(k,N)\to E^*$ be the map dual to
$\pi:E\to\g$. Let $p_\g:\g\times G(k,N)\to \g$
be the projection.
 We then have
$$
R\pi_*\gf_{M^\circ}^\lor=p_{\g!}P^{-1}\gf_{M^*}[-\dim_\Re E/G(k,N)]
$$ 

We then see that
$$
P^{-1}\gf_{M^*}=\gf_Z,
$$
where $Z\subset \g\times G(k,N)$; $$Z=\{(X,V)| 
XV\subset V;\lambda_{\text{min}}X|_V\geq \lambda_{\text{max}}X|_{V^\perp}\}.$$

Thus, $R\pi_*\gf_{M^\circ}[\dim \g]^\lor=
Rp_{\g!}\gf_Z[\dim \g-\dim E/G(k,N) ]=Rp_{\g!}\gf_Z[\dim G(k,N)].$
Next, $\gf_0^\lor=\gf_\g$. The map $B_k$ induces a map
of Fourier-Sato transforms:
$$
B^\lor:\gf_\g\to Rp_{\g!}\gf_Z
$$
Let us specify this map. By the conjugacy (since $p_\g$
is proper), one can instead specify a map
$$
B^\lor_\text{conj}:p_\g^{-1}\gf_\g=
\gf_{\g\times G(k,N)}\to \gf_Z.
$$
One can show that this map is simply the natural 
map induced by the closed embedding $Z\subset \g\times G(k,N)$. 

Let us now consider an open set $U\subset \g$ consisting
of all $X\in \g$ such that
 $\lambda_k(X)>\lambda_{k+1}(X)$. We then see that
the projection $Z\times_\g U\to U$ is a homeomorphism.
Therefore, $\cone B^\lor|_U=0$ that is
$\cone B^\lor=(\cone B)^\lor$ is  supported on the 
complement of $U$ which is precisely the set of
all $X\in \g$ such that $<\|X\|,f_k>=0$. This proves
the statement.
\end{proof}

\subsubsection{} Let $l\in \h$. Let $T_l:G\times \h\to
G\times \h$ be the shift in $ l$: $T_l(g,X)=(g,X+l)$.
We know that $T_{l}^{-1} \sp =\sp|_{G\times  l}*_G \sp$ (Lemma \ref{sdvigsp}).
Therefore, the maps $B_k$ induce maps
\begin{equation}\label{sdvigfazy}
B_k':\gf_{e^{-2\pi e_k}}*_G \sp\to 
\sp|_{G\times e^{-2\pi e_k}}*_G\sp[-D_k]=T_{-2\pi e_k}^{-1}\sp[-D_k],
\end{equation}
where $D_k=\dim G(k,N)$.  
The previous Proposition implies
that
\begin{Corollary} $\cone B_{k'}$ is locally  constant
on the fibers of the projection $G\times \h\to
G\times \h/f_k$.
\end{Corollary}
\begin{proof} We have
$$
\cone B_{k'}\cong  C_{k'}*_G \sp.
$$
Using the previous Proposition as well as Theorem \ref{specth}
one can easily show that 1-forms from $\mS(C_{k'}*_G\sp)$ do vanish
on the fibers of the projection $G\times \h\to
G\times \h/f_k$.\end{proof}

Let $z\in \Zentrum$ and
 restrict (\ref{sdvigfazy}) onto
$ze^{-2\pi e_k}\in G$. We will get a map
$$
B_k^g:\bfS_{z}\to
 T_{-2\pi e_k}^{-1} \bfS_{ze^{-2\pi e_k}}[-D_k].
$$
It follows that $\cone B_k^g$ is  also constant along 
the fibers of the projection $\h\to \h/f_k$.
\subsubsection{} The map $B_k^g$ induces a map
$$
\Delta_{I,m}(\bfS_{z})\to  
\Delta_{I,m} T_{-2\pi e_k}^{-1} 
\bfS_{ze^{-2\pi e_k}}[-D_k]. $$
for all $I$ and $m$. This is the same as a map
\begin{equation}\label{deltsasdvig}
\Delta_{I,m}\bfS_{z}\to \Delta_{I,m-2\pi e_k} \bfS_{
ze^{-2\pi e_k}}[-D_k].
\end{equation}
\begin{Proposition} If $k\in I$, the above map
is a quasi-isomorphism.
\end{Proposition}
\begin{proof}
Follows from Lemma \ref{pryzhok:fk}.
\end{proof}

Theorem \ref{period} now follows directly from the
previous Proposition.

\subsubsection{Corollary from Theorem \ref{period}} \label{deem}
 Let $u\in \h$, 
$u=2\pi \sum x_ie_i$, set $D(u):=-\sum x_kD_k$.

We then see:

\begin{Corollary}\label{pryzhokcor} Let $z\in \Zentrum$,
  $m\in \bL_z\cap C_-$.  Then there exists an isomorphism
\begin{equation}\label{jumpiso}
 \Delta_{m,I_m} \bfS_z\cong \Delta_{0,I_m} \bfS_e [D(m)].
\end{equation}
\end{Corollary}
\begin{proof} 
Follows directly from Theorem \ref{period}.
\end{proof}

\subsection{Computing $\Delta_{0,I} \bfS_e$}
Let $I:=\{j_1<j_2<\cdots<j_r\}$. Let $\Fl(I)$
 be the partial flag manifold with  dimensions of the subspaces
being $j_1,j_2,\ldots,j_r$. We will show
\def\cE{{\mathcal{E}}}

\begin{Proposition}\label{pryzhok2}
$$
\Delta_{0,I} \bfS_e \cong H^\bullet(\Fl(I)).
$$
\end{Proposition}
\begin{proof}

Let $b$ be as in Sec. \ref{specp}. Let $Z\in C_-^\circ$; $-Z<<b$.
  One can choose $\ve$ so small
that $Z+\sum_k a_kf_k\in C_-^\circ$ if  $0\leq a_k\leq \ve$.

For $A\in \h$, set $S_A:=\sp|_{G\times A}$.
We have
$$
\Delta_{0,I}\bfS_e=R\hom_\h(\gf_{W(I,0,\ve)};\bfS_e)[|I|]
$$

For $\delta>0$, let
$$
W(I,0,\ve,\delta)\subset \h
$$
be the set of all points $A$
such that for all $k\in I$, $<A,e_k>\in [0,\ve)$;
for all $k\notin I$, $-\delta<<A,e_k><0$.

We have a natural map 
$\gf_{W(I,0,\ve,\delta)}\to\gf_{W(I,0,\ve)}$. Using the complex
$D'$ from \ref{def:D'} one can easily prove
that for any object in $D(\h)$ whose microsupport
is contained within $\h\times C_+$, in particular, for 
$\bfS_e$, the natural map
$$
R\hom_\h(\gf_{W(I,0,\ve)};\bfS_e)\to 
R\hom_\h(\gf_{W(I,0,\ve,\delta)};\bfS_e)
$$
is an isomorphism.

One can choose $\ve,\delta$ so small that
$Z+W(I,0,\ve,\delta) \subset V_b\cap C_-^\circ$.
Set $W:=W(I,0,\ve,\delta)$.

By definition, we have:
$$
R\hom(\gf_W;\bfS_e)
$$
$$
=R\hom_{G\times\h}(\gf_e\boxtimes \gf_{W};\sp).
$$

We have a endofunctors on $D(G\times \h)$: $E_\pm: F\mapsto S_{\pm Z}*_G F$.  
The composition $$E_+E_-(F)=S_Z*_GS_{-Z}*_G F=S_{Z-Z}*_G F=S_0*_G F=F$$
is isomorphic to the identity (we have use an isomorophism $S_{Z_1}*_G S_{Z_2}=S_{Z_1+Z_2}$
which follows directly from (\ref{convsp}).) Thus, $E_+E_-\cong \Id$; likewise
$E_-E_+\cong \Id$, so $E_\pm$ are quasi-inverse autoequivalences of $D(G\times \h)$.
Hence, we have

$$
R\hom_{G\times\h}(\gf_e\boxtimes \gf_{W};\sp)=R\hom_{G\times \h}(S_Z\boxtimes \gf_{W};T_Z^{-1}\sp)
$$
$$
=R\hom_{G\times \h}(S_Z\boxtimes 
\gf_{Z+W};\sp),
$$
where the  last equality follows from Lemma \ref{sdvigsp}.
 As $Z+W\subset C_-^\circ\cap V_b$, 
we have:
$$R\hom(\gf_W;\bfS_e)=
R\hom_{G\times (C_-^\circ\cap V_b)}(S_Z\boxtimes \gf_{Z+W};
\sp|_{G\times( C_-^\circ\cap V_b)})
$$
Let $V:=C_-^\circ\cap V_b$.
As follows from the proof of Theorem \ref{specth},  we have
$$
\sp|_{G\times V}=\gf_{\{(e^X,v)|\|X\|<<-v\}}[\dim \g],
$$
Analogously,
$S_Z:=\{e^X|\|X\|<<-Z\}[\dim \g]$.

Let $V_Z\subset \g$, $V_Z:=\{X|\|X\|<<-Z\}$.
Let $\Omega\subset \g\times \h$;
$$
\Omega:=\{(X,A)| \|X\|<<-A\}.
$$
We then have
$$
\Delta_{I,0}\bfS_e=R\hom_{\g\times \h}(\gf_{V_Z}\boxtimes
\gf_{Z+W};\gf_{\Omega})[\|I\|]
$$

Let $\cO$ be the closure of $\Omega$ in $\g\times \h$.
As $\Omega$ is an open proper cone, we have
$$
\gf_\Omega=R\ihom(\gf_\cO;\gf_{\g\times \h}).
$$
Therefore,
$$
\Delta_{I,0}\bfS_e=R\hom_{\g\times \h}((\gf_{V_Z}\boxtimes
\gf_{Z+W})\otimes \gf_\cO;\gf_{\g\times \h})[|I|]
$$

Let $A:=(V_Z\times (Z+W))\cap \cO$ so that
$$
(\gf_{V_Z}\boxtimes
\gf_{Z+W})\otimes \gf_\cO=\gf_A.
$$
Let $p:\g\times \h\to \g$ be the projection.
We have $\gf_{\g\times \h}=p^!\gf_\g[-\dim\h]$. Hence, by the conjugacy,
\begin{equation}\label{promp}
\Delta_{I,0}\bfS_e=R\hom_\g(Rp_!\gf_A;\gf_\g)[-\dim \h+|I|].
\end{equation}

Let $X\in \g$ and consider 
$$
(Rp_!\gf_A)|_X=R\Gamma_c(\h;\gf_{A\cap X\times \h}).
$$

Let $X_k=<\|X\|,e_k>$; $Z_k=<Z,e_k>$. We see that 
$A\cap X\times \h$ is non-empty only if $X\in V_Z$, 
i.e. $X_k+Z_k<0$ for all $k$. In this case we see
that $A\cap X\times \h$ consists of all points of the form
$(X,Z+\sum_{k=1}^N t_kf_k)$,
where $0\leq t_k<\ve$ for all $k\in I$; $-\delta<t_k<0$
for all $k\notin I$; $Z_k+t_k+X_k\leq 0$ for all $k$.
Let $L$ be the set of all $k\in I$ such that
$Z_k+X_k>-\ve$. One then sees that these conditions 
are equivalent to
$$
0\leq t_k\leq -X_k-Z_k
$$
for all $k\in L$;

$$
0\leq t_k<\ve
$$
for all $k\in I\backslash L$;

$$
-\delta<t_k<0.
$$
for all $k\notin I$.

It follows that
$R\Gamma_c(\h;\gf_{A\cap X\times \h})=0$ if $L\neq I$.
Thus, the object $Rp_!\gf_A$ is supported on an open subset
$E_\ve\subset \g$ consisting of all  points
$X$ such that $X_k+Z_k<0$ for all $k\notin I$ and 
$-\ve<X_k+Z_k<0$ for all $k\in I$.

Let $F_\ve:=E_\ve\times \h\cap A$. It follows that the natural 
map
$Rp_!\gf_{F_\ve}\to Rp_!\gf_A$ is an isomorphism.

One also has a natural isomorphism
$$
Rp_!\gf_{F_\ve}=\gf_{E_\ve}[|I|-N+1]=\gf_{E_\ve}[|I|-\dim \h].$$
We can substitute this into (\ref{promp}):
$$
\Delta_{I,0}\bfS_e=R\hom_\g(\gf_{E_\ve};\gf_\g)
$$
which can be rewritten as
$$
\Delta_{I,0}\bfS_e[\|I\|]=H^\bullet(E_\ve),
$$
because $E_\ve\subset \g$ is an open subset.

\begin{Lemma} For $\ve>0$ small enough,
we get:
$$
\forall Y\in E_\ve; \forall i\notin I: <Y,f_i><0.
$$
\end{Lemma}
\begin{proof}
We have $<Y,f_i>=
<X+\sum t_jf_j,f_i><<X,f_i>+t_i<f_i,f_i>,$
because $<f_i,f_j>\leq 0$ for all $i\neq j$. 
Next,
$$
<X,f_i>+t_i<f_i,f_i>\leq <X,f_i>+2\ve<0
$$
for $\ve$ small enough.
\end{proof}

This implies that for any $X\in E_\ve$ and for every $k\in I$,
we have a well-defined $k$-dimensional eigenspace
space $V^k(X)$ spanned by
the eigenvectors of $X/i$ with top $k$ eigenvalues.
The spaces $V^\bullet(X)$ form a flag from $\Fl(I)$.
Thus we have a map $P:E_\ve\to \Fl(I)$; $P(X):=V^\bullet(X)$.

Let $\cE\to \Fl(I)$ be the vector bundle whose fiber at 
$V^\bullet\in \Fl(I)$ consists of all unitary matrices
preserving $V^\bullet$. One can easily check that
$E_\ve\subset \cE$ is an open convex subset. Therefore,
$P$ induces an isomorphism
$H^\bullet(E_\ve)=H^\bullet(\Fl(I))$
so that 
$$
\Delta_{I,0}\bfS_e[|I|]\cong H^\bullet(\Fl(I)).
$$
\end{proof}

\subsubsection{The sheaf $j_{C_-^\circ}^{-1}\bfS$,  up to an isomorphism}
Let us combine
Proposition \ref{emen}, Corollary \ref{pryzhokcor}, and Proposition \ref{pryzhok2}. We will then get the following statement:

\begin{Proposition} Let $g\in \Zentrum$. There exists an
 inductive system of sheaves on $C_-$:
$$
j_{C_-^\circ}^{-1}\bfS_g=F^0\to F^1\to \cdots\to F^n\to \cdots
$$
such that
 
$$
 L\limdir_n  F_n=0;
$$

$$
\Cone(F_{n-1}\to F_n)\cong \gf_{\umin_{m_n}}\otimes  H^*(\Fl(I_{m_n}))
[D(m_n)],
$$
where the sequence $m_1,m_2,\ldots,m_n,\ldots$ consists of
all elements of $\bL_g\cap C_-$, each term occuring once.
\end{Proposition}

It turns out that this Proposition allows us to recover
the isomorpism type of $j_{C_-^\circ}^{-1}\bfS_g$.

Let 
$A_n:=\gf_{\umin_{m_n}}\otimes H^*(\Fl(I_{m_n}^c))[D(m_n)]$. 
\begin{Lemma} There exist maps
$$
i_n:A_n\to F_0
$$
such that  for every $n$
the triangle
\begin{equation}\label{conus:}
\bigoplus_{n'\leq n} A_{n'}\to  F_0\to  F_{n}
\end{equation}
is exact.
\end{Lemma}
\begin{proof} Let us prove the statement by induction in $n$.For $n=1$ we have a natural map 
$i_1:A_1\to j_{C_-^\circ}^{-1}\bfS_g=F_0$ whose cone is $F_1$; this proves the base.

Let us now proceed to the incuction step. 

Suppose we have aready constructed an exact triangle
as in (\ref{conus:}) for some $n$.  Let us apply to this triangle the functor $R\hom(A_{n+1},\cdot)$.

We will then get an exact sequence
\begin{equation}\label{po-sled}
R^0\hom(A_{n+1};j_{C_-^\circ}^{-1}\bfS_g)\to R^0\hom(A_{n+1};F_{n})\to
\bigoplus_{n'\leq n}R^1\hom(A_{n+1};A_{n'}).
\end{equation}

Observe that the last arrow in this sequence is 0:
because of Lemma \ref{UX} and because all the spaces
$M_i$ are concentrated in the even degrees, therefore,
$R^{\text{odd}}\hom(A_i,A_j)=0$ for all $i,j$.

Therefore, the left arrow in  (\ref{po-sled}) is surjective. Next, we have a map $E_{n+1}:A_{n+1}= \cone (F_n\to F_{n+1})\to F_{n}$. Let $i_{n+1}:A_{n+1}\to \bfS_g$ be the lifting of $E_{n+1}$ (which exists precisely because of surjectivity 
of the left arrow in (\ref{po-sled}). It is straigtforward to see that so chosen $i_{n+1}$ satisfies the conditions
\end{proof}

\begin{Theorem}\label{nada:hvat} There exists an isomorphism
$$
\bigoplus_{l\in \bL_g\cap C_-} A_l\to j_{C_-^\circ}^{-1}\bfS_g,
$$
where $A_l:=\gf_{\umin_{l}}\otimes H^*(\Fl(I_{l}))[D(l)]$
\end{Theorem}
\begin{proof}

Indeed, the previous Lemma implies
that the map $\bigoplus_n i_n:\bigoplus_n A_n\to j_{C_-}^{-1}\bfS_g$
is an isomorphism, whence the statement.
\end{proof}

\def\strict{\text{strict}}
\def\DBSh{DB\Sh}
\def\DCSh{DC\Sh}
\def\DKSh{DK\Sh}
\def\bfb{\mathbf{b}}
\def\bfc{\mathbf{c}}
\def\Kosz{\text{Kosz}}
\def\cG{\mathcal{G}}
\def\ccG{\cG'}
\def\cH{\mathcal{H}}
\def\H{\cH}
\def\Z{{\mathcal{Z}}}

\section{$B$-sheaves}\label{bsche}

For  a manifold  $X$ let $\complexes_X$ be the dg-category of complexes of sheaves on $X$.

Suppose $X$ is equipped with an action of the monoid
$\bL_-$.
Let $T_l:X\to X$ be the translation by $l\in \bL_-$.   In all our examples 
all $T_l$ will be open embeddings.

Let $F\in \complexes_X$ and $l\in \bL_-$. Set $A(l):=A_F(l):=\hom(F,T_{l}^{-1}F)$.  These 
complexes obviously form  a $\bL_-$-graded dg-algebra to
be denoted by $A=A_F$. 

Let $B$ be another $\bL_-$-graded dg-algebra. We define {\em a $B$-sheaf structure on $F$} as a $\bL_-$-graded dg-algebra homomorphism
$B\to A_F$. $B$-sheaves form a triangulated dg-category in the obvious way.

We will only use algebras  $B$ of a special type. Namely,
We will assume that:

---$B(l)$ is concentrated in degrees $\leq -D(l)$;

---the cohomology $H^\bullet(B(l))$ is concentrated in
degree $-D(l)$ and is one dimensional;

---one can choose
generators $b_l\in H^{-D(l)}(B(l))$ which are stable
under the product induced by the product on $B$.

Call such a $B$ {\em homotopically standard}.

Let $\bfb$ be a $\bL_-$-graded dg-algebra defined by setting
$\bfb(l)=k[D(l)]$ .  Let $1_l:=1\in k[D(l)]^{-D(l)}$ be generators.

We then define the product on $\bfb$ by setting
$1_l1_m=1_{l+m}$.  
It  follows that we have a unique $\bL_-$-graded dg-algebra homomorphism $B\to \bfb$ such that the induced
 map $H^\bullet(B)\to H^\bullet(\bfb)=\bfb$ sends $b_l$ to 
$1_l$.

We call a $B$-sheaf $F${\em acyclic} if it is acyclic
as a complex of sheaves on $X$ (i.e. for each $x\in X$ the
complex of fibers $F_x$ is acyclic).

Following \cite{Dr} we can produce the derived dg-category by taking the quotient with respect to the full subcategory of acyclic objects.

However, in our situation one can prove

\begin{Proposition} The category of $B$-sheaves has enough 
injective objects.
\end{Proposition}

{\bf Remark} By an injective object we mean any $B$-sheaf $X$ such that
for any acyclic $B$-sheaf $Z$, the complex $\hom(Z,X)$ is acyclic.
\begin{proof} Let $A$ be a $B$-sheaf. Let
$\beta_A$ be another $B$-sheaf such that
$\beta_A:=\prod_{l\in \bL_-}\hom(B(l); T_l^{-1}A)$
We then get a $B$-structure on $\beta_A$ and a natural 
map of $B$-sheaves $A\to \beta_A$. Let now $A\to A'$ be
a termwise injective map in the category of complexes of sheaves on $X$
 (we forget the $B$-structure)
such that $A'$ is injective. We then have a termwise
injective map of $B$-sheaves
$$
A\hookrightarrow \beta(A)\hookrightarrow\beta_{A'}
$$
One sees that $\beta_{A'}$ is injective:  given
any $B$- sheaf $T$ on $X$ we have
$$
\hom(T,\beta_{A'})=\hom(T,A'),
$$
where $\hom$ on the RHS is in the category of complexes
of sheaves on $X$. As $A'$ is injective,  we see that
$\hom(T,A')\sim 0$ as long as $T$ is acyclic.

\end{proof}

As we know, in this case, the derived category is equivalent
to the full subcategory of injective objects. 

We will only need the homotopy category of the derived  cateogory of  $B$-sheaves.  Denote this category
by $\DBSh_{X}$.

Let $f:X_1\to X_2$ be a $\bL_-$-equivariant map 
We then have a right  derived functor of $f_*$:
$Rf_*:\DBSh_{X_1}\to \DBSh_{X_2}$: if we choose the category
of injective $B$-sheaves on $X_1$ as a model for $\DBSh_{X_2}$ then $Rf_*$ is given by the termwise application of 
the functor $f_*$. Similarly, one defines functors $Rf_!,f^{-1}$.
One can also define a functor $f^!$ as a right afjoint 
to $Rf_!$,
but we won't need this functor.

Recall that we have a natural map $p:B\to \bfb$. This map
induces an obvious functor $p^{-1}$ from the category of 
$\bfb$-sheaves to the category of $B$-sheaves on $X$
and one sees that this map has a right adjoint $p_*$.
This functor admits a right derived $\pi:=Rp_*:
\DBSh_X\to D\bfb\Sh_X$. This functor is an equivalence.
\subsubsection{A $B$-sheaf structure on the sheaves $\sp$ and
 $\bfS$}
Let $\sp\in D(G\times \h)$ be as in Theorem \ref{specth}.
Choose an injective representative for $\sp$, to be denoted by
the same symbol $\sp$. Define a diagonal $\bL_-$-action on $G\times
 \h$ by setting $l.(g,A):=(e^l g,A+l)$.  For $l\in \bL_-$ consider
the complex  $B'(l):=\hom_{G\times \h}(\sp;T_l^{-1}\sp)$ and compute its cohomology:
$$
H^\bullet(B'(l))=R^\bullet\hom(\sp;T_l^{-1}\sp).
$$

Let $i_0:G\to G\times \h$, $i_0(g)=(g,0)$.
By Theorem (\ref{equiv:restr}) we have
$$
R^\bullet\hom(\sp;T_l^{-1}\sp)=
R^\bullet\hom_{G}(i_0^{-1}\sp;i_0^{-1}T_l^{-1}\sp).
$$

We know that $i_0^{-1}\sp=\gf_{e_G}$.
 Thus,
$$
R^\bullet\hom_{G}(\sp;T_l^{-1}\sp)
=R^\bullet\hom_{G}(\gf_{e};i_0^{-1}T_l^{-1}\sp).
$$

As $T_l^{-1}\sp$ is non-singular along $i_0(G)\subset G\times \h$, we have an isomoprphism
$
i_0^{-1}T_l^{-1}\sp\cong i_0^!T_l^{-1}\sp[\dim \h]=i_0^!T_l^!\sp[\dim \h].
$
Thus,
$$
R^\bullet\hom_{G}(\gf_{e};i_0^{-1}T_l^{-1}\sp[\dim \h])=R^\bullet\hom_{G}(\gf_e;i_0^{!}T_l^{!}\sp[\dim \h])
$$
$$
=i_{(e, 0)}^!T_l^!\sp[\dim \h]=i_{(e^l,l)}^!\sp[\dim \h]
$$
$$
=i_l^{!}\bfS_l[\dim \h].
$$
Here $i_{(e,0)}; i_{(e^l,l)}$ denote embeddings of the points specified into $G\times \h$, and, 
likewise, $i_l$ is the embedding of the point $l$ into $\h$.

Theorem \ref{nada:hvat} implies that  $H^{<D(l)}i_l^{!}\bfS_{e^l}=0$ and $H^{D(l)}i_l^{!}\bfS_{e^l}$
is one dimensional.  Indeed, one sees that given $l'\in C_-$, 
we have: $i_l^!\gf_{\umin_{l'}}\cong \gf[-\dim \h]$
for all $l'\geq l$; otherwise $i_l^!\gf_{\umin_{l'}}=0.$ Therefore, 
$$
i_l^!\bfS_{e^l}[\dim \h]\cong \bigoplus_{l'\in \bL_{e^l},l'\geq l} H^\bullet(\Fl(I_{l'}))[D(l')],
$$
and the lowest degree contribution comes from $H^0(\Fl(I_l))[D(l)]=\gf[D(l)]$.

Set $B(l):=\tau_{\leq -D(l)}B'(l)$. It then follows that
$B$ is a homotopically standard $\bL_-$-graded algebra.
We thus automatically get a $B$-sheaf structure on $\sp$. 
Let $I_Z:\Zentrum\times \h \to G\times \h$ be the embedding.
This embedding is $\bL_-$-equivariant, where $\bL_-$-action on
$\Zentrum\times \h$ is defined by
$$
T_l(c,A)=(e^l c;A+c).
$$
Hence we get a $B$-sheaf structure on $\bfS:=I_Z^!\sp$ (as
$\sp$ is injective and $I_Z$ is a closed embedding
one can compute $I_Z^!$ by taking sections supported on
$\Zentrum\times \h\subset G\times\h$.

\subsection{A $B$-sheaf $j^{-1}_{C_-^\circ}\bfS$ on $\Zentrum\times
 C_-^\circ$.} Let $j_{C_-^\circ}:\Zentrum\times C_-^\circ\to 
\Zentrum\times \h$ be the open embedding. The $\bL_-$-action
on $\Zentrum\times \h$ preserves $\Zentrum\times C_-^\circ$,
thus making the embedding $j_{C_-^\circ}$ to be $\bL_-$-equivariant.

 We then have a $B$-sheaf
$j^{-1}_{C_-^\circ}\bfS$.
 Let $p:B\to \bfb$ be the canonical map. Let us choose
an injective model for $Rp_* \bfS$, to be still denoted
by $\bfS$.

Let us study the $\bfb$-structure on $j^{-1}_{C_-^\circ}\bfS$.
Let $I\subset \{1,2,\ldots,N-1\}$. Let $e_I:=\sum_{i\in I} e_i$.

According to Theorem \ref{nada:hvat} we have a map
$$
i_I:H^*(\Fl(I))[D(-2\pi e_I)]\otimes \gf_{e^{-2\pi e_I}\times U_{-2\pi e_I}^-}\to j^{-1}_{C_-^\circ}\bfS
$$
Set $H(I):=H^\bullet(\Fl(I))$.
For $I\subset J$ we have a tautological projection
$$
\Fl(J)\to \Fl(I)
$$
hence an induced map $H(I)\to H(J)$. Hence $H$ is a 
 functor from the poset of subsets of $\{1,2,\ldots,N-1\}$ to
the category of graded vector spaces.

One can show that this functor is actually free, i.e:
\begin{Lemma}\label{lemmagi}
There exist
graded vector spaces $G(I)$, where $I\subset \{1,2,\ldots,N-1\}$such that we have decompositions 
\begin{equation}\label{rvat}
H(I)=\bigoplus_{J\subset I} G(J),
\end{equation}
which are compatible with the   structure  maps $H(I_1)\to H(I_2)$, $I_1\subset I_2$ in the obvious way.
\end{Lemma}
\begin{proof} Let us use Schubert cellular decomposition of
partial flag varieties $\FL(I)$.   Let $\cF\subset \Fl(I)$ be the 
flag such that $\cF^r\subset \Co^N$ consists
of all points $(v_1,v_2,\ldots,v_N)\in \Co^N$ such that 
$v_k=0$ for all $k>i_r$.

Let $H:=\GL_N(\Co)$.
Let $P(I)\subset H$ be the standard parabolic subgroup,
namely the stabilizer of $\cF$. We have $\FL(I)=H/P(I)$.  
Let  $W\subset G$ be the standard Weyl group. For
any $w\in W/W\cap P(I)$ let $[w]\in H/P(I)$ be the image of 
$[w]$ and let $C_{I,w}:=C_w:=B.[w]$ where $B\subset H$ is the standard 
Borel subgroup of upper-triangular matrices. It is well known that the cells $C_w$,
$w\in W/W\cap P(I)$ form a cellular decomposition of 
$\FL(I)$.  We have $\dim_\Re C_w =2D_I(w)$, where
$D_I(w)$ is defined as follows. Let $\pi_I:\{1,2,\ldots,N\}\to 
\{1,2,\ldots,|I|\}$ be defined by letting
$\pi_I(k)$ be the minimal number $r$ such that $i_r\geq k$.

In particular, for any $M\in P(I)$, we have
$M_{ij}=0$ as long as $\pi_I(i)>\pi_I(j)$.
Let $w'\in W$ be any representative of $w\in W/W\cap P(I)$.

One then has that
$D_I(w)$ is equal to the number of all pairs
$(i,j)$ such that $i,j\in \{1,2,\ldots,N\}$,
$i<j$ and $\pi_I(w^{-1}(i))>\pi_I(w^{-1}j)$. 

Thus we have a basis of $H_*(\FL(I))$ labelled by the cells
$C_w$. Let $c_w\in H_{2D_I(w)}\FL(I)$ be the class 
corresponding to $C_w$.

We see that the map $p_{IJ}:\FL(I)\to \FL(J)$ is cellular.
We have $p_{IJ}C_w\subset C_{w'}$ where $w'$ is the image
of $w\in W/W\cap P(I)$ in $W/W\cap P(J)$. 
One sees that $\dim C_{w'}\leq \dim C_w$.
It then  follows that $p_{IJ*}c_w=c_{w'}$ is $D_I(w)=D_J(w')$. Otherwise $p_{IJ*}(c_w)=0$.  

Let us describe the dual map $p_{IJ}^*$. Let $c^w\in H^\bullet(\FL(I)$ be the element dual to $c_w$.  Let us identify 
$W/W\cap P(I)$ with the set $V(I)$ of partitions
$\{1,2,\ldots,N\}=A_1\sqcup A_2\sqcup A_{|I|}$ where
$|A_r|=i_r-i_{r-1}$ and we assume $i_0=0$, $i_{|I|}=N$.
We have a map $Q^{JI}:V(J)\to V(I)$ defined as follows.
Pick $t\leq N$. Let $i_m=j_{t-1}$; $i_M=j_{t}$. 
Order $A_t$ and subdivide it into  several
subsets, such that the first subset consits of the 
first
$i_{m+1}-i_m$ elements of $A_t$; the second subset consists
of the next$
i_{m+2}-i_{m+1}$ elements of $A_t$, etc.
This way we get a partition $Q^{JI}A$. One sees that 
$Q^{JI}=p_{IJ}^*$.  
For $A\in V(I)$ let $\sim_A$ be an equivalence relation
 on $I$  given by
$i_1\sim_A i_2$ if for all $ j_1< j_2$, $j_1,j_2\in [i_1,i_2]$, $A_{j_1}<A_{j_2}$. Call $A\in V(I)$ elementary if $\sim_A$ is
trivial. One then can set $G(I)$ to be the span of all 
elementary $A\in V(I)$.
\end{proof}

Let us now
consider through maps
$$
j_I: G(I)\otimes \gf_{e^{-2\pi e_I}\times U^-_{-2\pi e_I}}[D(-2\pi e_I)]\to
 H(I)\otimes \gf_{e^{-2\pi e_I}\times U^-_{-2\pi e_I}}[D(-2\pi e_I)]\to 
j^{-1}_{C_-^\circ}\bfS.
$$

Introduce a notation: for $l\in \bL_-$, set $\cU_l:=e^l\times \umin_l\subset \Zentrum\times C_-^\circ$.
Denote $\cG_I:= G(I)\otimes \gf_{e^{-2\pi e_I}\times \cU_{-2\pi e_I}}[D(-2\pi e_I)]$.
The $\bfb$-structure on $j^{-1}_{C_-^\circ}\bfS$ gives rise to  maps
$$
T_{l*}\cG_I\otimes \bfb(l)\to T_{l*}j^{-1}_{C_-^\circ}\bfS\otimes \bfb(l)\to 
T_{l*}T_l^{-1}j^{-1}_{C_-^\circ}\bfS=j^{-1}_{C_-^\circ}\bfS
$$
for all $l\in \bL_-$. 
Take the direct sum:
\begin{equation}\label{hvatoiot1}
\iota:\bigoplus\limits_{I\subset \{1,2,\ldots,N-1\};l\in \bL_-}
T_{l*}\cG_I[D(l)]\to j^{-1}_{C_-^\circ}\bfS
\end{equation}
(we have replaced $\bfb(l)=k[D(l)]$).
The sheaf on the LHS has an obvious structure of a $\bfb$-sheaf and
the map $\iota$ is a map of $\bfb$-sheaves.

Furthermore the $\bfb$-sheaf on the LHS splits into a direct sum \def \bS{\mathbb{S}}
of $\bfb$-sheaves
\begin{equation}\label{defbfSI}
\bS^I:=\bigoplus\limits_{l\in \bL_-}
T_{l*}\cG_I[D(l)]
\end{equation}
thus  we have a map of $\bfb$-sheaves
\begin{equation}\label{hvatoiot}
\iota:\bigoplus\limits_{I\subset \{1,2,\ldots,N-1\}}\bS^I\to \bfS
\end{equation}
For future purposes, let us rewrite the definition of $\bS^I$. 
We have
$$
\bS^I:=\bigoplus\limits_{l\in \bL_-}
T_{l*}\cG_I[D(l)]
$$
$$
=G_I[D(-2\pi e_I)]\otimes  [\bigoplus\limits_{l\in \bL_-}
\gf_{\cU_{-2\pi e_I+l}}[D(l)]]
$$
$$
=G_I[D(-2\pi e_I)]\otimes T_{-2\pi e_I*} [\bigoplus\limits_{l\in \bL_-}
\gf_{\cU_{l}}[D(l)]]
$$
Let 
\begin{equation}\label{defcx}
\cX:=\bigoplus\limits_{l\in \bL_-}
\gf_{\cU_{l}}[D(l)]
\end{equation}
with the obvious $\bfb$-structure. 
We then have an isomorphism of $\bfb$-sheaves:
\begin{equation}\label{razlbfs}
\bS^I\cong G_I[D(-2\pi e_I)]\otimes T_{-2\pi e_I*}\cX.
\end{equation}
\begin{Proposition} The map (\ref{hvatoiot}) is a quasi-isomorphism.
\end{Proposition}
\begin{proof} For any $z\in \Zentrum$ and any
 $F\in D(\Zentrum\times C_-^\circ)$ we set $F_z\in D(C_-^\circ)$; $F_z:=F|_{z\times C_-^\circ}$. 
We have   induced maps $$
\iota_z:\bigoplus_I \bS^I_z\to j_{C_-^\circ}^{-1}\bfS_z,$$
and it suffices to show that these maps are isomorphisms
for all $z\in \Zentrum$. 
We know (Proposition \ref{ms:bfs}) that 
$\mS(j_{C_-^\circ}^{-1}\bfS_z)\subset 
X(\bL_-^z)$.  One can easily check that 
$\bS^I_z\in X(\bL_-^z)$ for all $I$.  As follows from 
Proposition \ref{emen} and Lemma \ref{vyrazheniemn},
it suffices to show that the induced maps
\begin{equation}\label{jumpmaps}
R\hom_{C_-^\circ}(\gf_{V_{x,\ve}};\bigoplus_I \bS^I_z)\to 
R\hom_{C_-^\circ}(\gf_{V_{x,\ve}};j_{C_-^\circ}^{-1}\bfS_z)
\end{equation}
are isomorphisms for $\ve>0$ small enough and for
all $x\in \bL_-^z$.
 Let $F\in
D(\Zentrum\times
C_-^\circ)$ and $x\in \bL_-$. Set
$$
\Delta_x(F):=R\hom(\gf_{V_{x,\ve}};F|_{e^x}).
$$

 Let now $F$ be a $\bfb$-sheaf on $\Zentrum\times \h$.
The $\bfb$-structure gives 
rise to maps 
$$
\Delta_x(F)\to \Delta_{x+l}(F)[-D(l)],
$$
for all $l\in \bL_-$.  Set $\delta_F(x):=\Delta_x(F)[-D(x)]$.
Introduce a partial order $\preceq$ on $\bL_-$ by
setting $l_1\preceq l_2$ if $l_2-l_1\in \bL_-$. 
We see that $\delta_F$ is a functor from this poset, viewed
as a category, to the category of graded $\gf$-vector spaces.
As follows from Corollary \ref{pryzhokcor} and 
Proposition \ref{pryzhok2}, we have
$\delta_\bfS(l)=H^\bullet(\Fl(I_l))$. Let $l_1\preceq l_2$.
 As follows
from the  proof of
Proposition \ref{pryzhok2}, the induced map
$\delta_\bfS(l_1)\to \delta_\bfS(l_2)$ is induced
by the projection $\Fl(I_{l_2})\to \Fl(I_{l_1})$
coming from the embedding $I_{l_1}\subset I_{l_2}$.
It then follows from Lemma \ref{lemmagi} that $\delta_\bfS$, as a functor, is 
freely generated by subspaces 
\begin{equation}\label{generator1}
G(I)\subset H^\bullet(\Fl(I))=\delta_\bfS(-2\pi e_I)
\end{equation}
for all $I\subset \{1,2,\ldots,N-1\}$. 

One can easily check that $\delta_{\oplus_I \bS^I}$
is freely generated by the subspaces
\begin{equation}\label{generator2}
G(I)=\delta_{\bS(I)}(-2\pi e_I)\subset 
\delta_{\oplus_I \bS^I}(-2\pi e_I).
\end{equation}

The map $\iota$ preserves the generating supspaces
(\ref{generator1}), (\ref{generator2}). Hence, the maps
(\ref{jumpmaps}) are isomorphisms, which proves the 
Proposition.
\end{proof}

\subsection{Strict $B$-sheaves}\label{strict:sh} Let $F$ be a $B$-sheaf on 
$\h$. Let $v_k\in B(-e_k)$ be a representative
of $u_k\in H^{D(-e_k)}B(-e_k)$. We then have induced maps
$$
a_k:F\to T_{-e_k}^{-1}F[D(-e_k)]
$$
induced by $v_k$. 

Let 
\begin{equation}\label{konshish}
\Con_k:=\Cone\; a_k; 
\end{equation}
let $p_k:\h\to \h/\Re f_k$.
We call $F$ {\em strict} if

1) for all $k$, the natural
map $p_k^{-1}Rp_{k*}\Con_k\to \Con_k$ is an
 isomorphism in $D(\h)$ (that is, $\Con_k$ is constant
along fibers of $p_k$);

2) $F$ is microsupported on 
$\h\times C_+\subset \h\times \h^*$.

 Denote the full subcategory
of $\DBSh_\h$ consisting of all strict $B$-sheaves on 
$\h$ by $\DBSh_\h^\strict$.

Analogously, let $F$ be a sheaf on $C_-^\circ $. Let us define $a_k$ and $\Con_k$ in the same way as above. 

Let 
$C_-^\circ /\Re f_k$ be the image of $C_-^\circ$ under the 
map $C_-^\circ \to \h\to \h/\Re f_k$.  Let
 $p_k:C_-^\circ\to C_-^\circ/\Re f_k$ be the projection.  

As above, let us call $F$ {\em strict} if

1) the natural map
$p_k^{-1}Rp_{k*}\Con_k\to \Con_k$ is an isomorphism in $D(C_-^\circ)$
for all $k$;

2) $F$ is microsupported on $C_-^\circ \times C_+\subset
C_-^\circ \times \h^*$.

 Denote the full subcategory
of $\DBSh_{C_-^\circ}$ consisting of all strict $B$-sheaves on 
$C_-^\circ$ by $\DBSh_{C_-^\circ}^\strict$.

Let $\lambda\in \h$  and condider a shifted open set
$C_-^\circ+\lambda\subset \h$. We then have a notion of a 
$B$-sheaf and of a strict $B$-sheaf on $C_-^\circ+\lambda$ via 
an identification   $C_-^\circ+\lambda\cong C_-^\circ$ via the
shift $T_\lambda$. Hence we have categories 
$\DBSh_{C_-^\circ+\lambda}$; $\DBSh_{C_-^\circ+\lambda}^\strict$.

\subsubsection{} Let $\lambda\in \h$ and let  $j_\lambda:C_-^\circ+\lambda\to \h$ be an open embedding. We then see that the functor $j^{-1}_\lambda$ transforms strict
sheaves on $\h$ into strict sheaves on $C_-^\circ+\lambda$
\begin{Theorem}\label{teorhvat} The functor 
$$j^{-1}_\lambda:\DBSh_{\h}^\strict\to \DBSh_{C_-^\circ+\lambda}^\strict
$$
is an equivalence.
\end{Theorem}
\subsection{Proof of the theorem}\subsubsection{First reductions}
Without loss of generality one can put $\lambda=0$. We also
set $j:=j_0$.

Let $\pi:B\to \bfb$ be the projection. As the functor $R\pi_*$  is an equivalence, without loss
of generality, one can assume $B=\bfb$. 

Let $I\subset \{1,2,\ldots,N-1\}$.  Let $\cC(I,\h)
\subset D\bfb\Sh_\h$ be the full subcategory consisting of all
sheaves $F$ satisfying:

1) for all $i\in I$, we have: $\Con_i(F)=0$;

2) for all $i\notin I$ the natural map
$p_i^{-1}Rp_{i*}F\to F$ is an isomorphism.

It is clear that every object of $\cC(I,\h)$ is strict.

Let us define the category $\cC(I,C_-^\circ)$ in a similar way.

\begin{Lemma} Every  strict $\bfb$-sheaf on $C_-$ (resp. $\h)$ is quasi-isomorphic to a complex of objects from $\bigsqcup_I \cC(I,\h)$ 
(resp.
$\bigsqcup_I\cC(I,C_-^\circ)$).

\end{Lemma} 

\begin{proof}
We  will prove Lemma for strict sheaves on $C_-^\circ$. The proof for
$\h$ is similar.

Let us first consider a through map
$$
\pi_I:C_-^\circ\to \h\to \h/(\Re<f_j>_{j\notin I})
$$
let $C_I$ be the image of $\pi_I$.

We also have a through map
$$
\vs_I:\Re_{<0}<e_i>_{i\in I}\into \h\to \h/(\Re<f_j>_{j\notin I})
$$
\begin{sublemma} The map $\vs_I$ is an open embedding whose
image is the same as the image of $\pi_I$
\end{sublemma}
\begin{proof}(of sublemma) It is easy to see that the vectors
$f_j,j\notin I$; $e_i;i\in I$ form a basis of $\h$.  Therefore,
the vectors $e_i;i\in I$ (more precisely, their
images) form a basis of
$\h/(\Re<f_j>_{j\notin I})$.  Let $x\in C_-^\circ$. Let us expand
$$
x=\sum\limits_{i\in I} a_ie_i+
\sum\limits_{j\notin I} b_jf_j
$$

Then $p_I(x)=\sum_{i\in I}a_ie_i$.

We have: for all $j\notin I$:
$$
<x,f_j>=\sum_{k\notin I} b_k<f_j,f_k>\leq 0.
$$

Let $J:=\{1,2,\ldots,N-1\}\backslash I$ and let
us decompose $J$ into intervals as follows:
 $J=J_1\sqcup J_2\sqcup\cdots \sqcup J_s$
where each $J_t=[k_t;l_t]$ , $k_t\leq l_t<k_{t+1}-1$.
Set $b^t_k=b_k$ if $k\in J_t$; otherwise set $b^t_k=0$.
We then have $2b^t_k-b^t_{k-1}-b^t_{k+1}\leq 0$
for all $k\in J_t$.
Let $D^t_k:=b^t_k-b^t_{k-1}$. We then know
that $D^t_{k+1}\geq D^t_k$ if $k,k+1\in J_t$.
We then have
$b_k=D^t_{k_l}+\cdots+D^t_k$.  Assume $b_k>0$. Then 
$D_k^t>0$ (because $ D^t_{k_l}\leq
 D^t_{k_l+1}\leq \cdots\leq
D^t_k$). Hence, $0\leq D_k^t\leq D_{k+1}^t\leq \cdots$ and 
$0<b_k^t<b_{k+1}^t<\cdots <b_{l_t+1}^k=0$. Contradiction.
Thus, $b_k^t\leq 0$ for all $k$.  Therefore, for all $k$,
$b_k\leq 0$.

For every $i\in I$ we have
$$
0><x,f_i>=a_i+\sum_{j\notin I} b_j<f_i,f_j>.
$$
 Hence,
$$
<x,f_i>-\sum_{i\in I} b_i<f_i,f_j>=a_j.
$$
For $i\in I$, $j\notin I$, we have $i\neq j$ and
$<f_i;f_j>\leq 0$.
As  $b_i\leq 0$, we see that 
$0><x,f_j>\geq  a_j$. Hence, $\image(\pi_k)\subset
\image(\vs_k)$. Let us prove the inverse inclusion.
Let  $g:=\sum_{i\in I} a_ie_i -b\sum_{j\notin I} f_j$
We see that for $a_j>0$ and $0<b<<1$,we have
 $g\in C_-^\circ$ and
$\pi_k(g)=\sum_{i\in I} a_ie_i$.
\end{proof}

Let $\Gamma_I:=R_{<0}<e_i>_{i\in I}$. 

 We then have a 
surjection $P_I:C_-^\circ\to \Gamma_I$. It is easy
to see that $P_I$ is a trivial bundle whose fiber
is homeomorphic to $\Re^{N-1-|I|}$. 
Let $J\subset I$. It  follows that we have projections
$P_{IJ}:\Gamma_I\to \Gamma_J$ so that $P_J=P_{IJ}P_I$.

Let $F$ be a strict sheaf on $C_-^\circ$. Let $F_J:=
P_J^{-1}P_{J*}F$. It is easy to see that $F_J$ is a strict sheaf
on $C_-^\circ$. 
For $J\subset I$ we have a natural map $F_J\to F_I$.
Let $\subsets$ be the poset of all subsets of
 $\{1,2,\ldots,N-1\}$; view this subset as a category.
We then see that $I\mapsto F_I$ is a 
functor from $\subsets$ to the dg category of $B$-sheaves
 whose image lies in the full subscategory of strict $B$-sheaves. 

For a subset $I\subset \{1,2,\ldots,N-1\}$ consider
the standard complex
$$
K(I,F)=\bigoplus_{J:J\subset I}
F_J\otimes \Lambda^{\text{top}} (\gf[I\backslash J])[|I\backslash J|]
$$
with the standard differential. We then see
that 
1) $K(I,F)$ is a strict $B$-sheaf on $C_-^\circ$;

2) $p_I^{-1}Rp_{I*}K(I,F)\to K(I,F)$
is an isomorphism;

3) Let $J\subset I$ and $J\neq I$. Then
$Rp_{J*}K(I,F)=0$.  

2) and 3) imply that

4)  for any $J$ which intersects $I$, $Rp_{J*}K(I,F)=0$.

Let $k\in I$. Then we know that 
$p_k^{-1}Rp_{k*}\Con_k\to \Con_k$ is an isomorphism.
On the other hand, 4) implies that $Rp_{k*}\Con_k(K(I,F))=0$
Hence,

5) $\Con_kK(I,F)=0$ for all $k\notin I$.

Thus, $K(I,F)\in \cC(I,C_-^\circ)$, which proves Lemma 
for $C_-^\circ$. The proof for $\h$ is similar.
\end{proof}

\subsubsection{}
It is  clear that the functor $j^{-1}$ takes
$\cC(I,\h)$ to $\cC(I,C_-^\circ)$ for all $I$.

We will prove:
\begin{Lemma} Let $X$ be a $\bfb$-sheaf on $\h$ and let
$Y\in \cC(\h;I)$ for some $I\subset \{1,2,\ldots,N-1\}$.
Then the natural map
$$
R\hom_{D\bfb\Sh_\h}(X,Y)\to 
R\hom_{D\bfb\Sh_{C_-^\circ}}(j^{-1}X;j^{-1}Y)
$$
is an isomorphism
\end{Lemma}

\begin{proof} We see that $j_!j^{-1}X$ is a $\bfb$-sheaf on $\h$ and
that $$
R\hom_{D\bfb\Sh_{C_-^\circ}}(j^{-1}X;j^{-1}Y)=
R\hom_{D\bfb\Sh_{\h}}(j_!j^{-1}X;Y).
$$

We also have a natural map $j_!j^{-1}X\to X$ of $\bfb$-sheaves on $\h$.  Let $Z$ be the cone of this map. The statement of the Lemma
is equivalent to $R\hom_{D\bfb\Sh_\h}(Z,Y)=0$

For every $k\in I$ we have a structure map
\begin{equation}\label{er}Z\to T_{-2\pi e_k}^{-1}Z[D(-2\pi e_k)]
\end{equation}
\begin{sublemma}  The natural map
$$
R\hom_{D\bfb\Sh_\h}(T_{-2\pi e_k}^{-1}Z[D(-2\pi e_k)];Y)\to
 R\hom_{D\bfb\Sh_\h}(Z;Y)
$$
 is an isomorphism.
\end{sublemma}
\begin{proof}(of Sublemma) Let $W$ be the cone of the map (\ref{er}).We are to show that $R\hom_{D\bfb\Sh_\h}(W,Y)=0$. 
It follows that the structure map $W\to T_{-e_k}^{-1}W[D(-e_k)]$
is homotopy equivalent to 0. 

Choose an injective representative of $Y$ and consider
a $\bL_-$- graded complex  $H(l):=\hom(W;T_l^{-1}Y)$. This complex
is a $\bL_-$-graded $\bfb$-bimodule. We also have a $\bL_-$-graded $\bfb$-bimodule
structure on $\bfb$ . We then have
$$
R\hom_{D\bfb\Sh_\h}(W,Y)=R\hom_{\bfb-\text{bimod}}(\bfb;H).
$$
Let $R_k:=1\otimes 1_{-e_k}\in \bfb\otimes \bfb$;
$L_k=1_{-e_k}\otimes 1\in \bfb\otimes \bfb$.
We then see that the action of $R_k$ on $H$ is a quasi-isomorphism,
  whereas the action of $L_k$ is homotopy equivalent to 0. Hence
the action of $R_k-L_k$ on $H$ is a quasi-isomorphism. 
The action of $R_k-L_k$ on $\bfb$ is zero. Hence an induced
action of $R_k-L_k$ on $ R\hom_{\bfb-\text{bimod}}(\bfb;H)$
is simultaneously 0 and an isomorphism, meaning that 
$R\hom_{\bfb-\text{bimod}}(\bfb;H)=0$, whence the statement
\end{proof}

Let $g=-\sum\limits_{i\in I}2\pi  e_i$. Consider an inductive system 
of $\bfb$-sheaves on $\h$:
$$
Z\to T_g^{-1}[D(g)]Z\to T_{2g}^{-1}[D(2g)]Z\to \cdots\to 
T_{ng}^{-1}[D(ng)]Z\to \cdots
$$
and let $L(Z)$ be the derived direct limit of this system.
We have a natural map $Z\to L(Z)$. The previous Lemma easily 
implies that
the induced map
$$
R\hom_{D\bfb\Sh_\h}(L(Z);Y)\to R\hom_{D\bfb\Sh_\h}(Z;Y)
$$

is an isomorphism.

It also follows that the natural map

$$
R\hom(L(Z);Y)\to R\hom(Rp_{I!}L(Z);Rp_{I!}Y)
$$
is an isomorphism (because $Y$ is locally constant along fibers of
$p_I$).
Thus, the statement of our Lemma reduces to showing that
$$
Rp_{I!}L(Z)=0
$$

Let $x\in \h_I$ and show that  $Rp_{I!}L(Z)|_x=0$. We have
$$
Rp_{I!}L(Z)|_x=R\Gamma_c(p_I^{-1}x;L(Z)|_{p_I^{-1}x})
$$

Let $U_x\subset \h$; $U_x:=p_I^{-1}x$. 
By definition, we have
$$
R\Gamma_c(p_I^{-1}x;L(Z)|_{p_I^{-1}x})=\limdir_n
R\Gamma_c(U_{x+ng};Z|_{U_{x+ng}}[D(ng)])
$$
where the spaces $R\Gamma_c(U_{x+ng};Z|_{U_{x+ng}})$
form an inductive system by means of the structure maps
$Z\to T_g^{-1}Z[D(g)]$.
Next, we have
$$
R\Gamma_c(U_{x+ng};Z|_{U_{x+ng}})=\Cone[ R\Gamma_c(U_{x+ng}\cap C_-^\circ;X|_{U_{x+ng}})\to R\Gamma_c(U_{x+ng};X|_{U_{x+ng}})].
$$

We have maps
$$
U_x\to U_{x+g}\to U_{x+2g}\to \cdots \to U_{x+ng}\to \cdots
$$
induced by the shifts $T_g$.
Let $U:=\bigcup_n T_{ng}^{-1}(U_{x+ng}\cap C_-^\circ)$. 
We then have
$$
U_{x+ng}\cap C_-^\circ \subset T_{ng}U\subset U_{x+ng}.
$$
It follows that $U$ consists of all vectors 
$v=\sum_{i\in I} x_ie_i+\sum_{j\notin I} y_jf_j$, where
\begin{equation}\label{x:i}
x=\sum_{i\in I} x_ie_i
\end{equation}
and for all $l\notin I$,
\begin{equation}\label{y:i}
<\sum_{j\notin I} y_jf_j,f_l><0.
\end{equation}

It also follows that the natural maps
\begin{equation}\label{swed}
\limdir_n R\Gamma_c(U_{x+ng}\cap C_-^\circ;X|_{U_{x+ng}}][D(ng))\to 
\limdir_n R\Gamma_c(T_{ng}U;X|_{U_{x+ng}}[D(ng)])
\end{equation}
is an isomorphism.
Indeed, set $Z_n:= T_{ng}^{-1} Z|_{T_{ng}U}[D(ng)]$, $Z_n\in D(U)$.
The objects $Z_n$ form an inductive system.
Set $U_n:=T_{ng}^{-1}(U_{x+ng}\cap C_-^\circ)\subset U$.
We see that $U_0\subset U_1\subset U_2\subset \cdots$ and
$\bigcup_n U_n=U$
We then see that our inductive systems and their map can be rewritten as
$$
\limdir_n R\Gamma_c(U_n;Z_n)\to\limdir_n R\Gamma_c(U;Z_n)
$$

Let $K_n:=U\backslash U_n$. We then see that $\cap_n K_n=0$ and that the cone
of the above map is isomorphic
to
\begin{equation}\label{shishka}
\limdir_n R\Gamma_c(K_n;Z_n|_{K_n}).
\end{equation}
We see that for each $m$, the natural  map
$$
 R\Gamma_c(K_m;Z_m|_{K_m})\to
\limdir_n R\Gamma_c(K_n;Z_n|_{K_n}).
$$
factors as
$$
R\Gamma_c(K_m;Z_m|_{K_m})=\limdir_{n>m} R\Gamma_c(K_m\backslash K_n
;Z_m|_{K_m})\to  \limdir_n R\Gamma_c(K_n;Z_n|_{K_n})
$$
hence it is 0, which means that the space (\ref{shishka}) is 0
and the map (\ref{swed}) is an isomorphism.

Therefore, our original statement now reduces to showing that
\begin{equation}\label{cone:tng}
\Cone( R\Gamma_c(T_{ng}U;X|_{U_{x+ng}})\to
R\Gamma_c(U_{x+ng}:X|_{U_{x+ng}}))=0
\end{equation}
 for all $n>0$. 

let $A:=\Re<f_j>_{j\neq I}$.  We have an identification
$$
\alpha:A\to U_{x+ng},\quad a\mapsto \sum_{i\in I} x_ie_i+ng +a,$$
where $x_i$ are the same as in (\ref{x:i}).
Let $B\subset A$ be an open subset specified by
the condition (\ref{y:i}). It follows that
$\alpha(B)=T_{ng}U$. Let $Y\in D(A)$,
 $Y:=\alpha^{-1}X|_{U_{x+ng}}$.   We can rewrite (\ref{cone:tng}) as
$$
\Cone(R\Gamma_c(B,Y)\to R\Gamma_c(A,Y))
$$

Let us estimate the microsupport of $Y$.
We know that  $\mS(X)\subset \h\times C_+$.
Using Proposition (\ref{inverse:close}) one can show 
 that $Y$ is microsupported on the set
$
A\times \beta^*(C_+)$, where 
$\beta^*:\h^*\to A^*$ is dual to the embedding 
$\beta:A\to \h$; $\beta(f_j)=f_j$.  let $\ve^j\in A^*$
be the basis dual to $f^j$. One sees that $$\beta^*(C_+)=
\Re_\geq 0<\ve^j>_{j\notin I}.$$

Let $\gamma\subset  A$ be the dual cone to $\beta^*(C_+)$;
$\gamma= \Re_{>0}<f_j>_{j\notin I}$. 
One can check $B+\gamma=A$. As $\mS(Y)\subset A\times \beta^*(C_+)$, the Lemma follows.
\end{proof}
 
It now follows  that the functor 
$j^{-1}:D\bfb\Sh_\h\to D\bfb\Sh_{C_-^\circ}$ is conservative (the natural map
$R\hom(F,G)\to R\hom(j^{-1}F;j^{-1}G)$ is an isomorphism). We only
need to check the essential surjectivity of $j^{-1}$. It suffices
to check that for each $I\subset \{1,2,\ldots,N-1\}$,
the functor $j^{-1}:\cC(I,\h)\to \cC(I,C_-^\circ)$ is essentially
surjective. Let $F\in \cC(I,C_-^\circ)$ and consider a $\bfb$-sheaf
$G:=Rp_I^{!}Rp_{I!}L(j_!F):=Rp_I^{-1}Rp_{I!}L(j_!F)[N-1-|I|]$, where $L$ is the same as in the proof
of Lemma. 
 One easily checks  that $j^{-1}G\cong F$. This completes the proof
of the theorem.
\subsubsection{} Let us check that the $\bfb$- sheaf $\bfS$ is 
strict. Indeed, it follows that the structure map
$$
b_{-2\pi e_k}:\bfS\to T^{-1}_{-2\pi e_k}\bfS
$$
is induced by the correponding map
$$
b_{-2\pi e_k}^\sp:\sp\to T_{-2\pi e_k}^{-1}\sp=
\sp|_{G\times -2\pi e_k}*_G \sp
$$
which is in turn induced by the map
$$
\beta_{-e_k}:\gf_e\to \sp|_{G\times -e_k}
$$
as in Proposition \ref{efka}. let $B_k:=\cone b_k$.
We then get 
$$
\Cone b_{-2\pi e_k}^\sp=B_k*\sp.
$$
 
According to Proposition \ref{efka},
$\mS(B_k)\subset \{(g,\omega)|:<\|\omega\|,f_k>=0\}$
Standard computation shows that the sheaf 
$B_k*\sp$ is microsupported on the set
$$
\{(g,A,\omega,\eta)|(\eta,f_k)=0\}
$$
meaning that $\Cone b_{-2\pi e_k}^\sp=B_k*\sp$ is constant along the fibers of  the projection
$G\times \h\to G\times (\h/f_k)$. Hence, $\Cone b_{-2\pi e_k}=
i^{-1}b_{-2\pi e_k}^\sp$ is constant along the fibers of the projection
$$
\Zentrum\times \h\to\Zentrum\times \h/f_k
$$

It then follows that the sheaf $j_{C_-^\circ}^{-1}\bfS$ is a strict
$\bfb$-sheaf on $C_-^\circ$.  We know (see (\ref{hvatoiot})
that $j_{C_-^\circ}^{-1}S\cong \bigoplus_{I\subset \{1,2,\ldots,N-1\}}
\bS_I|_{C_-^\circ}$. 
It then easily follows that each $\bS_I$ is a strict  $\bfb$-sheaf
on $C_-^\circ$. Indeed, $\Cone b_k^{\bfS}=\bigoplus_I
 \Cone b_k^{\bS_I}$.
Let $C:=\Cone b_k^{\bfS}$ and $C_I:=\Cone b_k^{\bfS'_I}$.
Let $p_k:\Zentrum\times C_-^\circ \to \Zentrum\times C_-^\circ/f_k$.
One then sees that the natural  map
$$
p_k^{-1}Rp_{k*}C\to C
$$
is isomorphic to the  direct sum of  natural maps
$$
p_k^{-1}p_{k*}C_I\to C_I
$$

As the map $p_k^{-1}Rp_{k*}C\to C$ is an isomorphism, so is each of
its direct summands, i.e. all maps $
p_k^{-1}p_{k*}C_I\to C_I
$
are isomorphisms meaning that all sheaves $\bfS'_I$ are strict.

{\bf Remark} One can also prove that the sheaves $\bS_I$ are 
strict directly from the definition (\ref{defbfSI}).

According to Theorem \ref{teorhvat},   there exist strict $
\bfb$-sheaves on $\Zentrum\times \h$, to be denoted by $\bfS_I$
such that $i_{C_-^\circ}^!\bfS_I\cong \bS_I$ and
the sheaves $\bfS_I$ are unique up-to a unique isomorphism.
Same theorem implies that we should have an isomorphism
$$
\bfS\cong \bigoplus_I \bfS_I.
$$

\section{Identifying the sheaf $\bfS$} \label{lastbfs} One can check that
the $\bfb$- sheaf $\cX$ on $\Zentrum\times C_-^\circ$  as in  (\ref{defcx}) is strict. Indeed, this
follows from the fact that the 
$\bfb$- sheaf $\bfS_\emptyset=
G_\emptyset \otimes \cX$ is strict, or it can be checked directly.

It then follows that there exists a strict $\bfb$-sheaf $\cY$ on 
$\Zentrum\times \h$ such that $j_{C_-^\circ}^{-1}\cY=\cX$.
As $\bS_I\cong G_I[D(-2\pi e_I)]\otimes T_{-2\pi e_I*}\cX$, it then 
follows that
we have an isomorphism 
$\bfS_I\cong G_I[D(-2\pi e_I)]\otimes T_{-2\pi e_I*}\cY$ 
which is induced by the obvious isomorphism
$$
G_I[D(-2\pi e_I)]\otimes T_{-2\pi e_I*}\cX|_{(C_-^\circ-2\pi e_I)}
=G_I[D(-2\pi e_I)]\otimes T_{-2\pi e_I*}\cY|_{(C_-^\circ-2\pi e_I)}.
$$

Thus, we have an isomorphism
\begin{equation}\label{Y-bfS}
\bfS\cong \bigoplus_I   G_I[D(-2\pi e_I)]\otimes T_{-2\pi e_I*}\cY.
\end{equation}
It now remains to identify the $\bfb$-sheaf $\cY$.
\subsection{Identifying $\cY$}
\subsubsection{} For a subset $J\subset \{1,2,\ldots,N-1\}$ 
and $l\in L$ let $K(J,l)\subset e^l\times\h\subset \Zentrum\times 
h$ be defined as follows:
$$
K(J,l):=\{(e^{l},x)\in \Zentrum\times \h|\forall j\in J: 
<x-l,e_j>\geq 0\}.
$$
 Let $V(J,l):=\gf_{K(J,l)}[D(l)].
$
Let $\bL_J=\{l\in \bL|\forall i\notin J: <l,f_j>\leq 0\}$
Let
$\Psi^J:=\bigoplus\limits_{l\in \bL_J} V(J,l).$
Let us endow $\Psi^J$ with a $\bfb$-structure. 
Let $\lambda\in \bL_-$.  We have 
$$T_\lambda^{-1}V(J,l)=\gf_{T_\lambda^{-1}K(J,l)}[D(l)];$$
$$
T_\lambda^{-1}K(J,l')=\{(e^{l'},x)|\forall j\in J: e^\lambda e^{l'}=e^l;
<x+\lambda-l,e_j>\geq 0\}
$$
$$
=K(J,l-\lambda).
$$
Thus,
$$
T_\lambda^{-1}V(J,l)=\gf_{K(J,l-\lambda)}[D(l)]=V(J,l-\lambda)[D(\lambda)].
$$

It is clear that if $l\in \bL_J$,
then $l+\lambda\in \bL_J$. 
We then can define the map $b_\lambda:\Psi^J\otimes \bfb(\lambda)\to
T_\lambda^{-1}\Psi^J$ 
as a direct sum  of maps
$$
V(J,l)\otimes \bfb(\lambda)=V(J,l)[D(\lambda)]=
T_\lambda^{-1}V(J,l+\lambda).
$$

Let us now check that $\Psi^J$ are strict $\bfb$-sheaves.

Let $j\notin J$. Then it is clear that $\Psi^J$ is constant
along the fibers of the map $p_j:\Zentrum\times\h\to \Zentrum\times
\h/f_j$. Therefore so is the cone of $b_{-e_j}$. 
Let $j\in J$.  it is then easy to see that the map $b_{-e_j}$ is an
isomoprhism, whence the statement.

Let $J_1\subset J_2$. Construct a map of $\bfb$-sheaves
$$
I_{J_1J_2}:\Psi^{J_1}\to \Psi^{J_2}.
$$

It is defined as the  direct sum of the natural maps
$$
V(J_1,l)\to V(J_2,l)
$$
for all $l\in \bL_{J_1}\subset \bL_{J_2}$. These maps
come from the closed embeddings  $K(J_2,l)\subset K(J_1,l)$.

Let $\Subsets$ be the poset (hence the category) of all subsets
of $\{1,2,\ldots,N-1\}$. We then see that $\Psi$ is a functor
from $\Subsets$ to the category of $\bfb$-sheaves on
 $\Zentrum\times \h$.  We then construct the standard complex
$\Phi^\bullet$ such that
\begin{equation}\label{defphi}
\Phi^k:=\bigoplus\limits_{I,|I|=k} \Psi^I
\end{equation}
and the differential $d_k:\Phi^k\to \Phi^{k+1}$ is given by
\begin{equation}\label{defphi1}
d_k=\sum (-1)^{\vs(J_1,J_2)}I _{J_1J_2},
\end{equation}
where the sum is taken over all pairs $J_1\subset J_2$ such that
$|J_1|=k$ and $|J_2|=k+1$. The set $J_2\backslash J_1$ then consists
of a single element $e$ and $\vs(J_1J_2)$ is defined as the number
of  elements in $J_2$ which are less than $e$.

The constructed complex defines an object in 
$\DBSh_{\Zentrum\times \h}^\strict$, to be denoted by $\Phi$.

We will show $\Phi\cong\cY$. To this end it suffices to prove:

\begin{Lemma} We have $j_{C_-^\circ}^{-1}\Psi\cong \cX$.
\end{Lemma}
\begin{proof} 

We have a natural map $\iota:\cX\to j_{C_-^\circ}^{-1}\Phi^0=
j_{C_-^\circ}^{-1}\Psi^\emptyset$.
Indeed, 
$$
\cX=\bigoplus_{l\in \bL_-} \gf_{\cU_l}[D(l)]
$$
and $$
j_{C_-^\circ}^{-1}\Psi^\emptyset=\bigoplus_{l\in \bL_-}
 \gf_{e^{l}\times C_-^\circ}[D(l)].$$
  
The  map $\iota$ is defined as a direct sum of the obvious 
maps
$$
\gf_{\cU_l}[D(l)]\to \gf_{e^{l}\times C_-^\circ}[D(l)]
$$
coming from the open embeddings 
$\cU_l\subset e^{l}\times C_-^\circ$.

It is clear that $I_{\emptyset, J}\iota=0$ for all nonempty $J$.
Hence the map $\iota$ defines a map $\cX\to j_{C_-^\circ}^{-1}\Phi$. Let
us show that this map is an isomorphism.

For each $l\in L$ set 
$$
\Phi_l^n:= \bigoplus\limits_{J|l\in \bL_J;|J|=n} V(J,l)[D(l)].
$$
It is clear that for each $l$, $\Phi_l^\bullet\subset \Phi$ is
a subcomplex (in the category of comlexes of sheaves on $\Zentrum\times \h$) and
$$
\Phi =\bigoplus\limits_{l\in \bL} \Phi_l
$$

The map $\iota$ takes values in
 $\bigoplus_{l\in \bL_-}j_{C_-^\circ}^{-1}\Phi_l$
and splits into a direct sum of maps
$\iota_l:\gf_{\cU_l}\to j_{C_-^\circ}^{-1}\Phi_l$. 

We thus need to show that
1) complexes $j_{C_-^\circ}^{-1}\Phi_l$ are acyclic for all 
$l\notin \bL_-$;

2) the maps $\iota_l$ are quasi-isomorphisms.

Let us first study the complexes $\Phi_l$. Let us identify
$\h=\Re^{N-1}$ by means of the basis $f_1,f_2,\ldots,f_{N-1}$.
Let $X_j:\Zentrum\times \h\to \Zentrum\times\Re$ be defined by 
$$
X_j(c,A)=(c,x_j(A)),
$$
where 
$A=\sum_j x_j(A)f_j$.
Let $l_i=<l,f_i>$. Let $J_l:=\{i|l_i>0\}$. It follows that
$l\in \bL_J$ iff $J\supset J_l$. 
We also have $$V(J,l)=T_{l*}
(\bigotimes_{j\in J}X_j^{-1}\gf_{e\times [0,\infty)}\otimes 
\bigotimes_{i\notin J} X_j^{-1}\gf_{e\times \Re})[D(l)],
$$
where $e\in \Zentrum$ is the unit.
Let $E$ be the following complex of sheaves on $\Zentrum\times\Re$:
$$
\gf_{e\times\Re}\to \gf_{e\times [0,\infty)}.
$$
We then have an isomorphism of complexes
$$
\Phi_l =(T_{l*}\bigotimes\limits_{j\in J_l} X_j^{-1}\gf_{e\times[0,\infty)}\otimes
\bigotimes \limits_{i\notin J_l} X_i^{-1} E)[D(l)+|J_l|].
$$

We have a quasi-isomorhism $\gf_{e\times (-\infty,0)}\to E$ 
which induces
a quasi-isomorphism
$$
\Phi_l\cong (T_{l*} \bigotimes\limits_{j\in J_l} X_j^{-1}
\gf_{e\times[0,\infty)}\otimes
\bigotimes \limits_{i\notin J_l} X_i^{-1}\gf_{e\times (-\infty,0)})[D(l)+|J_l|]
$$
$$
=
T_{l*}\gf_{W_J}[D(l)+|J_l|],
$$
where $$W_J=\{(e,A)\in\Zentrum\times \h| j\in J\follows x_j(A)\geq 0;i\not\in J\follows x_i(A)<0\}$$

Let us now prove 1)  It follows that $\Phi_l$ is supported on the 
set $\overline{T_l(W_{J_l})}=\overline{W_{J_l}}+(e^{l},l)$. It suffices to prove
that $\overline{T_l(W_{J_l})}\cap \Zentrum\times C_-^\circ=0$. Suppose
$z'\in \overline{T_l(W_{J_l})}\cap\Zentrum\times C_-^\circ$. 
Let $z'=(e^l,z)$, $z\in \h$.

Let  $z=A+l$, $(e^{l},A)\in W_{J_l}$.
Let $A_j=(A,f_j)$ and $l_j=(l,f_j)$. We also set 
$A_0=A_N=l_0=l_N=0.$ Set $x_j:=x_j(A)$.
 We then know that
$l_j>0$ for all $j\in J_l$; $l_j\leq 0$ otherwise.
We also have 
$$A_j=<A,f_j>=<A,2e_j-e_{j-1}-e_{j+1}>
=2x_j-x_{j-1}-x_{j+1}.$$

As $A+l\in C_-^\circ$, we have $A_j+l_j<0$. Therefore, 
if $j\in J$, then $A_j<0$, thus $2x_j-x_{j-1}-x_{j+1}<0$.
We also know that if $j\in J$, then $x_j\geq 0$.

If $j\notin J$, then we know that $x_j\leq 0$.
For $j\in J$ let $j_1<j$ be the largest number 
such that   $j_1\notin J$, if it does not exist, set $j_1=0$.
Similarly, let $j_2>j$ be the smallest number such that $j_2\notin J$, if it does not exist set $j_2=N$.

We then have $x_{j_1}\leq 0$; $x_{j_2}\leq 0$; for all $j$ such that
$j_1<j<j_2$; $2x_j-x_{j-1}-x_{j+1}<0$, hence $x_j-x_{j-1}<x_{j+1}-x_j$,
 and $x_j\geq 0$. 

Therefore, we have
$$
0\leq x_{j_1+1}-x_{j_1}<x_{j_1+2}-x_{j_1+1}<\cdots<x_{j_2}-x_{j_2-1}
\leq 0.
$$
Observe that $j_2-j_1\geq 2$, therefore, we get $0<0$, which is a 
contradiction. Thus, indeed, $j_{C_-^\circ}^{-1}\Phi_l\cong 0$
for all $l\notin \bL_-$.

2) If $l\in \bL_-$, then $J_l=\emptyset$ and we have  a quasi-isomorphism
$ \gf_{(e^l,x)|x<<l}[D(l)]\to \Phi_l$. Therefore we have an induced quasi-isomorphism
 $\gf_{\cU_l}[D(l)]\to j_{C_-^\circ}^{-1}\Phi_l$.
One can easily check that this map is isomorphic to $\iota$, whence
the statement.
\end{proof}

From now on we set $\cY=\Phi$.

Let us summarize our results:

\begin{Theorem} Let $\cY=\Phi$, where $\Phi$ is as in (\ref{defphi}),(\ref{defphi1}).  Then we have an isomoprhism (\ref{Y-bfS})
\end{Theorem}
This theorem is equivalent to Theorem \ref{mainbfs}.

\section{Appendix 1:$\SU(N)$ and its Lie algebra: notations and a couple of Lemmas}\label{gnot}
Let us  introduce notation we will use when working with
$G=\SU(N)$. Let $\g$ be the Lie algebra of $G$; it is naturally
identified with the space of all skew-hermitian traceless $N\times
N$ matrices.  Let $\h\subset \g$ be the Cartan subalgebra of $\g$
consisting of all diagonal matrices from $\g$. 
The abelian Lie algebra $\h$ consists of all matrices of the form
$i\diag(\lambda_1,\lambda_2,\ldots,\lambda_N)$, where 
$\lambda_i\in \Re$ and $\sum_i \lambda_i=0$.

Let $C_+\subset \h$ be the positive Weyl chamber consisting
of all matrices $i\diag(\lambda_1,\lambda_2,\ldots,\lambda_N)$
with $\lambda_1\geq \lambda_2\geq\cdots \geq \lambda_N$.
For every $X\in \g$ there exists a unique element $\|X\|\in C_+$
such that $X$ is conjugate with $\|X\|$.

We have an invariant positive definite inner product 
$<,>$ on $\g$ such that $<X,Y>=-\Tr(XY)$. By means of this product
we identify $\g=\g^*$, $\h=\h^*$.

We will use the basis of roots in $\h^*$ which consists
of vectors $f_1,f_2,\ldots,f_{N-1}$, where
$$
f_k(i\diag(\lambda_1,\lambda_2,\ldots,\lambda_N))=\lambda_k-\lambda_{k+1}.
$$

Via identification $\h=\h^*$, the vector $f_k\in \h^*$ corresponds
to a vector in $\h$ denoted by the same symbol, and we have
$$
f_k=i\diag(0,0,\ldots,0,1,-1,0,\ldots,0)
$$
where $1$ is at the $k$-th position.

We also have the dual basis of coroots $e_1,e_2,\ldots,e_N$
determined by \newline
$<f_k,e_l>=\delta_{kl}$.
One has 
\begin{equation}\label{ebasis}
e_k=i\diag((N-k)/N,(N-k)/N,\ldots,(N-k)/N,-k/N,-k/N,\ldots,-k/N)
\end{equation}
 where there are total  $k$  entries equal to $(N-k)/N$.
One can check that $f_k=2e_k-e_{k-1}-e_{k+1}$ for $k=1,2,\ldots,
 N-1$ and we assume $e_0=e_{N}=0$.

One can rewrite $e_k=i\diag(1,1,\ldots,1,0,0,\ldots,0)-
ik/N\diag(1,1,1,\ldots,1)$, where it is assumed that we have $k$ 
entries of 1 in  $\diag(1,1,\ldots,1,0,0,\ldots,0)$.
In particular, we have
$$
<e_k,i\diag(\lambda_1,\lambda_2,\ldots,\lambda_n)>=\lambda_1+\lambda_2+\cdots+\lambda_k.
$$

One also sees that $C_+$ consists of all $X\in \h$ such that
$<X,f_k>\geq 0$. Therefore, $$
C_+=\{\sum\limits_{k=1}^N L_ke_k| L_k\geq 0\}.
$$

We have  a partial order on $\h$: $X\geq Y$ means $<X-Y,e_k>\geq 0$
for all $k$.

We also write $X>>Y$ if $<X-Y,e_k>>0$ for all $k$.

Let $\omega\in \g$. The matrix $-i\omega$ is hermitian and 
let $\lambda_1(\omega)>\lambda_2(\omega)>\cdots>\lambda_r(\omega)$
be eigenvalues of $-i\omega$.  Let $V^k(\omega)$ be the eigenspace
of $-i\omega$ of eigenvalue $\lambda_k$. Let $$
V_k(\omega)=
V^1(\omega)\oplus V^2(\omega)\oplus\cdots \oplus V^k(\omega).$$
We then get a partial flag
\begin{equation}\label{flagomega}
0\subset V_1(\omega)\subset \cdots\subset V_r(\omega)=\Co^N.
\end{equation}

Let $d_k(\omega):=\dim V_k(\omega)$.

\subsubsection{} 
In the future, we will need 

\begin{Lemma}\label{ineq} Let $X,\omega\in \g$. 
Let $\|X\|=i\diag(A_1,A_2,\ldots,A_N)\in C_+$ and let $$
0\subset V_1(\omega)\subset \cdots\subset V_r(\omega)=\Co^N
$$ be the flag as in (\ref{flagomega}).

Then
$$<\omega,X>\leq (\|\omega\|,\|X\|).
$$
 The equality takes place iff

a) $[X,\omega]=0$ (hence $XV_k(\omega)\subset V_k(\omega)$
for all $k$,  and 

b)$ \Tr X|_{V_k}=i(A_1+A_2+\cdots A_{d_k(\omega)})=i<e_{d_k(\omega)};\|X\|>$.

\end{Lemma} 
\begin{proof}
Let $\mu_k=\lambda_k(\omega)-\lambda_{k+1}(\omega)$; $k<r$.
Let us also set $\mu_r=\lambda_r(\omega)$. We then  have
$$
\omega=i\sum_{k=1}^r \mu_k \pr_{V_k(\omega)},
$$
where $\pr$ denotes the orthogonal projector;
$$
<\omega,X>=\sum\limits_{k=1}^{r-1} \mu_k \Tr(-iX\pr_{V_k(\omega)}).
$$

We know that $\Tr(-iX\pr_{V_k(\omega)})\leq 
A_1+A_2+\cdots +A_{d_k(\omega)}$ (this is a particular case of
the general factP: given any hermitian matrix $Y$ on $\Co^N$
(in our case $-iX$) and a vector subspace $V\subset \Co^N$
of dimension $n$ ( in our case $V_k(\omega)$), the maximal value
of $\Tr(Y\pr_V)$ equals the sum of top $n$ eigenvalues of $Y$).

Hence 
$$
<\omega,X>\leq \sum\limits_{k=1}^{r-1}\mu_k(A_1+\cdots+A_{d_k(\omega)})=\sum\limits_{j=1}^r A_j \sum\limits_{k|j\leq d_k(\omega)}\mu_k
$$
$$
=\sum\limits_{j=1}^r A_j\lambda_j(\omega)=<\|\omega\|,\|X\|>.
$$

The equality is only possible if for all $k$
$\Tr(-iX\pr_{V_k(\omega)})=A_1+\cdots A_{d_k(\omega)}$.
As $A_1,\ldots,A_{d_k(\omega)}$ are top $d_k(\omega)$ eigenvalues
of $-iX$, the equality occurs iff $V_k(\omega)$ is the span 
of eigenvectors of $-iX$ with eigenvalues $A_1,\ldots,A_{d_k(\omega)}$, which implies the statement b) of Lemma.
\end{proof}
\subsubsection{}\begin{Lemma}\label{ineqtriangle} Let $X,Y\in \g$. We have
$\|X+Y\|\leq \|X\|+\|Y\|$.
\end{Lemma}
\begin{proof} We need to show that for every $k$,
$$
<\|X+Y\|,e_k>\leq <\|X\|,e_k>+<\|Y\|,e_k>.
$$

For a Hermitian operator $A$ on  a finite-dimensional Hermitian vector
space $V$ we set $\bn(A):=\max_{|v|=1} <Av,v>$, where $<,>$ is the hermitian inner product on $V$.  We see that 
\begin{equation}\label{treugN}
\bn(A+B)\leq \bn(A)+\bn(B)
\end{equation}
and that $\bn(A)$ equals the maximal eigenvalue of $A$.

Let $\ve_k$ be the standard representation of 
$\g$ on $\Lambda^k\Co^N$.  Let $X\in \g$ and let $\lambda_1\geq \lambda_2\geq\cdots\geq  \lambda_N$ be the spectrum of a Hermitian matrix
$-iX$. This means that $\|X\|=i\diag(\lambda_1,\lambda_2,\ldots,\lambda_N)$.

 Eigenvalues of  $-i\ve_k(X)$ are of the form
$\lambda_{i_1}+\lambda_{i_2}+\cdots+\lambda_{i_k}$ where
$i_1<i_2<\ldots<i_k$. Therefore, the maximal eigenvalue
of $-i\ve_k(X)$ is $\lambda_1+\lambda_2+\ldots+\lambda_k$,
i.e.
$$
\bn(-i\ve_k(X))=<\|X\|,e_k>.
$$
As follows from (\ref{treugN}),
$$
\bn(-i\ve_k(X+Y))\leq \bn(-i\ve_k(X))+\bn(-i\ve_k(Y)),
$$
hence
$$
<\|X+Y\|,e_k>\leq <\|X\|,e_k>+<\|Y\|,e_k>,
$$
as was required.
\end{proof}
\subsubsection{}  Let $[a,b]\subset \Re$, $a\leq b$,
be a segment. Let $g\in \SU(N)$. Write $g\sim [a,b]$ if every 
eigenvalue of $g$ is of the form $e^{i\phi}$, where $\phi\in[a,b]$. 

\begin{Lemma}\label{eigeninterval} Let $g_k\sim [a_k,b_k]$, $k=1,2$. Then
$g_1g_2\sim [a_1+a_2,b_1+b_2]$.
\end{Lemma}
\begin{proof} If $b_1+b_2-(a_1+a_2)\geq 2\pi$, there is nothing
to prove, because $x\sim [a_1+a_2,b_1+b_2]$
for any element $x\in \SU(N)$.   Let now $b_1+b_2-(a_1+a_2)<2\pi$.
Let $c_k=(a_k+b_k)/2$ and $d_k=(b_k-a_k)/2$.  We have $d_1+d_2<\pi$,
hence $d_k<\pi$, $k=1,2$.

Let 
$h_k =e^{-ic_k}g_k$. We have $h_k\sim [-d_k,d_k]$.
Let $S\subset \Co^N$ be the unit sphere.  Let $\rho$ be the standard
metric on $S$; $\rho(v,w)= \arccos \text{Re}<v,w>$;
 $\rho(v,w)\in [0,\pi]$.
 For $g\in\SU(N)$, set 
$$\N(g):=\max_{v\in S}\rho(gv,v).
$$
It follows $\N(g_1g_2)\leq \N(g_1)+\N(g_2)$ for all
 $g_1,g_2\in \SU(N)$.

Let us estimate $\N(h_k)$.  Let $e_1,e_2,\ldots,e_N$ be an
 eigenbasis of $h_k$. We have $h_k(e_s)=e^{i\alpha_{ks}}e_s$, where
$\alpha_{ks}\in [-d_k,d_k]$. Let $v=\sum_s  v_se_s$, $v\in S$,
so that $1=\sum_s |v_s|^2$. We have
 $$h_kv=\sum_s v_se^{i\alpha_{ks}}e_s;
$$
$$
<h_kv,v>=\sum_s |v_s|^2e^{i\alpha_{ks}};
$$
$$
\Real<h_kv,v>=\sum_s |v_s|^2\cos \alpha_{ks}.
$$
As $\alpha_{ks}\in [-d_k,d_k]$ and $0\leq d_k<\pi$, 
we have $\cos \alpha_{ks}\geq \cos d_k$. Therefore,
$$
\Real<h_kv,v>\geq \sum_s |v_s|^2 \cos d_k=\cos d_k.
$$
Therefore, 
$$
\N(h_k)=\rho(h_kv,v)=\arccos \Real<h_kv,v>\leq d_k.
$$

Therefore, $\N(h_1h_2)\leq \N(h_1)+\N(h_2)\leq d_1+d_2$. It then follows
that $h_1h_2\sim [-d_1-d_2;d_1+d_2]$. Indeed, assuming 
  the contrary, we have an eigenvalue $e^{i\phi}$ of $h_1h_2$,
where $d_1+d_2<|\phi|\leq \pi$. let $h_1h_2v=e^{i\phi}v$, $|v|=1$.
We then have $\rho(h_1h_2v,v)=|\phi|>d_1+d_2$, which is a contradiction. 

Finally, we have $g_1g_2=e^{c_1+c_2}h_1h_2$, which implies that
$$g_1g_2\sim [c_1+c_2-d_1-d_2;c_1+c_2+d_1+d_2]=[a_1+a_2;b_1+b_2].
$$
\end{proof}

\subsubsection{}  Fix $b\in C_+^\circ$; $b< e_1/(100N)$.
Here and below $\circ$ means the interior.
\begin{Lemma}\label{Klyachko} Let $X,Y\in \g$ and 
$\|X\|,\|Y\|\leq b$. Then $e^Xe^Y=e^Z$, where 
$\|Z\|\leq \|X\|+\|Y\|$.
\end{Lemma}
\begin{proof} We have $$
e_1=((N-1)/N,-1/N,-1/N,\ldots,-1/N)
=(1,0,0,\ldots,0)-1/N(1,1,\ldots,1).
$$
Let $b=i\diag(b_1,b_2,\ldots,b_N)$. Since $b\in C_+^\circ$, we have
$b_1>b_2>\cdots>b_N$. We have $<b,e_k><<e_1/(100N),e_k>$ for all $k$.
In particular, $b_1=<b,e_1>\leq (N-1/N)\cdot (1/100N)<1/(100N)$;
 $b_1+b_2+\cdots+b_{N-1}=
(b,e_{N-1})\leq <e_1,e_{N-1}>/(100N)=1/(100N^2)<1/(100N)$.
 As $\sum_k b_k=0$,
we have $b_N>-1/(100N)$. Thus, $\forall k, |b_k|\leq 1/(100N)$.
 Let $\|X\|=i\diag(X_1,X_2,\ldots,X_N)$. As $\|X\|\leq b$, $|X_k|\leq 1/(100N)$.

 Therefore,
one has $e^X\sim [-1/(100N);1/(100N)]$. Analogously, 
$e^Y\sim [-1/(100N);1/(100N)]$. Lemma \ref{eigeninterval} implies
that 
$$e^Xe^Y=[-2/(100N);2/(100N)]=[-1/(50N);1/(50N)].
$$
Let $u_1,u_2,\ldots,u_N$ be the eigenbasis for $e^Xe^Y$.
It then follows that $e^Xe^Y=e^{i\phi_s}u_s$, where
$|\phi_s|\leq 1/(50N)$.  We have
 $1=\det(e^Xe^Y)=e^{i\sum_s \phi_s}.$ Therefore 
$\sum_s \phi_s =2\pi n$, $n\in \bfZ$. However,
$|\sum _s\phi_s|\leq 1/50<2\pi$. Hence, $n=0$ and
$\sum_s \phi_s=0$. Let $Z$ be  a skew-hermitian matrix defined by
$Zu_s=i\phi_su_s$. As $\sum_s\phi_s=0$, $Z\in \su(N)=\g$.
We also have $e^Xe^Y=e^Z$. Let us prove that 
$\|Z\|\leq \|X\|+\|Y\|$. 

Let $\Lambda_k$ (resp. $\ve_k$)
 be the standard representation of $G=\SU(N)$ (resp. $\g=\su(N)$)
on $\Lambda^k \Co^N$.
We then have
$$
e^{\ve_k(Z)}=e^{\ve_k(X)}e^{\ve_k(Y)}.
$$

Let $\|Z\|=i\diag(Z_1,Z_2,\ldots,Z_N)$. As 
was shown above, we have 
$
|Z_j|\leq 1/(50N). 
$

We then see that the spectrum of $\ve_k(Z)$ consists
of all numbers of the form
$$
i(Z_{j_1}+Z_{j_2}+\cdots Z_{j_k}),
$$
where $j_1<j_2<\cdots<j_k$.   We have
\begin{equation}\label{poltinnik}
|Z_{j_1}+Z_{j_2}+\cdots Z_{j_k}|\leq k/(50N)\leq 1/50.
\end{equation}

Let  $\|X\|=i\diag(X_1,X_2,\ldots,X_N)$. 
the spectrum of $e^{\lambda_k(X)}$ consists of numbers
of the form
$$
e^{i(X_{j_1}+X_{j_2}+\cdots+X_{j_k})},
$$
where $j_1<j_2<\ldots <j_k$.
Therefore
$$
e^{\lambda_k(X)}\sim 
 [X_{N-k+1}+X_{N-k+2}+\cdots +X_N;X_1+X_2+\cdots+X_k].
$$
We have $X_{N-k+1}+X_{N-k+2}+\cdots+X_N=-(X_1+X_2+\cdots+X_{N-k})=
-<X,e_{N-k}>$. Therefore,
$$
e^{\lambda_k(X)}\sim [-<\|X\|,e_{N-k}>;<\|X\|,e_k>].
$$
Analogously,
$$
e^{\lambda_k(Y)}\sim [-<\|Y\|,e_{N-k}>;<\|Y\|,e_k>].
$$
By Lemma \ref{eigeninterval}, we have
$$
e^{\lambda_k(Z)}=e^{\lambda_k(X)}e^{\lambda_k(Y)}
\sim [-<\|X\|+\|Y\|,e_{N-k}>;<\|X\|+\|Y\|,e_k>].
$$

As was shown above, we have $|X_j|,|Y_j|\leq 1/(100N)$ for
all $j$. Therefore, $|<\|X\|,e_{N-k}>|\leq (N-k)/(100N)<1/100$.
Analogously $$|<\|X\|,e_k>|,|<\|Y\|,e_k>|,<\|Y\|,e_{N-k}<1/100.
$$
Therefore
$$[-<\|X\|+\|Y\|,e_{N-k};<\|X\|+\|Y\|,e_k>]\subset [-1/50;
1/50].
$$ 
 According to (\ref{poltinnik}), all  eigenvalues
of $\lambda_k(Z)$ are of the form $it$, $|t|\leq 1/50$.
It now follows that all eigenvalues of $\lambda_k(Z)$
are of the form $it$, where
 $$t\in [-<\|X\|+\|Y\|,e_{N-k}>;<\|X\|+\|Y\|,e_k>].
$$
(otherwise, $e^{i\lambda_k(Z)}$ is not of the form $e^{it}$,
where $t\in [-<\|X\|+\|Y\|,e_{N-k}>;<\|X\|+\|Y\|,e_k>]$, as follows
from our estimates).
In particular, 
$$<\|Z\|,e_k>=Z_1+Z_2+\cdots+Z_k\in 
[-<\|X\|+\|Y\|,e_{N-k};<\|X\|+\|Y\|,e_k>],$$
whence 
$$
<\|Z\|,e_k>\leq <\|X\|+\|Y\|,e_k>.
$$
As $k$ is arbitrary, it follows that
$\|Z\|\leq \|X\|+\|Y\|$.
\end{proof}

For our future purposes we will need a stronger result.
\subsubsection{}\begin{Lemma}\label{spryamlenie} Let $X_1,X_2,\ldots X_n\in \g$;
 $\|X_i\|\leq b$. Let $V_1\subset V_2\subset \cdots V_r=\Co^N$
be a flag which is preserved by all $X_i$
 Then there exists an $X\in \g$ such that:

1) $e^{X_1}e^{X_2}\cdots e^{X_n}=e^X$;

2)$XV_k\subset V_k$ and $\Tr X|_{V_k}=\sum_k \Tr X_k|_{V_k}$
for all $k$;

3) $\|X\|\leq \sum_k \|X_i\|$
\end{Lemma}
\begin{proof}

1) Fix  an Ad- invariant  Hilbert norm $\bN$ on $\g$ 
(such an $\bN$
is unique up-to a scalar multiple). It follows 
 that $\bN(Z)\leq \bN(Y_1)+\bN(Y_2)$,
the equality being possible only if $Y_1$ and $Y_2$ are
proportional with non-negative coefficient( indeed:
$\bN(Z)$ is the length of the geodesic from the unit to
$e^Z$; $\bN(Y_1)+\bN(Y_2)$ is the length of a broken line, the equality is possible only if this broken line is actually
a geodesic).

2) Suppose $Y_1,Y_2\in \g$; $\|Y_1\|,\|Y_2\|\leq b$. According
to Lemma \ref{Klyachko} there exists a unique $Z:=Z(Y_1,Y_2)\in \g$; $\|Z\|\leq
\|Y_1\|+\|Y_2\|$ such that $e^Z=e^{Y_1}e^{Y_2}$.
We see that $e^ZV_k=V_k$, hence $(e^Z-\Id)V_k\subset V_k$. 
We can express $Z$ as a convergent series in powers of
$e^Z-\Id$, therefore, $ZV_k\subset V_k$. 
The equality $$\det e^Z|_{V_k}=\det e^{Y_1}|_{V_k}\det e^{Y_2}|_{V_k}$$
implies that $e^{\Tr Z|_{V_k}}=e^{\Tr(Y_1+Y_2)|_{V_k}}$. As
$\|Z\|\leq 2b$, this implies that 
$\Tr Z|_{V_k}=\Tr(Y_1+Y_2)|_{V_k}$.

3) Let $(Y_1,Y_2,\cdots Y_n)$ be a sequence of elements
$Y_i\in \g$; $|Y_i|\leq b$. Let 
$$
S_k(Y_1,Y_2,\ldots,Y_n):=(Y_1,\ldots,Y_{k-1},Z/2,Z/2,Y_{k+2},\ldots,Y_n), 
$$
where $k=1,2,\ldots,n-1$, $Z=Z(Y_k,Y_{k+1})$ is as explained above.

Let $\cX\subset \g^n$ be the set consisting of all sequences
of the form
$$
S _{k_1}S_{k_2}\cdots S_{k_R}(X_1,X_2,\ldots,X_n)
$$
for all $R$ and all $k_1,k_2,\ldots, k_R$. 
Let $\mu$ be the infimum of $\bN(Y_1)+\bN(Y_2)+\cdots \bN(Y_n)$
where $(Y_1,Y_2,\ldots,Y_n)\in \cX$.

Let $(Y_1(k),Y_2(k),\ldots,Y_n(k))\in \cX$, $k=1,2,\ldots,$
be  a sequence such that $\bN(Y_1(k))+\cdots \bN(Y_n(k))\to \mu$ as $k\to \infty$.  As $|Y_i(k)|\leq b$, one can choose a convergent subsequence, hence without loss of generality, one can assume that our sequence converges:
$$
\lim_{k\to \infty} Y_i(k)=Z_i.
$$

Then for all $(Y_1,Y_2,\ldots,Y_n)\in X$,
$$
\bN(Y_1)+\cdots \bN(Y_n)\geq \bN(Z_1)+\cdots \bN(Z_n).
$$

Let us show that there exists $Z\in \g$ such that 
each $Z_i$ is proportional to $Z$ with a non-negative coefficient. If not then there are $i<j$  such that

1) for all $i<k<j$, $Z_k=0$;

2) $Z_i$ and $Z_{j}$  are not
proportional to each other with a non-negative coefficient. 
Let 
$(Z'_1,\ldots,Z'_n)=T_{j-1}\cdots T_{i+1}T_i(Z_1,Z_2,\ldots,Z_n)$. 
We then have 
$\bN(Z'_1)+\cdots \bN(Z'_n)<\bN(Z_1)+\cdots \bN(Z_n)$.
Hence there exists a $k$ such that
$$
\bN(Y'_1)+\cdots +\bN(Y'_n)<\bN(Z_1)+\bN(Z_2)+\cdots+ \bN(Z_n),
$$
where 
$$
(Y'_1,Y'_2,\ldots,Y'_n)=T_{j-1}\cdots T_i(Y_1(k),Y_2(k),\ldots,Y_n(k)).
$$

But $(Y'_1,Y'_2,\ldots,Y'_n)\in \cX$, so we get
a contradiction.

Thus all $Z_i$ are proportional with non-negative coefficients.  Let us now set $X=Z_1+Z_2+\ldots Z_n$. Such an $X$ satisfies all the conditions 
\end{proof}

\section{Appendix 2: Results from \cite{KS} on functorial properties of microsupport}\label{appendix}

Despite the results to be quoted here are proved in \cite{KS} for
the bounded derived category, the same arguments 
work for the unbounded derived category so that we will omit the proofs.

\subsubsection{}\label{ks:locallyclosed} Let $S\subset X$ be a subset and $x\in X$. 
Following 
\cite{KS} Definition 5.3.6, one can define subsets
$N(S)\subset TX$ and $N^*(S)\subset T^*X$. As explained on
p 228, these subsets can be characterized as follows.
Let $x\in X$.  A non-zero
vector $\theta\in T_xX$ belongs to $N_x(S)$ if and only if,
in a local chart near $x$, there exists an open cone $\gamma$
containing $\theta$ and a neighborhood $U$ of $x$ such that
$U\cap((S\cap U)+\gamma)\subset S$.

One then defines $N^*_x(S)\subset T^*_xX$ as the dual cone
to $N_x(S)$. Finally one sets 
$N(S)=\cup_x N_xS$; $N^*(S)=\cup_x N^*_x(S)$. 
If $S\subset X$ is a closed submanifold, then $N^*(S)=T^*_S(X)$

Let now $x\in X$ and let $U$ be a neighborhood of $x$. 
Suppose that $S\cap U$ is defined by an inequality $f>0$ (or $f\\geq 0$),
where $f:U\to \Re$ is a smooth function and $d_xf\neq 0$. 
In this case $N^*_x(S)=\Re_{\geq 0}\cdot d_xf$.

For a subset $K\subset T^*X$ we set $K^a\subset T^*X$
to consist of all vectors $\omega$ such that $-\omega\in K$.

\begin{Proposition}(\cite{KS}, Proposition 5.3.8)
Let $X$ be a manifold, $\Omega$ an open subset and $Z$ a closed subsets. Then $\mS(\gf_\Omega)=N^*(\Omega)^a$; $\mS(\gf_Z)=
N^*(Z)$
\end{Proposition}

\subsubsection{}\label{ks:boxtimes}           
\begin{Proposition}(\cite{KS},Proposition 5.4.1) Let $F\in D(X)$ and
$G\in D(Y)$. Then 
$$
\mS(F\boxtimes G)\subset \mS(F)\times \mS(G).
$$
\end{Proposition} 

(Note that since our ground ring is a field $\gf$, the bifunctor $\boxtimes$ is
exact)
\subsubsection{}\label{ks:ihom}
Let $q_1:X\times Y\to X$; $q_2:X\times Y\to Y$ be the 
projections.  
\begin{Proposition}(\cite{KS}, Proposition 5.4.2)
Let $F\in D(X)$; $G\in D(Y)$. Then:
$$
\mS R\ihom(q_2^{-1}G;q_1^{-1}F)\subset
\mS(F)\times \mS(G)^a,
$$
where $\mS(G)^a\subset T^*Y$ consists of all points 
$\omega$ such that $-\omega\in \mS(G)$.
\end{Proposition}

\subsubsection{}\label{ks:direct}

Let $f:Y\to X$ be a morphism of  manifolds.
We have natural maps
$$
\tf:T^*X\times_X Y\to T^*Y
$$
and $f_\pi:T^*x\times_X Y\to T^*X$.

Thus, $T^*X\times_X Y$ is a correspondence between
$T^*X$ and $T^*Y$. Using this correspondence, one 
can transport sets from $T^*Y$ to $T^*X$ and
vice versa. Indeed, given a subset $A\subset T^*Y$
one  has a subset $f_\pi\tf^{-1}A\subset T^*X$. Given
a subset $B\subset T^*X$, one has a subset
$\tf f_\pi^{-1}(B)\subset T^*Y$.
\begin{Proposition}(\cite{KS},Proposition 5.4.4) Let $f:Y\to X$ be a morphism of manifolds,
$G\in D(Y)$, and assume $f$ is proper on $\text{Supp}(G)$. Then
$$
\mS(Rf_*G)\subset f_\pi(\tf^{-1}(\mS(G))).
$$
\end{Proposition}

Observe that under the hypothesis of this Proposition,
the natural map $Rf_! G\to Rf_* G$ is an isomorphism.
Therefore, the Proposition remains true upon replacement
of $Rf_*$ with $Rf_!$.
\subsubsection{}\label{ks:inverse}
 Let $f:Y\to X$ be a morphism of manifolds
and $A\subset T^*X$ a closed conic subset. We say that
$f$ is non-characteristic for $A$ if 
$f_\pi^{-1}A\cap T^*_Y X
\subset Y\times_X T^*_X X$.  Here
$T^*_Y X\subset T^*X\times X Y$ is the kernel of $\tf$ viewed
as a linear map of vector bundles.

\begin{Proposition}(\cite{KS},Proposition 5.4.13)
Let $F\in D(X)$ and assume $f:Y\to X$ is non-characteristic
for $\mS(F)$. Then 

$$\text{(i)}\quad \mS(f^{-1}F)\subset \tf(f_\pi^{-1}(\mS(F)));$$

(ii) The natural morphism $f^{-1}F\otimes \omega_{Y/X}\to 
f^!F$ is an isomorphism.
\end{Proposition}

\subsubsection{}\label{ks:otimeshom}
\begin{Proposition}(\cite{KS},Proposition
5.4.14) Let $F,G\in D(X)$. 

(i) Assume $\mS(F)\cap \mS(G)^a\subset T^*_X X$. Then
$\mS(F\otimes G)\subset \mS(F)+\mS(G)$;

(ii) Assume $\mS(F)\cap \mS(G)\subset T^*_X X$.
Then $\mS(R\ihom(G,F)\subset \mS(F)-\mS(G)$.
\end{Proposition}
\subsubsection{}\label{ks:openembedding} We need a notion of Witney sum of two 
conic closed subsets $A,B\subset T^*X$. We will
reproduce a definition in terms of local coordinates
from \cite{KS} Remark 6.2.8 (ii).

Let $(x)$ be a system of local coordinates on $X$, $(x,\xi)$
the asssociated coordinates on $T^*X$.  Then $x_o,\xi_o\in A
\hat{+} B$ iff there exist sequences $\{(x_n,\xi_n)\}$
in $A$ and $\{(y_n,\eta_n)\}$ in $B$ such that 
$x_n\to x_o$, $y_n\to y_o$, $\xi_n+\eta_n\to \xi_o$, and
$|x_n-y_n||\xi_n|\to 0$.

 \begin{Proposition} (\cite{KS},Theorem 6.3.1).
Let $\Omega$ be an open subset of $X$
and $j:\Omega\to X$ the embedding. Let $F\in D(X)$.
Then $\mS(Rj_*F)=\mS(F)\hat+ N^*(\Omega)$;
$\mS(j_!F)\subset  \mS(F)\hat+ N^*(\Omega)^a$.
\end{Proposition}

\subsubsection{} Let $f:Y\to X$ be a morphism 
of manifolds and $A\subset T^*X$ be a closed conic 
subset. One can define a closed conic subset $f^\#(A)\subset T^*M$ (\cite{KS}, Definition 6.2.3 (iv)). 
\begin{Proposition}\label{inverse:close}
(\cite{KS},Corollary 6.4.4)
 Let $F\in D(X)$. Then $\mS(f^{-1}F)\subset f^\#(\mS(F))$.
\end{Proposition}
In a particular case when $f$ is a closed embedding, 
the set $f^\#(A)$ 
 admits
an explicit description in local coordinates \cite{KS},
Remark 6.2.8, (i). That's the only case we will need.

Let $(x',x'')$ be a system of local coordinates on  $X$
such that $Y=\{(x',0)\}$.  Then $(x''_o;x''_o)\in f^\#(A)$
iff there exists a sequence of points
$(x'_n,x''_n,\xi'_n,\xi''_n)\in
A$ such that  $x'_n\to 0$; $x''_n\to x_o''$; $\xi''_n\to \xi''_o$, and $|x_n'||\xi'_n|\to 0$.


\begin{thebibliography}{xyz}
\bibitem{KS}{M. Kashiwara, P. Schapira,
Sheaves on Manifolds, Grundlehren der
mathematischen Wissenschaften 292, Springer-Verlag}
\bibitem{Dr}{V. Drinfeld,  DG quotients 
of DG categories, J. Algebra {\bf 272} (2004), no.
2, 643-691}
\bibitem{Pol}{M. Entov, L. Polterovich, Rigid subsets of
 symplectic manifolds, arXiv 0704.0105v1}

\bibitem{Cho}{C.-H. Cho, Holomorphic disks, spin structures and Floer cohomology
of Clifford torus, Int. Math. Res. Notes
{\bf 35} (2004), 1803-1843}
\bibitem{Bez}{R. Bezrukavnikov, M. Finkelberg, I. Mirkovich, Equivariant $(K-)$ homology of affine grassmanian
and Toda Lattice}
\bibitem{Kost}{B. Kostant, Flag manifold quantum cohomology, the Toda lattice, and the representation with highest weight $\rho$, Selecta Math. (N.S.) {\bf 2} (1996),
43-91}
\bibitem{NZ}{D. Nadler, E. Zaslow, Constructible sheaves and Fukaya category, 2006, To
appear in J. Amer. Math. Soc}
\bibitem{NT}{R. Nest, B. Tsygan, Remarks on modules over deformation quantization
algebras,math-ph:0411066}
\end{thebibliography}
\end{document}